\documentclass[12pt,centertags]{amsart}
\usepackage{amssymb}
\usepackage[all]{xy}
\usepackage{amsmath,amstext,amsthm,a4,amssymb,amscd}
\usepackage[mathscr]{eucal}
\usepackage{mathrsfs}
\usepackage{epsf}
\usepackage{tikz}
\usetikzlibrary{matrix,arrows,decorations.pathmorphing}
\usepackage[a4paper,top=3cm,bottom=3cm,
left=2.2cm,right=2.2cm]
{geometry}
\numberwithin{equation}{section}
\allowdisplaybreaks[3]

\usepackage{color}
\usepackage{soul}

\newcommand{\field}[1]{\mathbb{#1}}
\newcommand{\C}{\field{C}}
\newcommand{\N}{\field{N}}
\newcommand{\Q}{\field{Q}}
\newcommand{\R}{\field{R}}
\newcommand{\Z}{\field{Z}}

\def\cC{\mathscr{C}}

\def\mA{\mathcal{A}}

\def\mE{\mathcal{E}}
\def\mF{\mathcal{F}}

\def\mL{\mathcal{L}}
\def\mM{\mathcal{M}}
\def\mN{\mathcal{N}}

\def\mU{\mathcal{U}}
\def\mV{\mathcal{V}}

\newcommand\mS{\mathcal{S}}

\def\kg{\mathfrak{g}}

\def\Re{{\rm Re}}
\def\Im{{\rm Im}}

\def\la{\langle}
\def\ra{\rangle}

\DeclareMathOperator{\End}{End}

\DeclareMathOperator{\Ker}{Ker}

\DeclareMathOperator{\rank}{rk}
\DeclareMathOperator{\Id}{Id}

\DeclareMathOperator{\tr}{Tr}

\DeclareMathOperator{\ind}{Ind}

\DeclareMathOperator{\td}{Td}

\DeclareMathOperator{\ch}{ch}

\newtheorem{thm}{Theorem}[section]
\newtheorem{lemma}[thm]{Lemma}
\newtheorem{prop}[thm]{Proposition}
\newtheorem{cor}[thm]{Corollary}

\newtheorem{defthm}[thm]{Definition and Theorem}

\theoremstyle{definition}
\newtheorem{rem}[thm]{Remark}
\theoremstyle{definition}
\newtheorem{defn}[thm]{Definition}

\newtheorem{eg}[thm]{Example}
\newcommand{\be}{\begin{eqnarray}}
\newcommand{\ee}{\end{eqnarray}}
\newcommand{\ov}{\overline}
\newcommand{\wi}{\widetilde}
\newcommand{\var}{\varepsilon}
\numberwithin{equation}{section}
\numberwithin{thm}{section}

\newcommand{\comment}[1]{}

\begin{document}
	
\title[Localization formula for $\eta$-Invariants]
{Differential $K$-theory and localization formula for $\eta$-invariants}
	
\author{Bo LIU}
\address{School of Mathematical Sciences,
	Shanghai  Key Laboratory of PMMP,
	East China Normal University, 
	Shanghai,  
	P.R. China}
\email{bliu@math.ecnu.edu.cn}

\author{Xiaonan MA}

\address{Universit\'e de Paris, CNRS, 
	Institut de Math\'ematiques de Jussieu-Paris Rive Gauche, 
	F-75013 Paris, France,
\and School of Mathematical Sciences, University of Science and 
Technology of China, Herfei, P.R. China}  
\email{xiaonan.ma@imj-prg.fr}	

\date{\today}
	
\begin{abstract}
In this paper  
we obtain a localization formula 
in differential $K$-theory for $S^1$-actions. 
We establish a localization formula for 
equivariant $\eta$-invariants by combining this result with
our extension of Goette's result 
on the comparison of two types of equivariant $\eta$-invariants.
An important step in our approach is to 
 construct a pre-$\lambda$-ring structure
 in differential $K$-theory. 
\end{abstract}
	
\maketitle


\tableofcontents
\setcounter{section}{-1}
	
\section{Introduction}\label{s01}  
	
The famous Atiyah-Singer index theorem \cite{AS63} 
states that for an elliptic differential operator on a compact manifold, 
the analytical index (related to the dimension of the space of solutions) 
is equal to the topological index,  computed 
in terms of characteristic classes.
We can view the index as a primitive spectral invariant of 
	an elliptic operator, whereas global spectral invariants such as 
	the $\eta$-invariant of Atiyah-Patodi-Singer 
	and the analytic torsion of Ray-Singer
	as the secondary spectral invariants of an elliptic operator.
	In \cite[Proposition 2.10]{ASegal68}, Atiyah and Segal  
	established a localization formula for the equivariant index using 
	 topological $K$-theory, which 
	computes the equivariant index via the contribution of 
	the fixed point set of the group action. 
	Thus it is natural to ask whether the localization property holds for these
	secondary spectral invariants. 
Note that they are neither computable
	from local data, nor topological invariants as the index.

	Note that the Ray-Singer holomorphic analytic torsion  \cite{RS73} 
	(and its families version, Bismut-K\"ohler torsion form \cite{BK92}) 
	is the analytic counterpart of
	the direct image in Arakelov geometry \cite{Soule92}. 
	Bismut-Lebeau's  embedding formula \cite{BL91} for 
	the analytic torsion and Bismut's family extension \cite{Bismut1997}
	are the essential analytic ingredients of	
	the arithmetic Riemann-Roch-Grothendieck theorem 
	\cite{GRosslerS08,GS92}. 

K\"ohler-Roessler established in their proof of the equivariant arithmetic
	Riemann-Roch theorem, a Lefschetz type fixed point formula 
	\cite[Theorem 4.4]{KRo01} in the equivariant arithmetic $K$-theory.
In the arithmetic context, their result 
gives a relation between the equivariant holomorphic torsion 
and the contribution of the fixed point set corresponding to the
$n$-th roots of unity.
	In \cite[Lemma 2.3]{KRo02}, they discussed in detail this problem
	and made a conjecture
	for projective complex manifolds 
	\cite[Conjecture, p82]{KRo02}.
	K\"ohler-Roessler \cite{KRo02} did not use the comparison formula of
	Bismut-Goette \cite{BG00}, but  used instead their 
	equivariant arithmetic Riemann-Roch theorem.
	For more applications of  the equivariant arithmetic
	Riemann-Roch theorem, cf. Maillot-Roessler \cite{MaiRo04} 
	and later works.
	
Atiyah-Patodi-Singer \cite{APS75} developed an index theory 
	for the Dirac operator on compact manifolds with boundary. 
	Their index formula involves a contribution
	of the boundary, called $\eta$-invariant.
	Formally, it is equal to the number
	of positive eigenvalues of the Dirac operator minus the number of
	its negative eigenvalues.	
Cheeger-Simons \cite{CheegerS85} 
gave a formula in $\R/\Q$  for 
the $\eta$-invariant by using their differential characters,
cf. also the work of Zhang \cite{Z05}.

The $\eta$-invariant (and its families version,
Bismut-Cheeger $\eta$-form \cite{BC89})
is  the analytic counterpart 
of the direct image in differential $K$-theory
	(see e.g., \cite{Be09,Bunke2009,FreedLott10}).
	In particular, the embedding formula of 
	Bismut-Zhang \cite[Theorem 2.2]{BZ93}
	for the $\eta$-invariants plays an important role 
	in the proof of Freed-Lott's index theorem 
	in differential $K$-theory \cite[Theorem 7.35]{FreedLott10}.
	Various extensions of 
	Bismut-Zhang's embedding formula 
	have been recently established
	 by B. Liu \cite{Liu2017} and later work.
    
  In this paper we will establish 
a localization formula	in differential $K$-theory. 
Our result is formally similar to \cite[Theorem 4.4]{KRo01}, 
	but we employ here totally different arguments.
For $S^1$-actions we get a pointwise identification between
	the equivariant $\eta$-invariant and the fixed point set
	contribution to the $\eta$-invariant, 
	modulo a rational function with integral coefficients. 
	The definition of the fixed point set
	contribution is actually an important part 
	of the localization formula.
By combining this identification with our recent extension \cite{LM2018}
	of Goette's comparison formula for two kinds 
	of equivariant $\eta$-invariants \cite{Goette00} 
	we finally conclude our main result: the difference of the 
	equivariant $\eta$-invariant and its fixed point set
contribution, as a function on the complement of a finite subset of $S^1$,
 is the restriction of a rational function on $S^1$ with integral coefficients.
	It seems that our result is the first 
	geometric application of differential $K$-theory.
    
Let us recall first the Atiyah-Segal localization formula for the
 equivariant index. 
Let $Y$ be an $S^1$-equivariant compact Spin$^c$ manifold.
It induces an $S^1$-equivariant complex line bundle $L$ such that
$\omega_2(TY)=c_1(L)\,\mathrm{mod}\,(2)$, where $\omega_2$ is 
the second Stiefel-Whitney class and $c_1$ is the first Chern class 
\cite[Appendix D]{LaMi89}. Let $E$ be an $S^1$-equivariant 
complex vector bundle over $Y$.  Let $D^Y\otimes E$ be the spin$^c$
Dirac operator on $\mS(TY,L)\otimes E$, where $\mS(TY,L)$
is the spinor associated with this spin$^c$ structure 
(cf. \eqref{local208}). 

For any complex vector bundle $F$ over a manifold $X$,
we use the notation
\begin{align}\label{eq:0.1}
\mathrm{Sym}_t(F)=1+\sum_{k>0}\mathrm{Sym}^k(F)t^k,
\quad
\lambda_t(F)=1+\sum_{k>0}\Lambda^k(F)t^k
\end{align}
for the symmetric and exterior powers of $K^0(X)[[t]]$ respectively and 
denote by $\mathrm{Sym}(F):=\mathrm{Sym}_1(F)$.
Here in (\ref{eq:0.1}), 
$1$ is understood as the trivial complex line bundle
over $X$ in $K^0(X)$, the $K$-group of $X$.

Let $Y^{S^1}$ 
be the fixed point set of the  circle action on $Y$, then 
each connected component 
$Y^{S^1}_{\alpha}$, $\alpha\in \mathfrak{B}$, 
of $Y^{S^{1}}$, is a compact manifold. 
Let $N_{\alpha}$ be the normal bundle of $Y_{\alpha}^{S^1}$ in $Y$ 
and we can choose a complex structure on $N_{\alpha}$ through 
the circle action.
For any $\alpha\in \mathfrak{B}$, $Y_{\alpha}^{S^1}$ also has 
an equivariant spin$^c$ structure with associated equivariant line bundle 
$L_{\alpha}=L|_{Y^{S^1}_{\alpha}}\otimes(\det N_{\alpha})^{-1}$ 
(see \eqref{local210} for instance). 

Let $K_{S^1}^0(Y_{\alpha}^{S^1})_{I(g)}$ be the localization of 
the equivariant $K$-group $K_{S^1}^0(Y_{\alpha}^{S^1})$ 
at the prime ideal
$I(g)$, which consists of all characters of $S^1$ vanishing at $g$.
 
Assume temporarily that  $Y$ is even dimensional.
Then the spinor is naturally $\Z_2$-graded: 
$\mS(TY,L)=\mS^+(TY,L)\oplus \mS^-(TY,L)$.
Let $D_{\pm}^Y\otimes E$ be the restrictions of $D^Y\otimes E$ 
to the space $\cC^{\infty}(Y,\mS^{\pm}(TY,$ $L)\otimes E)$
of smooth sections of $\mS^{\pm}(TY,$ $L)\otimes E$ on $Y$.
Then the kernels $\Ker (D_{\pm}^Y\otimes E)$ of 
$D_{\pm}^Y\otimes E$ are finite dimensional 
$S^1$-complex vector spaces. Let 
\begin{align}\label{i17} 
\ind_g(D^Y\otimes E)=\tr|_{\Ker (D^Y_+\otimes E)}[g]
-\tr|_{\Ker (D^Y_-\otimes E)}[g]
\end{align}
be the equivariant index of $D^Y\otimes E$ 
corresponding to $g\in S^1$.

For $g\in S^1$ fixed, let $\chi$ be a character of $S^1$ 
such that $\chi(g)\neq 0$.
For $w=(F^+-F^-)/\chi\in K_{S^1}^0(Y_{\alpha}^{S^1})_{I(g)}$ we
define the equivariant index of $D^{Y_{\alpha}^{S^1}}\otimes w$ by
\begin{align}\label{i18} 
\ind_g(D^{Y_{\alpha}^{S^1}}\otimes w):=
\chi(g)^{-1}\left(\ind_g(D^{Y_{\alpha}^{S^1}}\otimes F^+)
-\ind_g(D^{Y_{\alpha}^{S^1}}\otimes F^-)\right).
\end{align} 
This does not depend on the choices of 
$F^+,F^-\in K_{S^1}^0(Y_{\alpha}^{S^1})$ and $\chi$. 
Here $D^{Y_{\alpha}^{S^1}}\otimes F^{\pm}$ are spin$^c$ 
Dirac operators on $Y_{\alpha}^{S^1}$ defined in a similar way 
as $D_{\pm}^Y\otimes E$. 

\begin{thm}\label{as01}\cite[Lemma 2.7 and Proposition 2.10]{ASegal68}
If $Y^{S^1}$ is the fixed point set of $g\in S^1$, 
then there exists an inverse 
$\lambda_{-1}(N_{\alpha}^*)^{-1}$
in $K_{S^1}^0(Y_{\alpha}^{S^1})_{I(g)}$. Moreover,
\begin{align}\label{i01}
\ind_g(D^Y\otimes E)=\sum_{\alpha}\ind_g\left(D^{Y_{\alpha}^{S^1}}
\otimes \lambda_{-1}(N_{\alpha}^*)^{-1}\otimes 
E|_{Y_{\alpha}^{S^1}}\right).
\end{align}
\end{thm}

For simplicity, we fix a complex structure on $N_{\alpha}$ such that
the weights of the $S^1$-action on $N_{\alpha}$ are all positive. Then 
by \cite[(1.15) and (1.17)]{LMZ03}, we can 
reformulate Theorem \ref{as01} 
in the following way: 
\begin{align}\label{i19}
\ind_g(D^Y\otimes E)=\sum_{\alpha}\ind_g\left(D^{Y_{\alpha}^{S^1}}
\otimes \mathrm{Sym}(N_{\alpha}^*)\otimes 
E|_{Y_{\alpha}^{S^1}}\right) \quad 
\text{as distributions on $S^1$}.
\end{align} 
Notice that 
$\ind(D^Y\otimes E)=\Ker (D^Y_+\otimes E)-\Ker (D^Y_-\otimes E)$
is a finite dimensional virtual representation of $S^1$, thus is
an element of the representation ring $R(S^1)$ of $S^1$.
For each $\alpha$,
\begin{align}\label{i20} 
\ind\left(D^{Y_{\alpha}^{S^1}}\otimes \mathrm{Sym}(N_{\alpha}^*)
\otimes E|_{Y_{\alpha}^{S^1}}\right)\in R[S^1]
\end{align}
is a formal representation of $S^1$, i.e., each 
component of weight $k$ of \eqref{i20}, denoted by 
$$\ind\left(D^{Y_{\alpha}^{S^1}}\otimes \mathrm{Sym}(N_{\alpha}^*)
\otimes E|_{Y_{\alpha}^{S^1}}\right)_k$$
is a finite dimensional virtual vector space.
As a consequence of \eqref{i19}, we have for any $|k|\gg 1$, 
\begin{align}\label{i21} 
\sum_{\alpha}\ind\left(D^{Y_{\alpha}^{S^1}}\otimes 
\mathrm{Sym}(N_{\alpha}^*)\otimes E|_{Y_{\alpha}^{S^1}}\right)_k=0.
\end{align}

From now on  
 we assume that $Y$ is odd dimensional.
Let $g^{TY}$ be an $S^{1}$-invariant Riemannian metric on $TY$, 
and $\nabla^{TY}$ be the Levi-Civita connection on $(Y, g^{TY})$.
Let  $h^L$ and $h^E$ be $S^1$-invariant metrics. Let
$\nabla^L$ and $\nabla^E$ be $S^1$-invariant Hermitian connections
on $(L,h^{L})$ and $(E, h^{E})$. Put
\begin{align}\label{i02} 
\underline{TY}=(TY,g^{TY},\nabla^{TY}),\quad 
\underline{L}=(L,h^L,\nabla^L),\quad \underline{E}=(E,h^E,\nabla^E).
\end{align}
 We call them equivariant geometric triples. 
 
 For $g\in S^1$  
let $\bar{\eta}_g(\underline{TY}, \underline{L}, \underline{E})$
be the associated equivariant APS reduced $\eta$-invariant  
(cf. Definition \ref{local128}). 

In the 
rest of this paper 
we  always assume  that
$Y^{S^1}\neq \emptyset$ except in Section \ref{s0405}.
 
 Let $\underline{L_{\alpha}}$, $\underline{N_{\alpha}^*}$,
 $\underline{\mathrm{Sym}(N_{\alpha}^*)}$  and
 $\underline{\lambda_{-1}(N_{\alpha}^*)}$ be the  induced 
geometric triples on $Y^{S^{1}}_{\alpha}$. 
In view of \eqref{i19}, it is natural to ask whether we can define 
$\bar{\eta}_g\left(\underline{TY_{\alpha}^{S^1}}, 
\underline{L_{\alpha}}, \underline{\mathrm{Sym}(N_{\alpha}^*)}
\otimes\underline{E}|_{Y_{\alpha}^{S^1}}\right)$ as 
a distribution on $S^1$ 
for each $\alpha$ and how to compute the difference 
\begin{align}\label{i22}
\bar{\eta}_g(\underline{TY}, \underline{L}, \underline{E})
-\sum_{\alpha}\bar{\eta}_g\left(\underline{TY_{\alpha}^{S^1}}, 
\underline{L_{\alpha}}, \underline{\mathrm{Sym}(N_{\alpha}^*)}
\otimes\underline{E}|_{Y_{\alpha}^{S^1}}\right)
\end{align}
as 
a distribution on $S^1$ by using geometric data on $Y$. 

In this paper 
we give a realization of
$\underline{\lambda_{-1}(N_{\alpha}^*)}^{-1}$ in the localization of 
equivariant differential $K$-theory, such that
\begin{align}\label{i03}  
	\sum_{\alpha}\bar{\eta}_g	\left(\underline{TY_{\alpha}^{S^1}}, 
	\underline{L_{\alpha}}, \underline{\lambda_{-1}(N_{\alpha}^*)}^{-1}
	\otimes\underline{E}|_{Y_{\alpha}^{S^1}}\right)
\end{align}
is well-defined, and then we identify it to
$\bar{\eta}_g(\underline{TY}, \underline{L}, \underline{E})$
up to a rational function on $S^{1}$ with integral coefficients. 
The 
remaining challenging problem is to compute precisely
this rational function on $S^{1}$ in a geometric way.

For $g\in S^1$ 
let $\widehat{K}_{g}^0(Y)$ be 
the $g$-equivariant differential $K$-group in Definition \ref{defn:2.12}, 
which is the Grothendieck group of 
equivalence classes $[\underline{E}, \phi]$ 
(see \eqref{eq:2.083} 
for the equivalence relation) of cycles $(\underline{E}, \phi)$, 
where $\underline{E}$ is an equivariant geometric triple and $\phi\in 
\Omega^{\mathrm{odd}}(Y^g, \C)/d\Omega^{\mathrm{even}}(Y^g, \C)$, 
the space of odd degree complex valued differential forms on the 
fixed point set $Y^g$ of $g$, modulo exact forms.
Let $\widehat{K}_{g}^0(Y)_{I(g)}$ be its localization at 
the prime ideal $I(g)$. Then as explained in (\ref{eq:2.086}),
an element of $\widehat{K}_{g}^0(Y)_{I(g)}$ can be written as
$\left([\underline{E}, \phi]-[\underline{E'}, \phi']\right)
/\chi$, where 
$\chi$ is a character of $S^1$ such that $\chi(g)\neq 0$.

With respect to the $S^1$-action, we have the decomposition of 
complex vector bundles $N_{\alpha}=\bigoplus_{v>0}N_{\alpha,v}$ 
such that $g\in S^1$ acts on $N_{\alpha,v}$ by multiplication by $g^v$. 

 Using the pre-$\lambda$-ring structure of the differential $K$-theory 
 constructed in Theorem \ref{thm:2.06}, 
 we obtain the differential $K$-theory 
 version of the first part of Theorem \ref{as01}:
\begin{thm}[See Theorem \ref{thm:2.14}]\label{i06} 
There exists a 
finite subset $A\subset S^1$ (cf. Proposition \ref{local32}), 
such that for $g\in S^1\backslash A$,
$\left[\underline{\lambda_{-1}(N_{\alpha}^*)}, 
0\right]$ is invertible in 
$\widehat{K}_{g}^0(Y_{\alpha}^{S^1})_{I(g)}$ and there exists 
$\mN_0>0$, which does not depend on $g\in S^1\backslash A$, 
such that for any $\mN\in \N, \mN>\mN_0$, we have
\begin{align}\label{i07} 
\left[\underline{\lambda_{-1}(N_{\alpha}^*)},0\right]^{-1}=
\left[\underline{\lambda_{-1}(N_{\alpha}^*)^{-1}_{\mN}},0\right]\in 
\widehat{K}_{g}^0(Y_{\alpha}^{S^1})_{I(g)}.
\end{align}
Here 
$\underline{\lambda_{-1}(N_{\alpha}^*)^{-1}_{\mN}}$
is defined by 
truncation up to degree $\mN>\mN_0$  
in the formal expansion of $\underline{\lambda_{-1}(N_{\alpha}^*)}^{-1}$ 
given by the $\gamma$-filtration (see the precise definition in 
\eqref{eq:2.061}, \eqref{eq:2.063}, \eqref{eq:2.064} 
and \eqref{eq:2.092}).
\end{thm} 

\begin{rem} \label{t0.3}       By	\eqref{eq:2.092}, 
$\underline{\lambda_{-1}(N_{\alpha}^*)^{-1}_{\mN}}$
is a sum of virtual vector bundles on $Y_{\alpha}^{S^1}$ 
with coefficients in 
\begin{align}\label{eq:0.12} 
F(x)/\prod_{ v: 
N_{\alpha, v}\neq 0}(x^v-1)^{\rank N_{\alpha,v}+\mN}\in R(S^1)_{I(g)}
\end{align}
with $F(x)\in \Z[x]$, where $\Z[x]$ means the ring of
polynomials 
in $x$ 
with integral coefficients and $\rank N_{\bullet}$
is the rank of the complex vector bundle $N_{\bullet}$. 
\end{rem}

For $g\in S^1$
set
\begin{align}\label{i05} 
\Q_g:=\{P(g)/Q(g)\in\C: P,Q\in \Z[x], Q(g)\neq 0 \}\subset\C.
\end{align}

Let $\iota:Y^{S^1}\rightarrow Y$ be the 
canonical embedding. 
Let $\hat{\iota}^*:\widehat{K}^0_g(Y)_{I(g)}\rightarrow 
\widehat{K}^0_g(Y^{S^1})_{I(g)}$ be the induced homomorphism. 

\begin{thm}\label{i08} 
For $g\in S^1$, the direct image map 
$\widehat{f_Y}_!: \widehat{K}_g^0(Y)_{I(g)}\rightarrow 
\C/\Q_g$,
\begin{align}\label{i09}
[\underline{E}, 
\phi]/\chi
\mapsto\chi
(g)^{-1}\left(-\int_{Y^g}\td_g(\nabla^{TY},\nabla^L)\wedge\phi
+\bar{\eta}_g(\underline{TY}, \underline{L}, \underline{E})\right)
\end{align}
is well-defined. 

For any $g\in S^1\backslash A$, $\hat{\iota}^*$ is an isomorphism and
the following diagram commutes
		\begin{equation}\label{eq:0.14} 
		\begin{split}
\xymatrix{
	\widehat{K}^0_g(Y^{S^1})_{I(g)} 
	\ar[rd]_{\widehat{f_{Y^{S^1}}}_!}
	&&\widehat{K}^0_g(Y)_{I(g)}
	\ar[ll]_{
		\left[\underline{\lambda_{-1}(N^*)},0\right]^{-1}
		\cup\,\hat{\iota}^*} 
	\ar[ld]^{\widehat{f_Y}_!}\\
	&\C/\Q_g,&
}
\end{split}
\end{equation}
where the product $\cup$ is defined in \eqref{eq:2.084}.
In particular, taking into account \eqref{i09}, 
we have for any $\mN\in \N$ with $\mN>\mN_0$, 
where $\mN_0$ is as in Theorem \ref{i06},
\begin{align}\label{i13} 
\bar{\eta}_g\left(\underline{TY}, \underline{L},
\underline{E}\right)-\sum_{\alpha}\bar{\eta}_g
\left(\underline{TY_{\alpha}^{S^1}}, \underline{L_{\alpha}},
\underline{\lambda_{-1}(N_{\alpha}^*)^{-1}_{\mN}}\otimes 
\underline{E}|_{Y^{S^1}_{\alpha}}\right)\in \Q_g.
\end{align}
\end{thm}

The final main result of our paper is as follows:
\begin{thm}\label{i14} 
Let $A\subset S^1$ and $\mN_0\in \N$ be as in Theorem \ref{i06}. 
Then for any $\mN\in \N$, $\mN>\mN_0$, for any equivariant 
geometric triple $\underline{E}$ on $Y$,
the function on $g\in S^1\backslash A$,
	\begin{align}\label{i16} 
	\bar{\eta}_g\left(\underline{TY}, \underline{L},
	\underline{E}\right)-\sum_{\alpha}\bar{\eta}_g
	\left(\underline{TY_{\alpha}^{S^1}}, \underline{L_{\alpha}},
	\underline{\lambda_{-1}(N_{\alpha}^*)^{-1}_{\mN}}\otimes 
	\underline{E}|_{Y^{S^1}_{\alpha}}\right)
	\end{align}
	is the restriction on $S^1\backslash A$ of a rational 
	function on $S^1$ with integral coefficients 
that does not have poles
 on $S^1\backslash A$. 	
\end{thm}

In the last part of this paper (see Section \ref{s0405}) 
we discuss the case when $Y^{S^1}=\emptyset$.

\begin{thm}[See Theorem \ref{thm:3.13}]\label{i10}
If $Y^{S^1}=\emptyset$ and $A=\{g\in S^1 : Y^g\neq\emptyset \}$,
then $\bar{\eta}_g(\underline{TY}, \underline{L},\underline{E})$ 
regarded as a function on $S^1\backslash A$, is the restriction 
of a rational function on $S^1$ with integral coefficients and
without poles in $S^1\backslash A$.
\end{thm}

Let us explain why we do not work directly with
all elements $g\in S^1\backslash A$ in Theorems \ref{i06} and \ref{i08}.
Note on one hand that the equivalence relation defining 
$\widehat{K}_g^0(Y)$ (cf.\ (\ref{eq:2.083}) in 
Definition \ref{defn:2.12})  really depends 
on $g\in S^1$ even for $g\in S^1\backslash A$. Thus we can neither 
define the differential $K$-group 
$\widehat{K}_{S^1}^0(Y)$, nor localize uniformly
on $g\in S^1\backslash A$. On the other hand, even in the classical
situation from Theorem \ref{as01}, we only localize at each 
element. To get Theorem \ref{i14}, 
a certain uniform version of Theorem \ref{i08} on 
$g\in S^1\backslash A$, we need to use
 Theorems \ref{local90} and  \ref{local96},
 i.e.,  our extension of Goette's result, which is roughly saying
that the equivariant $\eta$-invariant is a
meromorphic  function in $g\in S^1$ with possible 
poles in $A$ and whose singularity is locally computable on $Y^g$.

We give at the end of the introduction a
proof of Theorem \ref{i10} and a formal computation
for \eqref{i22} in a special case.

We suppose that there exists an oriented even dimensional 
$S^1$-equivariant Spin$^c$ Riemannian compact manifold $X$ 
with boundary $Y=\partial X$ and associated $S^1$-equivariant 
Hermitian line bundle $\underline{\mL}=(\mL, h^{\mL},\nabla^{\mL})$, 
and an $S^1$-equivariant Hermitian vector bundle $(\mE, h^{\mE})$ 
with $S^1$-invariant Hermitian connection 
$\nabla^{\mE}$ such that $(X,g^{TX})$, 
$\underline{\mE}$ and $\underline{\mL}$ are of product structure 
near the boundary and
\begin{align}\label{eq:0.20}
\mS(TX,\mL)=\mS^+(TX,\mL)\oplus 
\mS^-(TX,\mL),\quad \mS^+(TX,\mL)|_Y=\mS(TY,L).
\end{align}
Then the index $\ind_{\mathrm{APS}}(D^X\otimes \mE)$
of the Dirac operator $D^X\otimes \mE$ with respect to the APS 
boundary condition, is 
a finite dimensional virtual $S^1$-representation.
By \cite[Theorem 1.2]{Donnelly78}, 
we have for any $g\in S^1\backslash A_1$,
\begin{align}\label{eq:0.22} 
\ind_{\mathrm{APS}, g}(D^X\otimes \mE)=\int_{X^{S^1}}
\td_g(\underline{TX},\underline{\mL})\ch_g(\underline{\mE})
-\bar{\eta}_g(\underline{TY},\underline{L},\underline{E}),
\end{align}
where $A_1=\{h\in S^1: X^{S^1}\neq X^h \}$.

If $Y^{S^1}=\emptyset$, then $X^{S^1}$ is a 
manifold without boundary. 
Thus \eqref{eq:0.22} implies that 
$\bar{\eta}_g(\underline{TY},\underline{L},\underline{E})$ 
is the restriction to $S^1\backslash A_1$ of a rational function 
on $S^1$ with integral coefficients.

Assume now $Y^{S^1}\neq \emptyset$. We denote by 
$\ind_{\mathrm{APS}}(D^X\otimes \mE,k)$ the 
multiplicity of weight $k$ part of $S^1$-representation in
$\ind_{\mathrm{APS}}(D^X\otimes \mE)$. Then
\begin{align}\label{eq:0.24} 
\ind_{\mathrm{APS}, g}(D^X\otimes \mE)=\sum_k\ind_{\mathrm{APS}}
(D^X\otimes \mE,k)\cdot g^k,\quad \text{for any}\ g\in S^1.
\end{align}
Introduce the notation (cf. (\ref{local266}) and \eqref{eq:3.011})
\begin{multline}\label{eq:0.26} 
R(q)=q^{-\frac{1}{2}\sum_vv\rank N_v^X+\frac{1}{2}l}\bigotimes_{v>0}
\mathrm{Sym}_{q^{-v}}(\ov{N^X_v})\otimes 
\Big(\sum_v E_v q^v\Big)
\\
=\sum_k R_k q^k\in K(X^{S^1})[[q, q^{-1}]],
\end{multline}
where we denote by $\oplus_{v>0}N_v^X$ the normal bundle of 
$X^{S^1}$ in $X$ as in (\ref{local203}).
Now in view of \eqref{i19}, 
we apply the usual APS-index theorem
\cite{APS75} for 
the operator $D^{X^{S^1}}\otimes R_k$ taking into account that
$X^{S^1}$ is a manifold with boundary 
$Y^{S^1}=\partial X^{S^1}$, and 
we obtain
\begin{align}\label{eq:0.27} 
\ind_{\mathrm{APS}}(D^{X^{S^1}}\otimes R_k)=\int_{X^{S^1}}
\td(\underline{TX^{S^1}},\underline{\mL'})\ch(\underline{R_k})
-\sum_{\alpha}\bar{\eta}(\underline{TY_{\alpha}^{S^1}},
\underline{L_{\alpha}},\underline{R_k}),
\end{align}
where $\underline{\mL'}$ is the associated line bundle over 
$X^{S^1}$ defined as in \eqref{local210}.
By the general analytic localization technique in index theory 
by Bismut-Lebeau \cite{BL91}, we can expect that the arguments in 
\cite[Theorem 1.2]{DaiZ00}, \cite[\S 1.2]{LMZ03}  
extend to the APS-index case, i.e., we can expect that
\begin{align}\label{eq:0.28} 
\ind_{\mathrm{APS}}(D^X\otimes \mE,k)
=\ind_{\mathrm{APS}}(D^{X^{S^1}}\otimes R_k)+\mathrm{sf}(Y,k),
 \quad \text{for any}\ k\in \Z,
\end{align}
where $\mathrm{sf}(Y,k)\in \Z$ is the spectral flow of a family
of  deformed $D^Y\otimes E$ operators 
 via the vector field generated by the $S^1$-action.
As in the case of manifolds without boundary we write \textit{formally}
\begin{align}\label{eq:0.30}
\sum_{k\in \Z} \td(\underline{TX^{S^1}},\underline{\mL'})
\ch(\underline{R_k})\cdot g^k
=\td_g(\underline{TX},\underline{\mL})\ch_g(\underline{\mE}).
\end{align}
Thus, we have at least formally for any $g\in S^1\backslash A_1$,
\begin{align}\label{eq:0.32}
\bar{\eta}_g(\underline{TY},\underline{L},\underline{E})
=\sum_{k\in \Z}\left(\sum_{\alpha}\bar{\eta}
(\underline{TY_{\alpha}^{S^1}}
,\underline{L_{\alpha}},\underline{R_k})-\mathrm{sf}(Y,k)
\right)\cdot g^k.
\end{align}
This heuristic discussion hints to the 
possibility of computing \eqref{i22} geometrically.
  
However, the authors are not aware of any result stating that for 
$S^1$-equivariant geometric triples
$(\underline{TY},\underline{L},\underline{E})$ as above there exist 
$k>0$ and $S^1$-equivariant $X$, $\underline{\mL}$, $\underline{\mE}$ 
such that $\partial X$ 
consists of $k$ properly oriented copies of $Y$ and by restriction 
to $\partial X$ we get $\underline{TY},\underline{L},\underline{E}$. 
Another difficulty is how to make a proper sense of the 
right-hand side of \eqref{eq:0.32}. This is why our 
intrinsic formulation of Theorem \ref{i14} does not rely on the existence 
of such an $X$. Also, it shows the usefulness of the $\gamma$-filtration 
that we introduce in differential $K$-theory.

Finally, it is natural to ask whether there is a similar 
localization formula 
\eqref{i16} for the real analytic torsion \cite{Bismut1992,RS71}. 
However, unlike the case of the holomorphic torsion and 
the $\eta$-invariant, a suitable $K$-theory 
 where the real analytic torsion is the analytic 
ingredient of a Riemann-Roch type theorem 
(cf. \cite{BG04}) is still lacking.

The main result of this paper is announced in \cite{LM18}.

This paper is organized as follows. In Section \ref{s02},
we introduce the main object of our paper, the 
equivariant $\eta$-invariant, and we review some of its 
analytic properties, which we will use in this paper, 
such as the variation formula, the embedding formula and
the comparison of equivariant $\eta$-invariants.
In Section \ref{s03}, 
we prove that the differential $K$-ring is a pre-$\lambda$-ring 
and construct the inverse of 
$[\underline{\lambda_{-1}(N_{\alpha}^*)},0]$ in 
$\widehat{K}_{g}^0(Y_{\alpha}^{S^1})_{I(g)} $ explicitly. 
In Section \ref{s04}, 
we prove Theorems \ref{i08} and \ref{i14} 
and study the case 
$Y^{S^1}=\emptyset$.
We also compute  in detail the equivariant $\eta$-invariant in the 
case $Y=S^{1}$.\\

\textbf{Notation}: For any vector space $V$ and $B\in \End(V)$, 
we denote by $\tr[B]$ the trace of $B$ on $V$.
We denote by $\dim_\R$ or $\dim_\C$ the real or complex dimension 
of a vector space, and we skip the subscript if it is clear from the context.
For a complex vector bundle $E$, we will denote by 
$\rank E$ its rank as a complex vector bundle, and
$E^{\R}$ the underlying  real vector bundle.

For $\mathbb{K}=\R$ or $\C$,  
we denote by 
$\Omega^{\bullet}(X,\mathbb{K})$ 
the space of smooth $\mathbb{K}$-valued differential forms 
on a manifold $X$, and its subspaces of even/odd degree forms 
by $\Omega^{\rm even/\rm odd}(X,\mathbb{K})$.
Let $d$ be the exterior differential, then the image of $d$
is the space of exact forms, $\mathrm{Im}\, d$.  

Let $R(S^1)$ be the representation ring of the circle group $S^{1}$. 
For any finite dimensional virtual 
$S^1$-representation $V=M-M'\in R(S^1)$ and $h\in S^1$, its character 
\begin{align}\label{eq:0.27a}
\chi_V(h)=\tr|_{M}[h]-\tr|_{M'}[h]\in\Z[h,h^{-1}],
\end{align}
a polynomial in $h$ 
and $h^{-1}$ with integral coefficients. Conversely, for any 
$f\in\Z[h,h^{-1}]$, there exists a finite dimensional virtual 
$S^1$-representation 
$V_f\in R(S^1)$ such that $f=\chi_{V_f}$ on $S^1$. So in this paper, 
we will not 
distinguish the finite dimensional virtual $S^1$-representation and 
$f\in\Z[h,h^{-1}]$ as an element of $R(S^1)$.   

\noindent{\bf Acknowledgments}. 
We would like to thank Professors Jean-Michel Bismut and Weiping Zhang
for helpful discussions.
We are especially grateful to the referees for their 
 very helpful comments and suggestions. 
We are also indebted to George Marinescu for his critical comments.
B.\ L.\ is partially supported by NNSFC No.11931997, No.11971168 and
 Science and Technology Commission 
of Shanghai Municipality (STCSM), grant No.18dz2271000.
X.\ M.\ is partially supported by
NNSFC No.11528103,  No.11829102, ANR-14-CE25-0012-01,  and
funded through the Institutional Strategy of
the University of Cologne within the German Excellence Initiative.

\section{Equivariant $\eta$-invariants}\label{s02}

The heat kernel approach to the Atiyah-Singer index theorem, 
introduced by Mckean-Singer, Gilkey,
Atiyah-Bott-Patodi, \ldots, establishes 
the local index theorem for Dirac operators.
It finds immediately many applications, in particular
the discovery of the Atiyah-Patodi-Singer index theorem
for manifolds with boundary and of the $\eta$-invariant
as the boundary contribution in this index formula.  
Bismut, along with his various collaborators,
has made groundbreaking contributions in this direction,
in particular by developing various ideas and techniques
to study the global spectral invariants such as the 
$\eta$-invariant and the analytic torsion.

In this section 
we review some facts 
about the equivariant $\eta$-invariants. These results have largely been 
influenced both philosophically and technically 
by the analytic localization technique in index theory
developed by Bismut-Lebeau. The variation formula 
which computes the difference of the equivariant
$\eta$-invariants associated with different metrics
and connections, 
is a direct consequence of Donnelly's 
equivariant APS index theorem for manifolds with boundary.
It is used in Theorem \ref{defn:3.01} to show that 
the direct image in $g$-equivariant differential $K$-theory
is well-defined. The embedding formula guarantees that 
the direct image in $g$-equivariant differential $K$-theory is 
compatible with the embedding. Note that the embedding 
of the fixed point set into the total manifold appears
naturally in our problem (\ref{i22}).

To conclude Theorem \ref{i14} from Theorem \ref{i08},
we  need to understand the analyticity of the 
equivariant $\eta$-invariant 
as a function of $g\in S^1$. In the same
way as fixed-point formulas have two equivariant versions,
the Lefschetz fixed-point formula and Kirillov-like 
formulas of Berline-Vergne, 
also equivariant $\eta$-invariants have two versions. 
In Theorems \ref{local90}, \ref{local96}, we show that 
the difference of these two equivariant $\eta$-invariants
 is given by an explicit local formula,
involving natural Chern-Simons currents.
Moreover, the Kirillov-like equivariant $\eta$-invariant
is analytic near $0\in \mathrm{Lie}(S^1)$.

This section is organized as follows.
In Section \ref{s0201}, 
we study the equivariant decomposition of $TY$,
in particular, we define the finite subset
$A\subset S^1$ in Theorem \ref{i06}. In Section \ref{s0202},
we define the equivariant 
$\eta$-invariant. In Section \ref{s0203}, 
we first introduce  some characteristic classes
and Chern-Simons classes which appear in
various situations in the whole paper, then 
we recall the variation formula.
In Section \ref{s0204}, 
we review the geometric construction of the
direct image for an embedding in topological
$K$-theory, in particular the natural metrics
and connections on the direct image constructed by Bismut-Zhang. 
Finally, we explain the embedding formula. 
In Section \ref{s0205}, 
we compare the  equivariant $\eta$-invariant 
with the equivariant infinitesimal $\eta$-invariant.

\subsection{Circle action}\label{s0201}

Let $Y$ be a smooth compact manifold with a smooth circle action. 
For $g\in S^{1}$, set
\begin{align}\label{eq:1.1}\begin{split}
&Y^{g}= 	\{y\in Y: gy=y\},\\
&Y^{S^1}=\{y\in Y: hy=y\ \text{for any}\ h\in S^1 \}.
\end{split}\end{align}
Then $Y^{g}$ is the fixed point set of $g$-action on $Y$
and $Y^{S^{1}}$ is the fixed point set of the circle action 
on $Y$ with connected components 
$\{Y^{S^1}_{\alpha} \}_{\alpha\in \mathfrak{B}}$. 
Since $Y$ is compact, the index set $\mathfrak{B}$ is a finite set.
Certainly, for any $g\in S^{1}$, $Y^{S^1}\subseteq Y^{g}$.

If $g\in S^1$ is a generator of $S^1$, that is, $g=e^{2\pi i t}$ with 
$t\in\R$ irrational, then $Y^{S^1}$ is the fixed point set $Y^g$ of $g$. 
We have the decomposition of real vector bundles over 
$Y^{S^1}_{\alpha}$
\begin{align}\label{local30}
TY|_{Y^{S^1}_{\alpha}}=TY^{S^1}_{\alpha}\oplus
\bigoplus_{v\neq 0}N_{\alpha,v}^{\R},
\end{align}
where $N_{\alpha,v}^{\R}$ is the underlying real vector bundle 
of a complex vector bundle $N_{\alpha,v}$ over $Y^{S^1}_{\alpha}$ 
such that $g$ acts on $N_{\alpha,v}$ by multiplication by $g^v$. 
Let $N$ be the normal bundle of 
$Y^{S^1}$ in $Y$.
Then \eqref{local30} induces the canonical identification
$N_{\alpha}:=N|_{Y^{S^1}_{\alpha}}
 =\oplus_{v\neq 0}N_{\alpha,v}^{\R}$. 
We will regard $N_{\alpha}$ as a complex vector bundle. 
The complex conjugation provides a $\C$-anti-linear isomorphism
between the complex vector bundles 
$N_{\alpha,v}$ and $\ov{N}_{\alpha,-v}$.
Since we can choose either $N_{\alpha,v}$ or $\ov{N}_{\alpha,v}$ 
as the complex vector bundle for $N_{\alpha,v}^{\R}$, 
in what follows,
we may 
and we will assume that
\begin{align}\label{local201}
TY|_{Y^{S^1}_{\alpha}}=TY^{S^1}_{\alpha}\oplus
\bigoplus_{v>0}N_{\alpha,v}^{\R},\quad 
N_{\alpha}=\bigoplus_{v>0}N_{\alpha,v}.
\end{align}

Since the dimension of $TY$ is finite, there are only finitely many $v$ 
such that $\rank N_{\alpha,v}\neq 0$. 
Set
\begin{align}\label{local205}
q=\max\{v : \text{there exists}\ \alpha\in \mathfrak{B} \
\text{such that}\
\rank N_{\alpha,v}\neq 0\}.
\end{align}
Then we have the splittings 
\begin{align}\label{local203} 
TY|_{Y^{S^1}_{\alpha}}=TY^{S^1}_{\alpha}\oplus
\bigoplus_{v=1}^{q}N_{\alpha,v}^{\R},\quad N_{\alpha}
=\bigoplus_{v=1}^{q}N_{\alpha,v}.
\end{align}

\begin{prop}\label{local32} 
	The set 
\begin{align}\label{eq:1.6} 
A=\{g\in S^{1}: Y^{S^{1}}\neq Y^{g} \}
\end{align}	
is  finite. 
\end{prop}
\begin{proof}
Let $g^{TY}$ be an $S^1$-invariant metric on $Y$. Then there exists 
$\var>0$ such that the exponential map
\begin{align}\label{eq:1.5b}
(y,Z)\in \mU_{\var}=\{(y,Z)\in N_y: y\in Y^{S^1},  |Z|<\var \}\rightarrow 
\exp_y(Z)
\end{align}
is a diffeomorphism from $\mU_{\var}$ into the tubular neighborhood
$\mV_{\var}$ of $Y^{S^1}$ in $Y$. 
Then for any $g_t=e^{2\pi it}\in S^1$, $(y,Z)\in \mU_{\var}$ 
with $Z=(z_{\alpha,v})_{v=1}^q\in N_{\alpha, y}$, we have
\begin{align}\label{eq:1.9b} 
g_{t}(y,Z)=(y, (e^{2\pi i t v}z_{\alpha,v})_{v=1}^q).
\end{align}
Thus for $t\in (0,1)$, 
\begin{align}\label{eq:1.10b} 
\mU_{\var}^{g_t}=Y^{S^1},\quad \text{if}\ t\notin \mF_q
:=\left\{\frac{k}{p}: k,p\ \text{coprime}, 0\leq k<p\leq q\right\}.
\end{align}

Now $S^1$ acts locally freely on $Y_1:=Y\backslash\mU_{\var/2}$.
Thus for any $x\in Y_1$, the stabilizer $S_x^1$ of $x$ is 
a finite group $\{e^{2\pi i j/k}: 0\leq j<k \}\simeq \Z_k$ for certain 
$k\in \N^*$. 
Set $N_x=T_xY_1/T_x(S^1\cdot x)$, which is a linear representation
of $S_x^1\simeq \Z_k$.
By the slice theorem, there exists an $S^1$-equivariant 
diffeomorphism from an equivariant open neighborhood $U_x$
of the zero section in $S^1\times_{S_x^1}N_x$ to an 
open neighborhood of $S^1\cdot x$ in $Y_1$, which sends the zero 
section $S^1/S_x^1$ onto the orbit $S^1\cdot x$ by the map
$g\in S^1\rightarrow g\cdot x$. Now in this neighborhood $U_x$,
for any $g\in S^1\backslash S_x^1$, $U_x^g=\emptyset$. 
By using the compactness of $Y_1$  
there is a finite set 
$A_0\subset S^1$ such that for any $g\in S^1\backslash A_0$,
$Y_1^g=\emptyset$.
Combining with (\ref{eq:1.10b}), we know $A$ in (\ref{eq:1.6})
is finite.

The proof of Proposition \ref{local32} is completed.
\end{proof}

\subsection{Equivariant $\eta$-invariants}\label{s0202}	

In the 
remainder of this section,
let $Y$ be an odd dimensional compact oriented manifold
with circle action. Then 
the circle action automatically preserves the orientation of $Y$.
Let $g^{TY}$ be an $S^1$-invariant metric on $TY$.

Assume that $Y$ has an $S^1$-equivariant spin$^c$ structure,
i.e., the $S^{1}$-action on $Y$ lifts naturally to
 the associated 
Spin$^c$  principal bundle, in particular,
it induces an $S^1$-equivariant complex line bundle $L$
such that
$\omega_2(TY)=c_1(L)\,\mathrm{mod}\,(2)$, where $\omega_2$
is the second 
Stiefel-Whitney class and $c_1$ is the first Chern class \cite[Appendix 
D]{LaMi89}.  Let $\mS(TY,L)$ be the fundamental complex spinor 
bundle associated with this spin$^c$ structure.
It is an $S^1$-equivariant complex vector bundle in a canonical way, 
and formally
\begin{align}\label{local206} 
\mS(TY,L)=\mS_0(TY)\otimes L^{1/2},
\end{align}
where $\mS_0(TY)$ is the fundamental spinor bundle for the (possibly 
non-existent) spin structure on $TY$ and $L^{1/2}$ is the (possibly 
non-existent) square root of $L$. 

 Let $E$ be an $S^1$-equivariant complex vector bundle over $Y$. 
Then $S^1$ acts on $\cC^{\infty}(Y,\mS(TY,L)\otimes E)$ by
\begin{align}\label{eq:1.16b} 
(g.s)(x)=g(s(g^{-1}x)),\quad \text{for}\ g\in S^1.
\end{align} 
 
Let $h^L$ and $h^E$ be $S^1$-invariant 
Hermitian metrics on $L$ and $E$ respectively. 
Let $h^{\mS_Y}$ be the $S^1$-invariant Hermitian metric on 
$\mS(TY,L)$ induced by $g^{TY}$ and $h^L$. 

Let $\nabla^{TY}$ be the Levi-Civita connection on $(TY,g^{TY})$. Let 
$\nabla^L$ and $\nabla^E$ be $S^1$-invariant Hermitian 
connections on $(L,h^L)$ and $(E,h^E)$ respectively.
Let $\nabla^{\mS_Y}$ be the connection on $\mS(TY,L)$ 
induced by $\nabla^{TY}$ and $\nabla^L$ 
\cite[Appendix D]{LaMi89}. 
Let $\nabla^{\mS_Y\otimes E}$ be the connection on 
$\mS(TY,L)\otimes E$ induced by $\nabla^{\mS_Y}$ and $\nabla^E$,
\begin{align}\label{local207} 
\nabla^{\mS_Y\otimes E}=\nabla^{\mS_Y}\otimes 1+1\otimes\nabla^{E}.
\end{align}

Let $\{e_j \}$ be a locally orthonormal frame of $(TY,g^{TY})$. 
We denote 
by $c(\cdot)$ the Clifford action of $TY$ on $\mS(TY,L)$. 
Let $D^Y\otimes E$ 
be the spin$^c$ Dirac operator on $Y$ defined by  
\begin{align}\label{local208} 
D^Y\otimes E=\sum_jc(e_j)\nabla_{e_j}^{\mS_Y\otimes E}:
\cC^{\infty}(Y, \mS(TY,L)\otimes E)\to 
\cC^{\infty}(Y, \mS(TY,L)\otimes E).
\end{align}
Then $D^Y\otimes E$ is an $S^1$-equivariant first order self-adjoint 
elliptic differential operator on $Y$ and its kernel $\Ker
(D^Y\otimes E)$ is a finite dimensional $S^1$-complex vector space.

Let $\exp(-u(D^Y\otimes E)^{2})$, 
$u>0$, be the heat semi-group of $(D^Y\otimes E)^{2}$.

We denote by $\underline{TY}, \underline{L}, \underline{E}$ 
the equivariant geometric data 
\begin{align}\label{local209} 
\underline{TY}=(TY,g^{TY},\nabla^{TY}),\quad
\underline{L}=(L,h^L,\nabla^L),\quad \underline{E}=(E,h^E,\nabla^E).
\end{align}
We also call $\underline{TY}, \underline{L}, \underline{E}$ 
equivariant geometric triples over $Y$. 

\begin{defn}\label{local128}
For $g\in S^1$,  the equivariant  (reduced) $\eta$-invariant associated 
with $\underline{TY}, \underline{L}, \underline{E}$ is defined by 
\begin{multline}\label{local59} 
\bar{\eta}_g(\underline{TY}, \underline{L}, \underline{E})
=\int_0^{+\infty} \tr\big[g(D^Y\otimes E)\exp(-u(D^Y\otimes E)^{2})
\big]\frac{du}{2\sqrt{\pi u}}    \\
+\frac{1}{2}\tr|_{\Ker (D^Y\otimes E)}[g]\in \C.
\end{multline}
\end{defn}
The convergence of the integral at $u=0$ in \eqref{local59} is nontrivial 
(see e.g., \cite[Theorem 2.6]{BF86b}, \cite{Donnelly78}, 
\cite[Theorem 2.1]{Z90}).

\subsection{Variation formula}\label{s0203}

Since $g^{TY}$ is $S^1$-invariant, 
the fixed point set $Y^g$ is an odd dimensional totally geodesic 
submanifold of $Y$ for any $g\in S^1$. 
Let $N^{\R}$ be the normal bundle of $Y^g$ in $Y$, 
which we identify to the orthogonal complement of $TY^g$ in $TY$.	

Since the $S^1$-action preserves the spin$^c$ structure, 
we see that $Y^g$ is canonically oriented (cf. 
\cite[Proposition 6.14]{BeGeVe04}, \cite[Lemma 4.1]{LMZ002}). 

Assume first $g=e^{2\pi i t}\in S^1\backslash A$ 
(cf. \eqref{eq:1.6}), then $Y^g=Y^{S^1}$ 
and by (\ref{local201}), we have the 
decomposition of real vector bundles over $Y^g$,
\begin{align}\label{e01047}
TY|_{Y^g}=TY^g\oplus\bigoplus_{v>0}N_{v}^{\R},
\quad N^{\R}=\bigoplus_{v>0} N_v^{\R},
\end{align}
where $N_{v}^{\R}$ is the underlying real vector bundle
of the complex vector bundle $N_{v}$
such that $h\in S^1$ acts by multiplication by $h^v$.
We will fix the orientation on $Y^g=Y^{S^1}$ induced by
the canonical orientation on $N_v$ as complex vector bundles
and the orientation on $TY$.

Since $g^{TY}$ is $S^1$-invariant, the 
decomposition \eqref{e01047} is orthogonal and the restriction of 
$\nabla^{TY}$ on $Y^g$ is split under the decomposition. 
Let $g^{TY^g}$, $g^{N^{\R}}$ and $g^{N_v^{\R}}$ be the metrics 
induced by $g^{TY}$ on $TY^g$, $N^{\R}$ and $N_v^{\R}$.
Let $\nabla^{TY^g}$, $\nabla^{N^{\R}}$ and $\nabla^{N_v^{\R}}$
be the corresponding induced connections 
on $TY^g$, $N^{\R}$ and $N_v^{\R}$,
with curvatures $R^{TY^g}$, $R^{N^{\R}}$ and $R^{N_v^{\R}}$. 
Then under the decomposition \eqref{e01047},
\begin{align}\label{local300}
g^{TY}=g^{TY^g}\oplus g^{N^{\R}}, \quad g^{N^{\R}}
=\bigoplus_vg^{N_v^{\R}},\quad \nabla^{TY}|_{Y^g}
=\nabla^{TY^g}\oplus\bigoplus_v\nabla^{N_v^{\R}}.
\end{align}

Similar to \eqref{e01047}, we have the orthogonal decomposition of 
complex vector bundles with connections on $Y^g$
\begin{align}\label{local266} 
E|_{Y^g}=\bigoplus_{v}E_v, \quad \nabla^{E}|_{Y^g}
=\bigoplus_v\nabla^{E_v}.
\end{align}
Here $h\in S^1$ acts by multiplication by $h^v$ on $E_v$
and the connection $\nabla^{E_v}$ on $E_v$ is induced by $\nabla^E$.
Let $R^E$, $R^{E_v}$ be the curvatures of $\nabla^E$, $\nabla^{E_v}$.

\begin{defn}\label{local213}
For $g=e^{2\pi i t}\in S^1\backslash A$, 
set
\begin{align}\label{e01051}
\begin{split}
\widehat{\mathrm{A}}(TY^g,\nabla^{TY^g})
:&=\mathrm{det}^{1/2}\left(\frac{\frac{i}{4\pi}	R^{TY^g}}{\sinh
	\left(\frac{i}{4\pi}R^{TY^g}\right)}\right),
\\
\widehat{\mathrm{A}}_g(N^{\R},\nabla^{N^{\R}}):&=
\left(i^{\frac{1}{2}\dim 
	N^{\R}}\mathrm{det}^{1/2}|_{N^{\R}}\left(1-
g\cdot \exp\left(\frac{i}{2\pi}R^{N^{\R}}\right)\right)\right)^{-1}
\\
&=\prod_{v>0}\left(i^{\frac{1}{2}\dim 
	N_v^{\R}}\mathrm{det}^{1/2}|_{N_v^{\R}}\left(1-
g\cdot \exp\left(\frac{i}{2\pi}R^{N_v^{\R}}\right)\right)\right)^{-1},
\\
\widehat{\mathrm{A}}_g(TY,\nabla^{TY})
:&=\widehat{\mathrm{A}}(TY^g,\nabla^{TY^g})\cdot
\widehat{\mathrm{A}}_g(N^{\R},\nabla^{N^{\R}}) \in
\Omega^{\bullet}(Y^g, \C),
\\
\ch_g(\underline{E}):&=\tr\left[g\exp\left(\frac{i}{2\pi}R^{E} 
\right)\right]
\\
&=\sum_v\tr\left[\exp\left(\frac{i}{2\pi}R^{E_v} 
+2i\pi v t  \right)\right]\in
\Omega^{\bullet}(Y^g, \C).
\end{split}
\end{align}
The sign convention in 
$\widehat{\mathrm{A}}_g(N^{\R},\nabla^{N^{\R}})$
is that the degree $0$ part is given by
$\prod_{v>0}(2i\sin(\pi vt))^{-\frac{1}{2}\dim N_v^{\R}}$. 
\end{defn}
The forms in \eqref{e01051} are closed forms on $Y^{g}$ and their
cohomology class does not depend on the $S^{1}$-invariant metrics
$g^{TY}$, $h^{E}$ and connection $\nabla^{E}$.  We
denote by $\widehat{\mathrm{A}}(TY^g)$,
$\widehat{\mathrm{A}}_g(TY)$, $\ch_g({E})$
 their cohomology classes, two of which appear 
 in the equivariant index theorem \cite[Chapter 6]{BeGeVe04}.

Comparing with \eqref{e01051}, 
if $h\in S^1$ acts on $L|_{Y^{S^1}}$ by multiplication by $h^l$, we write
\begin{align}\label{e01138}
\ch_g(\underline{L^{1/2}}):=\exp\left(\frac{i}{4\pi}R^{L}|_{Y^g}
+i\pi l t\right)\in \Omega^{\bullet}(Y^g, \C).
\end{align}
We denote by
\begin{align}\label{bl0660} 
\td_g(\nabla^{TY}, 
\nabla^{L}):=\widehat{\mathrm{A}}_g(TY,\nabla^{TY}) 
\ch_g(\underline{L^{1/2}}).
\end{align}	
Note that the natural lift of $g=e^{2\pi it}$
on $\mS(TY,L)$ over $Y^{S^1}$ is given by
\begin{align}\label{eq:1.21a}
\prod_{v}\prod_j\left(\cos\left(\pi vt\right)
+\sin \left(\pi vt\right)c\left(e_{2j-1}^{v} \right) 
c\left(e_{2j}^{v} \right)\right)\cdot e^{i\pi lt},
\end{align}
where $\left\{e_{j}^{v} \right\}_j$ is an 
oriented orthonormal frame of $N_v$.
This explains the sign convention in (\ref{bl0660}).

If $g\in A$, we have the decomposition of real vector bundles over $Y^g$,
\begin{align}\label{eq:1.24b} 
TY|_{Y^g}=TY^g\oplus \bigoplus_{0<\theta\leq \pi}N(\theta),
\end{align}
where $N(\theta)$ is a real vector bundle over $Y^g$ which has
a complex structure such that $g$ acts 
by multiplication by $e^{i\theta}$
if $\theta\neq \pi$; or an even dimensional oriented real vector bundle 
on which $g$ acts by multiplication by $-1$ if $\theta=\pi$. 
We fix the orientation on $Y^g$ induced by the orientations
on $Y$ and on $N(\theta)$.
Then we can still define $\ch_g(\underline{E})$ as in (\ref{e01051}).
Let $g$ act on $L_{Y^g}$ by multiplication by $e^{i\theta'}$,
$0\leq \theta'<2\pi$.
Now the lift of the $g$-action on $\mS(TY,L)$
over $Y^g$ is given by 
\begin{align}\label{eq:1.22a}
\epsilon\prod_{0<\theta\leq \pi}\prod_j\left(\cos\left(
\frac{\theta}{2}\right)+\sin \left(
\frac{\theta}{2}\right)c\left(e_{2j-1}^{\theta} \right) 
c\left(e_{2j}^{\theta} \right)\right)\cdot e^{i\theta'/2}
\quad \text{and}\quad \epsilon=1\ \text{or}-1,
\end{align}
where $\left\{e_{j}^{\theta} \right\}_j$ is an 
oriented orthonormal frame of $N(\theta)$.
From (\ref{eq:1.22a}), the sign convention of 
$\td_g(\nabla^{TY}, 
\nabla^{L})$ in (\ref{bl0660}) is that its degree $0$
part is given by $\epsilon\prod_{0<\theta\leq\pi}
(2i\sin(\theta/2))^{-\frac{1}{2}\dim N(\theta)}e^{i\theta'/2}$.
This situation is only used in 
Sections \ref{s0203}, \ref{s0205} and \ref{s0401}.

We explain now the construction of Chern-Simons 
classes. Let 
\begin{align}\label{eq:1.25a}
\underline{TY_j}=(TY, g_j^{TY},\nabla_j^{TY}),\  
\underline{L_j}=(L, 
h_j^L,\nabla_j^L),\  \text{and}\ 
\underline{E_j}=(E, h_j^{E}, \nabla_j^{E})\quad \text{for}\ j=0,1
\end{align}
 be equivariant geometric triples over $Y$
as in (\ref{local209}).

Let $\pi:(y,s)\in Y\times \R\rightarrow y\in Y$ 
be the obvious projection. Then the 
$S^1$-action
lifts naturally on $Y\times \R$, by acting only on the factor $Y$.  Let 
$g^{\pi^*TY}$, $h^{\pi^*L}$ and $h^{\pi^*E}$ be 
$S^1$-invariant metrics on $\pi^*TY$, $\pi^*L$ and $\pi^*E$ over 
$Y\times\R$ such that for $j=0,1$,
\begin{align}\label{local301} 
g^{\pi^*TY}|_{Y\times\{j\}}=g_j^{TY},\quad h^{\pi^*L}|_{Y\times\{j\}}
=h_j^L,\quad h^{\pi^*E}|_{Y\times\{j\}}=h_j^E .
\end{align}
Let $\nabla^{\pi^*TY}$, $\nabla^{\pi^*L}$ and 
$\nabla^{\pi^*E}$ be $S^1$-invariant Hermitian 
connections on $(\pi^*TY, g^{\pi^*TY})$, $(\pi^*L, h^{\pi^*L})$ 
and $(\pi^*E, h^{\pi^*E})$ such that for $j=0,1$,
\begin{align}\label{local302} 
\nabla^{\pi^*TY}|_{Y\times\{j\}}=\nabla_j^{TY},\quad 
\nabla^{\pi^*L}|_{Y\times\{j\}}=\nabla_j^L,\quad 
\nabla^{\pi^*E}|_{Y\times\{j\}}=\nabla_j^E .
\end{align} 
Let $\underline{\pi^*E}=(\pi^*E, h^{\pi^*E}, \nabla^{\pi^*E})$ be 
the associated geometric triple on $Y\times \R$. 

If $\alpha=\alpha_0+ ds\wedge\alpha_1$ with 
$\alpha_0, \alpha_1\in \Lambda^{\bullet}(T^*Y)$, put
\begin{align}\label{eq:1.27}
\{\alpha\}^{ds}:=\alpha_1.
\end{align}

For $g\in S^1$, the equivariant Chern-Simons classes
$\widetilde{\ch}_g(\underline{E_0},\underline{E_1}), 
\widetilde{\td}_g(\nabla_0^{TY}, \nabla_0^L, \nabla_1^{TY},
\nabla_1^{L})\in 
\Omega^{\mathrm{odd}}(Y^g,\C)/ \Im\, d$ 
are defined by
\begin{align}\label{local40} 
\begin{split}
&\widetilde{\ch}_g(\underline{E_0},\underline{E_1})
=\int_0^1\{\ch_g(\underline{\pi^*E})\}^{ds}ds\in
\Omega^{\mathrm{odd}}(Y^g,\C)/ \Im \, d,
\\
&\widetilde{\td}_g(\nabla_0^{TY}, \nabla_0^L, \nabla_1^{TY},
\nabla_1^{L})
=\int_0^1\{\td_g(\nabla^{\pi^*TY},\nabla^{\pi^*L})\}^{ds}ds\in
\Omega^{\mathrm{odd}}(Y^g,\C)/ \Im \, d.
\end{split}
\end{align}
Moreover, we have
\begin{align}\label{local126}
\begin{split} 
&d\, \widetilde{\ch}_g(\underline{E_0}, 
\underline{E_1})=\ch_g(\underline{E_1})-\ch_g(\underline{E_0}),
\\
&d\, \widetilde{\td}_g(\nabla_0^{TY}, \nabla_0^L, \nabla_1^{TY},
\nabla_1^{L})=\td_g(\nabla_1^{TY},\nabla_1^L)
-\td_g(\nabla_0^{TY},\nabla_0^L).
\end{split}
\end{align}
Note that the Chern-Simons classes depend only on $\nabla_j^{TY}$, 
$\nabla_j^{L}$ and $\nabla_j^{E}$ for $j=0,1$
(see \cite[Theorem B.5.4]{MM07}).

Let $\bar{\eta}_g(\underline{TY_j}, \underline{L_j}, \underline{E_j})$ 
 for $j=0,1$
be the equivariant reduced $\eta$-invariants associated 
with ($\underline{TY_j}$, $\underline{L_j}$, $\underline{E_j}$).
The following variation formula is proved in 
\cite[Proposition 2.14]{Liu2016} (see also \cite[Theorem 2.6]{Liu2017}),
which extends the usual well-known non-equivariant 
variation formula for $\eta$-invariants (cf. \cite[p95]{APS76} 
or  \cite[Theorem 2.11]{BF86b}). 

Recall that for a finite dimensional virtual $S^1$-representation $V$,
we denote its character by $\chi_V$ (cf. (\ref{eq:0.27a})).

\begin{thm}\label{local85} 
There exists $V \in R(S^1)$ such that for any $g\in S^1$, 
\begin{multline}\label{local73} 
\bar{\eta}_g(\underline{TY_1}, \underline{L_1}, 
\underline{E_1})-\bar{\eta}_g(\underline{TY_0}, \underline{L_0}, 
\underline{E_0})=\int_{Y^{g}}
\widetilde{\td}_g(\nabla_0^{TY}, \nabla_0^L, \nabla_1^{TY},
\nabla_1^{L})\ch_g(\underline{E_1})\\
+\int_{Y^g}\td_g(\nabla_0^{TY},\nabla_0^L)
\wi{\ch}_g(\underline{E_0},\underline{E_1})
+\chi_{V}(g).
\end{multline}
\end{thm}


\subsection{Embedding formula for equivariant $\eta$-invariants}
\label{s0204}

Recall that $\{Y^{S^1}_{\alpha}\}_{\alpha\in \mathfrak{B}}$ is 
the set of the connected components of $Y^{S^1}$ and 
$N_{\alpha}$ is the normal bundle of $Y^{S^1}_{\alpha}$ in $Y$. 
We consider $N_{\alpha}$ as a complex vector bundle and denote by 
$N_{\alpha}^{\R}$ the underlying real vector bundle of $N_{\alpha}$.
Then 
$N_{\alpha}^{\R}\otimes_{\R}\C= N_{\alpha}\oplus\ov{N}_{\alpha}$.
Let $h^N$ be the Hermitian metric on $N_{\alpha}$
induced by $g^{N_{\alpha}}$.

Let $C(N_{\alpha}^{\R})$ be the Clifford algebra bundle of 
$(N_{\alpha}^{\R},g^{N^{\R}})$. 
Then $\Lambda(\ov{N}_{\alpha}^*)$ is a 
$C(N_{\alpha}^{\R})$-Clifford module. Namely, if $u\in N_{\alpha}$, let 
$u^*\in \ov{N}_{\alpha}^*$ be the metric dual of $u$.
The Clifford action on $\Lambda(\ov{N}_{\alpha}^*)$
is defined by
\begin{align}\label{local211} 
c(u)=\sqrt{2}u^*\wedge,\quad c(\ov{u})=-\sqrt{2}i_{\ov{u}}
\quad \text{ for any } u\in N_{\alpha}.
\end{align} 
Here $\wedge, i_{\cdot}$ are the 
exterior and interior products on forms.

Set
\begin{align}\label{local210} 
L_{\alpha}=L|_{Y^{S^1}_{\alpha}}\otimes(\det N_{\alpha})^{-1}.
\end{align}
Then $TY^{S^1}_{\alpha}$ has an equivariant spin$^c$ structure as 
$\omega_2(TY^{S^1}_{\alpha})=c_1(L_{\alpha})
\in H^2_{S^{1}}(Y_{\alpha}^{S^1},\Z)\,\mathrm{mod}\ (2)$ 
(cf. \cite[(1.47)]{LMZ03}). Let 
$\mS(TY^{S^1}_{\alpha},L_{\alpha})$ be the associated fundamental 
spinor bundle for $TY^{S^1}_{\alpha}$ such that
\begin{align}\label{local75} 
\mS(TY,L)|_{Y^{S^1}_{\alpha}}=\mS(TY_{\alpha}^{S^1},
L_{\alpha})\otimes \Lambda^{\bullet}(\overline{N}_{\alpha}^*).
\end{align}
As in \eqref{local206}, formally, we have
\begin{align}\label{eq:1.33}
\begin{split}
\mS(TY^{S^1}_{\alpha},L_{\alpha})&=\mS_0(TY^{S^1}_{\alpha})
\otimes L^{1/2}|_{Y^{S^1}_{\alpha}}\otimes (\det N_{\alpha})^{-1/2},
\\
\Lambda^{\bullet}(\ov{N}_{\alpha}^*)
&=\mS_0(N_{\alpha}^{\R})\otimes (\det N_{\alpha})^{1/2}.
\end{split}
\end{align}

Let $\nabla^N$ be the Hermitian connection on 
$(N_{\alpha},h^N)$ induced by $\nabla^{N^{\R}}$ 
in \eqref{local300}.
Note that the equivariant geometric triple 
$\underline{N_{\alpha}}=(N_{\alpha},h^N, \nabla^N)$ 
induces equivariant geometric 
triples $\underline{\Lambda^{\mathrm{even}}(N_{\alpha}^*)}$, 
$\underline{\Lambda^{\mathrm{odd}}(N_{\alpha}^*)}$ and 
$\underline{\det N_{\alpha}}$.
Denote by 
\begin{align}\label{eq:1.35}\begin{split}
\lambda_{-1}(N_{\alpha}^*)=\Lambda^{\mathrm{even}}(N_{\alpha}^*)
-\Lambda^{\mathrm{odd}}(N_{\alpha}^*), \\
\underline{\lambda_{-1}(N_{\alpha}^*)}
=\underline{\Lambda^{\mathrm{even}}(N_{\alpha}^*)}
-\underline{\Lambda^{\mathrm{odd}}(N_{\alpha}^*)}.
 \end{split} \end{align}
Let $\underline{L_{\alpha}}$ be the equivariant geometric triple 
induced from \eqref{local210}.

From \cite[(6.26)]{B95}, 
we have
\begin{align}\label{local324} 
\ch_g\left(\underline{\lambda_{-1}(N_{\alpha}^*)}\right)
=\widehat{\mathrm{A}}_g(N^{\R},\nabla^{N^{\R}})^{-1}
\cdot \ch_g\left(\underline{(\det N_{\alpha})^{-1/2}} \right).
\end{align}
From \eqref{e01051}-\eqref{bl0660}, \eqref{local210} 
and  \eqref{local324}, on $Y_{\alpha}^{S^1}$, we have
\begin{align}\label{local76} 
\td_g(\nabla^{TY},\nabla^L)
\ch_g\left(\underline{\lambda_{-1}(N_{\alpha}^*)}\right)
= \widehat{\mathrm{A}}(TY_{\alpha}^{S^1}, 
\nabla^{TY_{\alpha}^{S^1}})
 \ch_g\left(\underline{L_{\alpha}^{1/2}} \right)
=\td_g\left(\nabla^{TY_{\alpha}^{S^1}},\nabla^{L_{\alpha}}\right). 
\end{align}

We call $F$ a trivial $S^1$-equivariant vector bundle over $Y$
if there is a finite dimensional $S^1$-representation
$M$ such that $F=Y\times M$ with the $S^1$-action on 
$F$ by $g(y,u)=(gy,gu)$.

Let $(\mu, h^{\mu})$ be an $S^1$-equivariant Hermitian vector bundle 
over $Y^{S^1}$ with an $S^1$-invariant Hermitian connection 
$\nabla^{\mu}$. 
Let $\iota:Y^{S^1}\rightarrow Y$ be the obvious embedding. 
In the following, we 
describe the geometric construction of Atiyah-Hirzebruch's direct image 
$\iota_!\mu\in K^0(Y)$ of $\mu$ for the embedding 
in $K$-theory \cite{AH59}, \cite[\S 1b]{BZ93}. It will be clear from 
its construction that it is compatible with the group action.	

For any $\delta>0$  
set $\mU_{\alpha,\delta}:=\{Z\in N_{\alpha}^{\R}: 
|Z|<\delta\}$. Then there exists $\var_0>0$ such that
the exponential map $(y,Z)\in N_{\alpha}^{\R}\rightarrow \exp_y^Y(Z)$
is a diffeomorphism between $\mU_{\alpha,2\var_0}$ and 
an open $S^1$-equivariant tubular neighbourhood of 
$Y^{S^1}_{\alpha}$ in $Y$ for any $\alpha$. Without confusion we 
will also regard $\mU_{\alpha,2\var_0}$ as this neighbourhood of 
$Y^{S^1}_{\alpha}$ in $Y$ via this identification. We choose $\var_0>0$ 
small enough such 
that for any $\alpha\neq \beta\in \mathfrak{B}$, 
$\ov{\mU_{\alpha,2\var_0}}\cap \ov{\mU_{\beta,2\var_0}} 
=\emptyset$.

Let $\pi_{\alpha}:N_{\alpha}\rightarrow 
Y^{S^1}_{\alpha}$ denote the projection of the normal bundle 
$N_{\alpha}$ over $Y^{S^1}_{\alpha}$.
For $Z\in N_{\alpha}^{\R}$, let $\tilde{c}(Z)\in 
\End(\Lambda^{\bullet}(N_{\alpha}^*))$ be the transpose of the 
canonical Clifford action $c(Z)$ on 
$\Lambda^{\bullet}(\ov{N}_{\alpha}^*)$ in \eqref{local211}. 
In particular, for $u\in N_{\alpha}$, let $\bar{u}^*\in N_{\alpha}^*$
be the metric dual of $\bar{u}\in \ov{N}_{\alpha}$, then
\begin{align}\label{eq:1.39b} 
\tilde{c}(u)=\sqrt{2}\,i_{u},\quad \tilde{c}(\bar{u})=
-\sqrt{2}\,\bar{u}^*\wedge.
\end{align}
Let $\pi_{\alpha}^*(\Lambda^{\bullet}(N_{\alpha}^*))$ be the pull back 
bundle of $\Lambda^{\bullet}(N_{\alpha}^*)$ over $N_{\alpha}$. 
For any $Z\in N_{\alpha}^{\R}$ with $Z\neq 0$, let $\tilde{c}(Z): 
\pi_{\alpha}^*(\Lambda^{\mathrm{even/odd}}(N_{\alpha}^*))|_Z
\rightarrow
\pi_{\alpha}^*(\Lambda^{\mathrm{odd/even}}(N_{\alpha}^*))|_Z$ 
denote the corresponding pull back isomorphism at $Z$.

As $S^1$ acts trivially on $Y^{S^1}_{\alpha}$, we can just 
apply  \cite[Chapter I, Corollary 9.9]{LaMi89} for each 
weight part to see that (cf. also \cite[Proposition 2.4]{Segal68})
there exists an $S^1$-equivariant vector bundle 
$F_{\alpha}$ over $Y^{S^1}_{\alpha}$ such that 
$(\Lambda^{\mathrm{even}}(N_{\alpha}^*)\otimes \mu_{\alpha})\oplus 
F_{\alpha}$ with $\mu_{\alpha}= \mu|_{Y^{S^1}_{\alpha}}$,
is a trivializable $S^1$-equivariant complex vector bundle 
over $Y^{S^1}_{\alpha}$ with the $S^1$-equivariant
trivialization map 
$\varphi_{\alpha}:Y_{\alpha}^{S^1}\times M_{\alpha}\rightarrow 
\Lambda^{\mathrm{even}}(N_{\alpha}^*)\otimes
\mu_{\alpha}\oplus F_{\alpha}$.
Then
\begin{align}\label{bl0034}
\sqrt{-1}\tilde{c}(Z)\oplus 
\pi_{\alpha}^*\mathrm{Id}_{F_{\alpha}}:
\pi_{\alpha}^*(\Lambda^{\mathrm{even}}(N_{\alpha}^*)\otimes
\mu_{\alpha}\oplus F_{\alpha})|_Z\rightarrow 
\pi_{\alpha}^*(\Lambda^{\mathrm{odd}}(N_{\alpha}^*)\otimes 
\mu_{\alpha}\oplus F_{\alpha})|_Z
\end{align}
induces an $S^1$-equivariant isomorphism between
two $S^1$-equivariant vector bundles over 
$\overline{\mU_{\alpha,2\var_0}}\backslash Y^{S^1}_{\alpha}$.
By adding trivial  $S^1$-equivariant vector bundles for the 
part $F_{\alpha}$, we 
can also assume that 
$M_{\alpha}=M_{\beta}=M$ for any $\alpha\neq \beta\in \mathfrak{B}$.
Now the identification $\pi_{\alpha}^*(\Lambda^{\mathrm{even}}
(N_{\alpha}^*)\otimes \mu_{\alpha}\oplus F_{\alpha})$
(resp. $\pi_{\alpha}^*(\Lambda^{\mathrm{odd}}
(N_{\alpha}^*)\otimes \mu_{\alpha}\oplus F_{\alpha})$)
with $(Y\setminus\cup_{\alpha}
	\mU_{\alpha,\var_0})
\times M$ on $\mU_{\alpha,2\var_0}\backslash 
\mU_{\alpha,\var_0}$ via the map $\varphi_{\alpha}$
(resp. $(\sqrt{-1}\tilde{c}(Z)\oplus 
\pi_{\alpha}^*\mathrm{Id}_{F_{\alpha}})\circ\varphi_{\alpha}$)
defines an $S^1$-equivariant vector bundle $\xi_+$ (resp. $\xi_-$)
over $Y$. Moreover the identity map of the above trivializations
of $\xi_+$ and $\xi_-$ over $Y\backslash \cup_{\alpha}
\mU_{\alpha,\var_0}$ extends smoothly the map (\ref{bl0034})
to a map $v:\xi_+\rightarrow \xi_-$.
Thus for each $\alpha\in \mathfrak{B}$, there exists an 
$S^1$-equivariant vector bundle $F_{\alpha}$ over
$Y_{\alpha}^{S^1}$ such that
\begin{align}\label{eq:1.40a}
\begin{split}
\xi_{\pm}|_{\mU_{\alpha,2\var_0}}&=
\pi_{\alpha}^*(\Lambda^{\mathrm{even/odd}}(N_{\alpha}^*)\otimes
\mu_{\alpha}\oplus F_{\alpha})|_{\mU_{\alpha,2\var_0}},
\\
v|_{\mU_{\alpha,2\var_0}}&=\sqrt{-1}\tilde{c}(Z)\oplus 
\pi_{\alpha}^*\mathrm{Id}_{F_{\alpha}},
\end{split}
\end{align}
and the restriction of $v$ to $Y\backslash \cup_\alpha
\mU_{\alpha,2\var_0}$ is invertible.
Then the direct image of $\mu$ by $\iota$ is given by
\begin{align}\label{local303} 
\iota_!\mu=\xi_{+}-\xi_-\in K^0(Y).
\end{align}

By a partition of unity argument, we get 
a metric $h^{\xi}=h^{\xi_+}\oplus h^{\xi_-}$ 
over $Y$ such that
\begin{align}\label{bl0035}
h^{\xi_{\pm}}|_{\mU_{\alpha,\var_0}}
=\left.\pi_{\alpha}^*\left(h^{\Lambda^{\mathrm{even/odd}}
(N_{\alpha}^*)\otimes	\mu_{\alpha}}\oplus 
h^{F_{\alpha}}\right)\right|_{\mU_{\alpha,\var_0}},
\end{align}
where $h^{\Lambda^{\mathrm{even/odd}}(N_{\alpha}^*)\otimes 
	\mu_{\alpha}}$ is the $S^1$-invariant Hermitian metric on 
$\Lambda^{\mathrm{even/odd}}(N_{\alpha}^*)\otimes \mu_{\alpha}$  
induced by  $h^{N}$ and $h^{\mu}$.
Again by a partition of unity argument, we get 
an $S^1$-invariant $\Z_2$-graded Hermitian connection 
$\nabla^{\xi}=\nabla^{\xi_+}\oplus \nabla^{\xi_-}$ on 
$\xi=\xi_+\oplus \xi_-$ over $Y$ such that
\begin{align}\label{bl0037}
\nabla^{\xi_{\pm}}|_{\mU_{\alpha,\var_0}}
=\left.\pi_{\alpha}^*\left(\nabla^{\Lambda^{\mathrm{even/odd}}
(N_{\alpha}^*)\otimes \mu_{\alpha}}\oplus
\nabla^{F_{\alpha}}\right)\right|_{\mU_{\alpha,\var_0}},
\end{align}
where $\nabla^{\Lambda^{\mathrm{even/odd}}(N_{\alpha}^*)\otimes 
	\mu_{\alpha}}$ is the 
	Hermitian connection on 
$\Lambda^{\mathrm{even/odd}}(N_{\alpha}^*)\otimes \mu_{\alpha}$  
induced by  $\nabla^{N}$ and $\nabla^{\mu}$.
We denote now
\begin{align}\label{local304} 
\underline{\xi_{\pm}}=(\xi_{\pm}, h^{\xi_{\pm}}, 
\nabla^{\xi_{\pm}})\  \text{over}\ Y.
\end{align} 

An equivariant extension of Bismut-Zhang embedding formula
\cite[Theorem 2.2]{BZ93} (cf.\ also \cite[Theorem 4.1]{DaiZ00} 
or \cite[Theorem 2.1]{FengXuZ09})
for $\eta$-invariants was proved in \cite[Corollaries 3.8, 3.9]{Liu2017}.
The following result follows from \cite[Corollaries 3.8, 3.9]{Liu2017}
applied for $G=S^1$, $g\in S^1\backslash A$.
\begin{thm}\label{local63}
There exists $V' \in R(S^1)$, such that for any $g\in S^1\backslash A$, 
	\begin{align}\label{bl1001}
	\ov{\eta}_g(\underline{TY}, \underline{L}, \underline{\xi_+})
	-\ov{\eta}_g(\underline{TY}, \underline{L}, \underline{\xi_-})
	=\sum_{\alpha}\ov{\eta}_g(\underline{TY^{S^1}_{\alpha}}, 
	\underline{L_{\alpha}}, \underline{\mu})
	+\chi_{V'}(g).
	\end{align}
\end{thm}

Remark that in \cite[Theorem 3.7]{Liu2017} applied in the case 
when the base space is a point, $V'$ is 
an equivariant spectral flow of a family of deformed Dirac 
operators on $Y$ with a pseudodifferential operator perturbation
obtained from the corresponding perturbation of the Dirac operator 
on $Y^{g}$.   
Since for any $g\in S^1\backslash A$, $Y^g=Y^{S^1}$ does not change,
thus $V'$ does not depend on $g\in S^1\backslash A$.  

\begin{rem}\label{t1.6}
Note that in the general setting of \cite[Theorem 3.7]{Liu2017} for the 
embedding $i:Y\rightarrow X$, there is an additional term,
the equivariant Bismut-Zhang current. 
Note that the equivariant Bismut-Zhang current is defined for 
the normal bundle of $Y^g$ in $X^g$. In 
our case, since $(Y^g)^g=Y^g$, this term is zero. 
\end{rem}

\subsection{Comparison of equivariant $\eta$-invariants}\label{s0205}

In this subsection, we review the comparison formula for equivariant 
$\eta$-invariants in \cite{LM2018}, which is an extension of the result 
of \cite{Goette00} and the analogue of the comparison formulas for
the holomorphic torsions \cite{BG00} and for the 
de Rham torsions \cite{BG04}.

For $K\in \mathrm{Lie}(S^1)$, let 
$K^Y(x)=\left.\frac{\partial}{\partial t}\right|_{t=0}e^{tK}\cdot x$ 
be the induced vector field on $Y$, and $\mL_K$ be 
the corresponding Lie derivative given
by $\mL_Ks=\left.\frac{\partial}{\partial t}
\right|_{t=0}\left(e^{-tK}.\, s\right)$
for $s\in \cC^{\infty}(Y,E)$ (cf. (\ref{eq:1.16b})). 
The associated moment maps are defined 
by \cite[Definition 7.5]{BeGeVe04},
\begin{align}\label{ct310} 
\begin{split}
&m^{E}(K):=\nabla^{E}_{K^Y}-\mL_K\vert_E\in 
\cC^{\infty}(Y,\End(E)),
\\
&m^{TY}(K)
:=\nabla^{TY}_{K^Y}-\mL_K\vert_{TY}=\nabla^{TY}_{\cdot}K^Y \in 
\cC^{\infty}(Y,\End(TY)), 
\end{split}
\end{align}
where the last equation holds 
since the Levi-Civita connection $\nabla^{TY}$ is torsion free.

Let $R_K^{E}$ and $R_K^{TY}$ be the equivariant curvatures 
of $E$ and $TY$ defined in \cite[\S 7.1]{BeGeVe04}:
\begin{align}\label{local79} 
R_K^{E}=R^{E}-2i\pi m^{E}(K), 
\quad R_K^{TY}=R^{TY}-2i\pi m^{TY}(K).
\end{align}

Observe that $m^{TY}(K)|_{Y^g}$ commutes with the circle action 
for any $g\in S^1$. Then it preserves 
the decompositions \eqref{e01047} and (\ref{eq:1.24b}). 
Let $m^{TY^g}(K)$, $m^N(K)$ and $m^{N_v}(K)$ 
be the restrictions of $m^{TY}(K)|_{Y^g}$ to $TY^g$, $N^{\R}$
and $N_v^{\R}$. Similarly, 
$m^{E}(K)|_{Y^g}$ preserves the decomposition \eqref{local266}. 
We define the corresponding equivariant curvatures $R_K^{TY^g}$ and
$R^{N^{\R}}_K$ as in \eqref{local79}. 
The following definition is an analogue of Definition \ref{local213} and 
\eqref{e01138}. 

\begin{defn}\label{local214}
For $g=e^{2\pi it}\in S^1$, $K\in \mathrm{Lie}(S^1)$, 
$|K|$ small enough, set
\begin{align}\label{local215}
\begin{split}
\widehat{\mathrm{A}}_{g,K}(TY,\nabla^{TY})
:&=\mathrm{det}^{1/2}\left(\frac{\frac{i}{4\pi}
R_K^{TY^g}}{\sinh \left(\frac{i}{4\pi}R_K^{TY^g}\right)}\right)
	\\
&\times
\left(i^{\frac{1}{2}\dim
N^{\R}}\mathrm{det}^{1/2}\left(1-
g\cdot \exp\left(\frac{i}{2\pi}R_K^{N^{\R}}\right)\right)\right)^{-1}
\in \Omega^{\bullet}(Y^g, \C),
\\
\ch_{g,K}(\underline{E}):&=\tr\left[g\exp\left(\frac{i}{2\pi } 
R_K^{E}	\right)\right]\in \Omega^{\bullet}(Y^g, \C).
\end{split}
\end{align}
Let $R_K^L$ be the corresponding equivariant 
curvature of $L$.
For $g\in S^1\backslash A$, as in (\ref{e01138}), we define
\begin{align}\label{eq:1.49b} 
\ch_{g,K}(\underline{L^{1/2}})
:&=\exp\left(\frac{i}{4\pi}R_K^{L}|_{Y^g}+i\pi l t\right).
\end{align}
For $g\in A$, as we discussed after (\ref{eq:1.24b}),
 we replace $i\pi lt$ by $\frac{i}{2}\theta'$ in (\ref{eq:1.49b}).
As in \eqref{bl0660}, we denote by
\begin{align} \label{eq:1.48}
\td_{g,K}(\nabla^{TY}, 
\nabla^{L}):=\widehat{\mathrm{A}}_{g,K}(TY,\nabla^{TY}) 
\ch_{g,K}(\underline{L^{1/2}}).
\end{align}	
Certainly for $K=0$, $\widehat{\mathrm{A}}_{g,K}(\cdot)
=\widehat{\mathrm{A}}_{g}(\cdot)$
and $\ch_{g,K}(\cdot)=\ch_g(\cdot)$.
\end{defn}

For $K\in \mathrm{Lie}(S^1)$  
set
\begin{align}\label{local83} 
d_K=d-2i\pi i_{K^Y}.
\end{align}
Then by \cite[Theorem 7.7]{BeGeVe04}, 
$\widehat{\mathrm{A}}_{g,K}(TY,\nabla^{TY})$, 
$\ch_{g,K}(\underline{E})$ and $\ch_{g,K}(\underline{L^{1/2}})$
are $d_K$-closed.

For $K\in \mathrm{Lie}(S^1)$ 
let $\vartheta_K\in T^*Y$ be the $1$-form which is dual to $K^Y$
by the metric $g^{TY}$, i.e., 
\begin{align}\label{eq:1.50}
\vartheta_K(X)=\la K^Y, X\ra\quad \text{for}\ X\in TY.
\end{align}
For $g\in S^1$, $K\in \mathrm{Lie}(S^1)$, $|K|$ small enough, set 
(cf. \cite[Definition 1.7]{B11a})
\begin{align}\label{local217} 
\mM_{g,K}(\underline{TY}, \underline{L}, \underline{E})
=-\int_0^{\infty}\left\{\int_{Y^g}
\frac{\vartheta_K}{2i\pi}\exp\left(\frac{v\,d_K
\vartheta_K}{2i\pi} 
\right)\td_{g,K}(\nabla^{TY},\nabla^L)
\ch_{g,K}(\underline{E})\right\}dv.
\end{align}
Note that if $g\in S^{1}\setminus A$ 
we have $Y^{g}=Y^{S^{1}}$ from \eqref{eq:1.6}, thus 
\begin{align}\label{eq:1.52b} 
\vartheta_K=0 \text{ on }Y^{g} 
\text{  and } \mM_{g,K}(\underline{TY}, \underline{L}, \underline{E})=0 
\text{ for } g\in S^{1}\setminus A.
\end{align}

By the argument of \cite[Proposition 2.2]{Goette09b}, 
$\mM_{g,K}(\underline{TY}, \underline{L},\underline{E})$ is well-defined 
for $|K|$ small enough. Moreover, for $K_0\in\mathrm{Lie}(S^1)$, 
$t\in \R$ and $|t|$ small enough, 
$\mM_{g,tK_0}(\underline{TY}, \underline{L}, \underline{E})$
is smooth at $t\neq0$ for $g\in A$ and there exist 
$c_j(K_0)\in \C$  ($j\in \N^{*})$ such that as $t\rightarrow 0$, we have
\begin{align}\label{local218} 
\mM_{g,tK_0}(\underline{TY}, \underline{L}, \underline{E})
=\sum_{j=1}^{(\dim 
	Y^g+1)/2}c_j(K_0)t^{-j}+\mathcal{O}(t^0).
\end{align}

In the following definition of the equivariant 
infinitesimal $\eta$-invariant, the operator 
$\sqrt{t}D^Y\otimes E+\frac{c(K^Y)}{4\sqrt{t}}$
was introduced by Bismut \cite{Bi85}
in his heat kernel proof of the Kirillov formula for the 
equivariant index. As observed by Bismut \cite[\S 1d)]{Bi86}
(see also \cite[\S 10.7]{BeGeVe04}),  its square plus 
$\mathcal{L}_{K^Y}$ is the square of  the Bismut 
superconnection for a fibration with compact structure group,
by replacing $K^Y$ by the curvature of the fibration, 
thus $\bar{\eta}_{g,K}$ in (\ref{eq:1.51}) should be understood as 
certain universal $\eta$-forms of Bismut-Cheeger 
\cite[Definition 4.33]{BC89}.

\begin{defn}\cite[Definition 2.3]{LM2018}\label{local88} 
For $g\in S^1$, $K\in\mathrm{Lie}(S^1)$ and $|K|$ small enough,
the equivariant infinitesimal (reduced) $\eta$-invariant is defined by
\begin{multline}\label{eq:1.51}
\bar{\eta}_{g,K}(\underline{TY}, \underline{L}, \underline{E})
=\int_{0}^{+\infty}\frac{1}{2\sqrt{\pi t}}
\tr\left[g\left(D^Y\otimes E-\frac{c(K^Y)}{4t} 
\right)\right.
\\
\left.\cdot\exp\left(-t\left(D^Y\otimes 
E+\frac{c(K^Y)}{4t}\right)^2-\mathcal{L}_{K^Y}
\right) \right]dt+\frac{1}{2}\tr|_{\Ker
	(D^Y\otimes E)}[ge^K]\in \C .
\end{multline}
\end{defn}

From \eqref{local59}, \eqref{local217}  
and \eqref{eq:1.51}, we know that 
$\bar{\eta}_{g,0}(\cdot)= \bar{\eta}_{g}(\cdot)$,
$\mM_{g,0}(\cdot)=0$.

The following two theorems are special cases of 
\cite[Theorems 0.1, 0.2]{LM2018}, which extend 
Goette's result \cite[Theorem 0.5]{Goette00} as an equality
of formal Laurent series in $t$ at $t=0$ when $g=1$ and 
$K_0^Y$ does not vanish on $Y$.
Here the equivariant $\eta$-forms are just equivariant 
$\eta$-invariants and the compact Lie group is $S^1$.

\begin{thm}\label{local90} 
Fix $K_0\in\mathrm{Lie}(S^1)$, $g\in S^{1}$. 
There exists $\beta>0$ such that for $t\in \R$ and $|t|<\beta$, the 
equivariant infinitesimal $\eta$-invariant 
$\bar{\eta}_{g,tK_0}(\underline{TY}, \underline{L}, \underline{E})$
is well-defined and is an analytic function of $t$.
Furthermore, as a function of $t$ near $0$, 
$t^{(\dim 	Y^g+1)/2}\mM_{g,tK_0}(\underline{TY}, \underline{L}, 
\underline{E})$ is real analytic.
\end{thm}

\begin{thm}\label{local96} Fix $0\neq K_0\in\mathrm{Lie}(S^1)$. 
For any $g\in S^1$, there exists $\beta>0$ such that 
	for $|t|<\beta, t\neq 0$, we have
	\begin{align}\label{local97} 
	\bar{\eta}_{g,tK_0}(\underline{TY}, \underline{L}, \underline{E})
	=\bar{\eta}_{ge^{tK_0}}(\underline{TY}, 
	\underline{L}, \underline{E})
	+\mM_{g,tK_0}(\underline{TY}, \underline{L}, \underline{E}).
	\end{align}
\end{thm}

Since 
$\bar{\eta}_{g,tK_0}(\underline{TY}, \underline{L}, \underline{E})$
is an analytic function of $t$, when 
$t\rightarrow 0$, the singularity of 
$\bar{\eta}_{ge^{tK_0}}(\underline{TY}, \underline{L}, \underline{E})$
is the same as that of 
$-\mM_{g,tK_0}(\underline{TY}, \underline{L}, \underline{E})$
in \eqref{local218}. Thus from Theorem \ref{local96}, we know 
$\bar{\eta}_{g}(\underline{TY}, \underline{L}, \underline{E})$ 
as a function of $g\in S^1$, is analytic on $S^1\backslash A$,
moreover, at $g\in A$, $\bar{\eta}_{ge^{tK_0}}(\underline{TY}, 
\underline{L}, \underline{E})
+\mM_{g,tK_0}(\underline{TY}, \underline{L}, \underline{E})$ 
on $0<t<|\beta|$ can be extended
as an analytic function on $|t|<|\beta|$.

\section{Differential $K$-theory}\label{s03}

$K$-theory and the $\lambda$-ring structure were first introduced by 
Grothendieck in 1957. The arithmetic $K$-theory in Arakelov
geometry was introduced by Gillet-Soul\'e in \cite{GS90c}.
It extends Grothendieck's $K$-theory by adding
Hermitian metrics on holomorphic vector bundles and 
differential forms of type $(p,p)$ modulo $\mathrm{Im}\,\partial
+\mathrm{Im}\,\bar{\partial}$.  
In the same way as the topological $K$-theory
of Atiyah and Hirzebruch is the $\cC^{\infty}$-version of Grothendieck's
$K$-theory, also the differential $K$-theory
introduced  by Freed-Hopkins \cite{FH00} 
and developed further by Hopkins-Singer,
Simons-Sullivan, Bunke-Schick, Freed-Lott, is a 
$\cC^{\infty}$-version of the arithmetic $K$-theory.
It extends the topological $K$-theory by adding Hermitian
metrics and connections on vector bundles and differential forms
modulo exact forms.

Note that the $\lambda$-ring structure on the arithmetic $K$-theory
was introduced by Gillet-Soul\'e \cite[Theorem 7.3.4]{GS90c}
and exploited in detail by Roessler \cite{Ro99,Ro01}, who
studied also the associated $\gamma$-filtration.

In this section, we start to exploit the $\lambda$-ring structure on the 
vector space of even degree real closed forms on $Y$ 
and on its direct sum with the space of odd degree
real forms modulo exact forms. Then we study
its compatibility with the 
Chern forms and Chern-Simons classes of geometric triples.
With this preparation, we can equip 
the differential $K$-theory with a pre-$\lambda$-ring structure. 
An important result is that the associated $\gamma$-filtration is locally 
nilpotent.

We consider the circle action on $Y$ now. Recall that $N$ is the 
normal bundle of $Y^{S^1}$, the fixed point set of
the circle action, in $Y$. When we apply the above
results to our $g$-equivariant differential $K$-theory
$\widehat{K}^0_g(Y^{S^1})$, it implies that
$\underline{\lambda_{-1}(N^*)}:=\sum_{i\geq 0}(-1)^i
\underline{\Lambda^i(N^*)}$,
the exterior algebra bundle of $N^*$ with corresponding
metric and connection, is invertible in 
$\widehat{K}^0_g(Y^{S^1})_{I(g)}$, the localization
of $\widehat{K}^0_g(Y^{S^1})$ at the prime ideal $I(g)$
of $R(S^1)$. This result allows us to define the 
counterpart of the $\eta$-invariant on the fixed point set.

This section is organized as follows.
In Section \ref{s0301}, we define 
the pre-$\lambda$-ring structure and study some examples. 
In Section \ref{s0302}, we construct the 
pre-$\lambda$-ring structure in differential $K$-theory. 
In Section \ref{s0303}, we 
study the locally nilpotent property of the $\gamma$-filtration
in differential $K$-theory. In Section 
\ref{s0304}, we define the $g$-equivariant differential $K$-theory and 
explicitly construct the inverse of 
$\underline{\lambda_{-1}(N^*)}$ 
at differential $K$-theory level.

\subsection{Pre-$\lambda$-ring structure}\label{s0301}

\begin{defn}\label{defn:2.01}\cite[(1.1)-(1.3)]{BerthelotTh}
For a commutative ring $R$ with identity,
a pre-$\lambda$-ring structure 
is defined by a countable set of maps $\lambda^n:R\rightarrow R$
with $n\in \N$ 	such that for all $x,y\in R$,
	
a) $\lambda^0(x)=1$;
	
b) $\lambda^1(x)=x$;
	
c) $\lambda^n(x+y)=\sum_{j=0}^n\lambda^j(x)\lambda^{n-j}(y)$.
	
If $R$ has a pre-$\lambda$-ring structure, 
we call it a pre-$\lambda$-ring.
\end{defn}

Remark that in \cite[\S 1]{ATall69} 
the pre-$\lambda$-ring here is called the $\lambda$-ring.

If $t$ is an indeterminate,  we define for $x\in R$,
\begin{align}\label{eq:2.001}
\lambda_t(x)=\sum_{n\geq 0}\lambda^n(x)t^n.
\end{align}
Then the relations a), c) show that $\lambda_t$ is 
a homomorphism from 
the additive group of $R$ into the multiplicative group $1+R[[t]]^+$, of 
formal power series in $t$ with constant term 1, i.e.,
\begin{align}\label{eq:2.002}
\lambda_t(x+y)=\lambda_t(x)\lambda_t(y)\quad \text{for any}\ x,y\in R.
\end{align}

Now we study some pre-$\lambda$-rings which we will use later.

Let $Y$ be a manifold. Let $Z^{\mathrm{even}}(Y,\R)$ be the
vector space of even degree real closed forms 
 on $Y$.
We define the Adams operation 
$\Psi^k:Z^{\mathrm{even}}(Y,\R)\rightarrow 
Z^{\mathrm{even}}(Y,\R)$ for $k\in \N$ by
\begin{align}\label{eq:2.003}
\Psi^k(x)=k^lx \quad\text{for}\quad x\in Z^{2l}(Y,\R).
\end{align}	 
For $x\in Z^{\mathrm{even}}(Y,\R)$, we define
\begin{align}\label{eq:2.004}
\lambda_t(x)=\sum_{n\geq 
	0}\lambda^n(x)t^n:=
	\exp\left(\sum_{k=1}^{\infty}\frac{(-1)^{k-1}\Psi^k(x)t^k}{k}\right).
\end{align}
From the Taylor expansion of the exponential function,
we have 
$\lambda^0(x)=1$ and $\lambda^1(x)=x$.
Since
\begin{multline}\label{eq:2.005}
\lambda_t(x+y)=\exp\left(\sum_{k=1}^{\infty}\frac{(-1)^{k-1}
\Psi^k(x+y)t^k}{k}\right)
\\
=\exp\left(\sum_{k=1}^{\infty}\frac{(-1)^{k-1}\Psi^k(x)t^k}{k}
+\sum_{k=1}^{\infty}\frac{(-1)^{k-1}\Psi^k(y)t^k}{k}\right)
=\lambda_t(x)\lambda_t(y),
\end{multline}
we have
\begin{align}\label{eq:2.006}
\lambda^n(x+y)=\sum_{j=0}^n\lambda^j(x)\lambda^{n-j}(y).	
\end{align}
Thus \eqref{eq:2.004} gives a pre-$\lambda$-ring structure on 
$Z^{\mathrm{even}}(Y,\R)$.

Consider the vector space (comparing with
 \cite[\S 7.3.1]{GS90c})
\begin{align}\label{eq:2.007}
\Gamma(Y):=Z^{\mathrm{even}}(Y,\R)\oplus 
\left(\Omega^{\mathrm{odd}}(Y,\R)/\Im \, d\right).
\end{align}
We give degree $l\geq 0$ to 
$Z^{2l}(Y,\R)\oplus \left(\Omega^{2l-1}(Y,\R)/\Im \, d\right)$
with $\Omega^{-1}(\cdot)=\{0\}$. 
We define a pairing on $\Gamma(Y)$ by the formula
\begin{align}\label{eq:2.008}
(\omega_1,\phi_1)*(\omega_2, \phi_2)
:=(\omega_1\wedge\omega_2, 
\omega_1\wedge\phi_2+\phi_1\wedge\omega_2
-d\phi_1\wedge\phi_2).
\end{align}
It is easy to verify that this pairing is 
commutative and associative. Since 
$*$ is clearly bilinear 
and $(1,0)$ is a unit, the pairing $*$ defines a graded associative, 
commutative and unitary $\R$-algebra structure on $\Gamma(Y)$
(comparing with \cite[Theorem 7.3.2]{GS90c}).
We define the Adams operation 
$\Psi^k:\Gamma(Y)\rightarrow \Gamma(Y)$ 
for $k\in \N$ (cf. \cite[\S 7.3.1]{GS90c}, \cite{Ro01}) by
\begin{align}\label{eq:2.009}
\Psi^k(\alpha,\beta)=(k^l\alpha,k^l\beta)\quad\text{for}\quad 
(\alpha,\beta)\in Z^{2l}(Y,\R)\oplus (\Omega^{2l-1}(Y,\R)/\Im \, d).
\end{align}
By using the pairing $*$ to replace the multiplicity in \eqref{eq:2.004},
similarly as $Z^{\mathrm{even}}(Y,\R)$,
we obtain a pre-$\lambda$-ring structure on $\Gamma(Y)$.

Let $p$ be the projection from $\Gamma(Y)$ to its component 
$Z^{\mathrm{even}}(Y,\R)$
and $\jmath$ be the following injection:
\begin{align}\label{eq:2.010} \begin{split}
p:\Gamma(Y)\rightarrow Z^{\mathrm{even}}(Y,\R),  
\quad (\omega,\varphi)\mapsto \omega,\\
\jmath: Z^{\mathrm{even}}(Y,\R)\rightarrow \Gamma(Y),
\quad \omega\mapsto (\omega,0).
\end{split}\end{align}
By \eqref{eq:2.008}, $p,\jmath$ are homomorphisms of 
pre-$\lambda$ rings, in particular, 
\begin{align}\label{eq:2.011} 
	\lambda^{k}(\omega,0)= (\lambda ^{k}\omega, 0).
\end{align}

Let $G$ be a Lie group and $\kg$ its Lie algebra. A polynomial 
$\varphi:\kg\rightarrow \C$
is called a $G$-invariant polynomial if 
\begin{align}\label{eq:2.012}
\varphi(\mathrm{Ad}(g^{-1})A)
=\varphi(A),\quad \text{ for any } g\in G,\ A\in \kg.
\end{align}
The set of all $G$-invariant polynomials is denoted by 
$\C[\kg]^G$. 

Let $U(r)$ be the unitary group with Lie algebra $\mathfrak{u}(r)$. 
For $A\in \mathfrak{u}(r)$, the characteristic polynomial of $-A$ is
\begin{align}\label{eq:2.013}
\det(\lambda I+A)=\lambda^r+c_1(A)\lambda^{r-1}+\cdots+c_r(A).
\end{align}
So $c_j\in \C[\mathfrak{u}(r)]^{U(r)}$ for $1\leq j\leq r$. 
It is well-known that $\C[\mathfrak{u}(r)]^{U(r)}$ is 
generated by $c_1,\cdots,c_r$
as a polynomial ring:
\begin{align}\label{eq:2.014}
\C[\mathfrak{u}(r)]^{U(r)}=\C[c_1,\cdots,c_r].
\end{align}
Let $T^r=\{(e^{it_1},\cdots, e^{it_r}): 
t_1,\cdots,t_r\in \R \}$ be a maximal torus of $U(r)$ with Lie 
algebra $\mathfrak{t}^r$. 
Then
\begin{align}\label{eq:2.015}
\C[\mathfrak{t}^r]^{T^r}=\C[u_1,\cdots,u_r],
\end{align}
where $u_j(x)=x_j$, for any 
$x=\sum_{j=1}^rx_j\left.\frac{\partial}{\partial t_j}\right|_{t=0}
\in T_e(T^r)=\mathfrak{t}^r$.
Let 
\begin{align}\label{eq:2.016} 
\theta:T^r\rightarrow U(r),\quad (e^{it_1},\cdots, e^{it_r})\mapsto 
\mathrm{diag}(e^{it_1},\cdots, e^{it_r})
\end{align}
be the 
diagonal injection, here $\mathrm{diag}(\cdots)$
is the diagonal matrix. It is well-known that 
\begin{align}\label{eq:2.017}
\theta^*: 
\C[\mathfrak{u}(r)]^{U(r)}\rightarrow \C[\mathfrak{t}^r]^{T^r}
\end{align}
is an injective homomorphism and 
\begin{align}\label{eq:2.018} 
\theta^*(c_j)=\sigma_j(u_1,\cdots,u_r), \quad 
\theta^*\left(\C[\mathfrak{u}(r)]^{U(r)} \right)
=\C[\sigma_1,\cdots,\sigma_r],
\end{align}
where $\sigma_j$ 
is the $j$-th elementary symmetric polynomial. 
We define the Adams operations for any $k\in \N$
\begin{align}\label{eq:2.020}\begin{split}
&\Psi^k:\C[\mathfrak{u}(r)]^{U(r)}\rightarrow 
\C[\mathfrak{u}(r)]^{U(r)},  \quad \Psi^k(c_j)=k^jc_j;\\
&\Psi^k:\C[\mathfrak{t}^r]^{T^r}\rightarrow 
\C[\mathfrak{t}^r]^{T^r}, \qquad\qquad \, \Psi^k(u_j)=k u_j.
\end{split}\end{align}
By constructing $\lambda^n$ as in \eqref{eq:2.004}, 
$\C[\mathfrak{u}(r)]^{U(r)}$ and $\C[\mathfrak{t}^r]^{T^r}$ 
are equipped now pre-$\lambda$-ring structures.

Let $E$ be a complex vector bundle over $Y$ of rank $r$. Let $h^E$ be 
a Hermitian metric 
on $E$. Let $\nabla^E$ be a Hermitian connection on $(E,h^E)$. We also 
denote by $\underline{E}=(E,h^E,\nabla^E)$ the {\it geometric triple} 
for this non-equivariant setting. 
For $\varphi\in \C[\mathfrak{u}(r)]^{U(r)}$, we define the 
characteristic form $\varphi(\underline{E})$ by
\begin{align}\label{eq:2.021}
\varphi(\underline{E})=\psi_Y\varphi\left(-R^E\right)\in 
\Omega^{\mathrm{even}}(Y,\C),
\end{align}
where $\psi_Y:\Omega^{\mathrm{even}}(Y,\C)\rightarrow 
\Omega^{\mathrm{even}}(Y,\C)$ is defined by 
\begin{align}\label{eq:2.022}
\psi_Y\omega=(2i\pi)^{-j}\omega\quad \text{for}\ \omega\in 
\Omega^{2j}(Y,\C).
\end{align}
Then by the Chern-Weil theory (cf. 
\cite[Appendix D]{MM07}), $\varphi(\underline{E})$
is closed. Moreover, $\varphi(\underline{E})\in
\Omega^{\mathrm{even}}(Y,\R)$ if $\varphi\in 
\R[\mathfrak{u}(r)]^{U(r)}$.

The triple $\underline{E}$ induces a homomorphism of rings
\begin{align}\label{eq:2.023} 
f_{\underline{E}}: \C[\mathfrak{u}(r)]^{U(r)}\rightarrow
Z^{\mathrm{even}}(Y,\C),
\quad \varphi\mapsto \varphi(\underline{E}).
\end{align}

Let $\underline{\pi^*E}=(\pi^*E, h^{\pi^*E}, \nabla^{\pi^*E})$ be the 
triple defined in Section \ref{s0202} without the group action. 
As in \eqref{local40}, the Chern-Simons class 
$\widetilde{\varphi}(\underline{E_0},\underline{E_1})\in 
\Omega^{\mathrm{odd}}(Y,\C)/\Im \, d$ is defined by 
(cf. \cite[Definition B.5.3]{MM07})
\begin{align}\label{eq:2.024}
\widetilde{\varphi}(\underline{E_0},\underline{E_1}):=
\int_0^1\{\varphi(\underline{\pi^*E})\}^{ds}ds\in
\Omega^{\mathrm{odd}}(Y,\C)/\Im \, d.
\end{align}
Then by \cite[Theorem B.5.4]{MM07},
\begin{align}\label{eq:2.025} 
d\widetilde{\varphi}(\underline{E_0},\underline{E_1})=
\varphi(\underline{E_1})-\varphi(\underline{E_0}),
\end{align}
and the Chern-Simons class depends only on $\nabla^{E_0}$ and 
$\nabla^{E_1}$.

\begin{lemma}\label{lem:2.02}
	Let $\underline{E_j}=(E, h_j^{E}, \nabla_j^{E})$ for $j=0,1,2$. Let 
	$\varphi,\varphi'\in \C[\mathfrak{u}(r)]^{U(r)}$. Then we have
\vskip 2 mm		
	(a) 
	$\widetilde{\varphi}(\underline{E_0},\underline{E_2})
	=\widetilde{\varphi}(\underline{E_0},\underline{E_1})
	+\widetilde{\varphi}(\underline{E_1},\underline{E_2})$;
\vskip 2 mm	
	(b) 
	$\widetilde{\varphi+\varphi'}(\underline{E_0},\underline{E_1})
	=\widetilde{\varphi}(\underline{E_0},\underline{E_1})
	+\widetilde{\varphi'}(\underline{E_0},\underline{E_1})$;
\vskip 2 mm		
	(c) 
	$\widetilde{\varphi\varphi'}(\underline{E_0},\underline{E_1})
=\widetilde{\varphi}(\underline{E_0},\underline{E_1})
\varphi'(\underline{E_1})
+\varphi(\underline{E_0})\widetilde{\varphi'}(\underline{E_0},
\underline{E_1})$;
	\vskip 2 mm	
	(d) $(\varphi(\underline{E_1}), 
	\widetilde{\varphi}(\underline{E_0},\underline{E_1}))*
	(\varphi'(\underline{E_1}),
	\widetilde{\varphi'}(\underline{E_0},\underline{E_1}))
	=(\varphi\varphi'(\underline{E_1}),
	\widetilde{\varphi\varphi'}(\underline{E_0},\underline{E_1}))$.
\end{lemma}
\begin{proof}
	By \eqref{eq:2.024}, (a) and (b) are obvious.
	Let $\pi^*\underline{E_1}$ be the pull-back
	of the triple $\underline{E_1}$ on $Y\times\R$.
	Then $\varphi'(\pi^*\underline{E_1}) 
	=\pi^*\varphi'(\underline{E_1})$. Thus by
	(\ref{eq:2.024}),
	\begin{multline}\label{eq:2.25a}
	\wi{\varphi\varphi'}(\underline{E_0},
	\underline{E_1})-\wi{\varphi}(\underline{E_0},
	\underline{E_1})\varphi'(\underline{E_1})
	=\int_0^1\Big\{\varphi(\underline{\pi^*E})
	\left(\varphi'(\underline{\pi^*E})-
	\varphi'(\pi^*\underline{E_1}) \right) 
	\Big\}^{ds}ds	\\ 
	=\int_0^1\left\{d^{Y\times \R}
	\left[\varphi(\underline{\pi^*E})
	\wi{\varphi'}(\pi^*\underline{E_1},
	\underline{\pi^*E}) \right] \right\}^{ds}ds
	\\
	=-d^Y\int_0^1\left\{\varphi(\underline{\pi^*E})
	\wi{\varphi'}(\pi^*\underline{E_1},
	\underline{\pi^*E}) \right\}^{ds}ds+
	\left.\left.\varphi(\underline{\pi^*E})
	\wi{\varphi'}(\pi^*\underline{E_1},
	\underline{\pi^*E})\right|_{Y\times\{t\}}
	\right|_{t=0}^1.
	\end{multline}
	From (\ref{eq:2.25a}), we get (c).
	
	We establish (d) now.
	From \eqref{eq:2.008}, \eqref{eq:2.025} and (c), we have
	\begin{multline}\label{eq:2.026}
	(\varphi(\underline{E_1}), 
	\widetilde{\varphi}(\underline{E_0},\underline{E_1}))*
	(\varphi'(\underline{E_1}),
	\widetilde{\varphi'}(\underline{E_0},\underline{E_1}))
	\\
	=\Big(\varphi(\underline{E_1})\varphi'(\underline{E_1}) , 
	\varphi(\underline{E_1})\widetilde{\varphi'}(\underline{E_0},
	\underline{E_1})+\widetilde{\varphi}(\underline{E_0},\underline{E_1})
	\varphi'(\underline{E_1})-(\varphi(\underline{E_1})
	-\varphi(\underline{E_0}))\widetilde{\varphi'}(\underline{E_0},
	\underline{E_1})\Big)
	\\
	=(\varphi\varphi'(\underline{E_1}), 
	\widetilde{\varphi\varphi'}(\underline{E_0},\underline{E_1})).
	\end{multline}
	
	The proof of Lemma \ref{lem:2.02} is completed.
\end{proof}

As in \eqref{eq:2.023}, the triple $\underline{\pi^*E}$ 
induces a map
\begin{align}\label{eq:2.027} 
\tilde{f}_{\underline{E}}: \C[\mathfrak{u}(r)]^{U(r)}
\rightarrow \Gamma(Y),
\quad \varphi\mapsto (\varphi(\underline{E_1}), 
\widetilde{\varphi}(\underline{E_0},\underline{E_1})).
\end{align}
It is a ring homomorphism by Lemma \ref{lem:2.02}.

\begin{lemma}\label{lem:2.03} 
	The ring homomorphisms $p$, $\theta^*$, $f_{\underline{E}}$ 
	and $\tilde{f}_{\underline{E}}$ in the following diagram are all 
	homomorphisms of pre-$\lambda$-rings, 
	\begin{center}
		\begin{tikzpicture}[>=angle 90]
		\matrix(a)[matrix of math nodes,
		row sep=3em, column sep=3em,
		text height=1.5ex, text depth=0.25ex]
		{\Gamma(Y) & \C[\mathfrak{u}(r)]^{U(r)} & 
		\C[\mathfrak{t}^r]^{T^r}\\
		Z^{\mathrm{even}}(Y,\R).	& & \\};
		\path[<-](a-1-1) edge 
		node[above]{$\tilde{f}_{\underline{E}}$} 
		(a-1-2);
		\path[->](a-1-2) edge 
		node[below]{$f_{\underline{E_1}}$} 
		(a-2-1);
		\path[->](a-1-1) edge node[left]{$p$} 
		(a-2-1);
		\path[->](a-1-2) edge node[above]{$\theta^*$} 
		(a-1-3);
		\end{tikzpicture}
	\end{center}
Moreover, we have 
\begin{align}\label{eq:2.028} 
f_{\underline{E_1}}=p\circ \tilde{f}_{\underline{E}}.
\end{align}
\end{lemma}
\begin{proof}
From \eqref{eq:2.003}, \eqref{eq:2.009}, \eqref{eq:2.010},
\eqref{eq:2.017}, \eqref{eq:2.020}, 
\eqref{eq:2.023} and \eqref{eq:2.027}, 
we see that all homomorphisms here commute with the corresponding
Adams operations. So by \eqref{eq:2.004}, they are all homomorphisms 
of pre-$\lambda$-rings. 

The relation \eqref{eq:2.028} follows directly from \eqref{eq:2.023} 
and \eqref{eq:2.027}.

The proof of Lemma \ref{lem:2.03} is completed. 
\end{proof}

By \eqref{eq:2.018}, $(\theta^*)^{-1}(H)\in \C[\mathfrak{u}(r)]^{U(r)}$ 
is well-defined for any homogeneous symmetric polynomial $H$ 
in $u_1,\cdots,u_r$.
We define the Chern character to be the formal power series
\begin{align}\label{eq:2.029} 
\ch=\sum_{k=0}^{\infty}(\theta^*)^{-1}
\Big(\frac{1}{k!}\sum_{j=1}^ru_j^k\Big).
\end{align}
It is easy to see that 
\begin{align}\label{eq:2.030} 
f_{\underline{E}}(\ch)
:=\sum_{k=0}^{\infty}f_{\underline{E}}\circ(\theta^*)^{-1}
\Big(\frac{1}{k!}\sum_{j=1}^ru_j^k\Big)
\end{align}
is the same as the canonical Chern character $\ch(\underline{E})$
in Definition \ref{local213} for $g=1$.

Since the manifold is finite dimensional, 
the right-hand side of \eqref{eq:2.030} is a finite sum. 
In this paper, we only care about the characteristic forms. 
We could apply $\theta^*$ and $\lambda^i$ on $\ch$ formally
and obtain the rigorous equality of the characteristic forms after
taking the map $f_{\underline{E}}$.

From this point of view  
we write
\begin{align}\label{eq:2.031} 
\ch=(\theta^*)^{-1}\Big(\sum_{j=1}^r\exp(u_j)\Big).
\end{align}
From Lemma \ref{lem:2.03},
\eqref{eq:2.004}, \eqref{eq:2.020} and \eqref{eq:2.031}, we have
\begin{multline}\label{eq:2.032} 
\theta^*\circ\lambda_t(\ch)
=\lambda_t\Big(\sum_{j=1}^r\exp(u_j)\Big)
=\exp\left(\sum_{k=1}^{\infty}\frac{1}{k}(-1)^{k-1}t^k\Psi^k
\Big(\sum_{j=1}^r\exp(u_j)\Big)\right)
\\
=\exp\Big(\sum_{j=1}^r\sum_{k=1}^{\infty}
\frac{1}{k}(-1)^{k-1}t^k\exp(k 	u_j)\Big)
=\exp\Big(\sum_{j=1}^r\log(1+t\exp(u_j))\Big)\\
=\prod_{j=1}^r(1+t\exp(u_j)).
\end{multline}	

From \eqref{eq:2.032}, we get the following equality 
(comparing with \cite[Lemma 7.3.3]{GS90c}),
\begin{align}\label{eq:2.033} 
\lambda^k(\ch)(\underline{E})=\ch(\Lambda^k(\underline{E}))\in 
Z^{\mathrm{even}}(Y,\R).
\end{align}

\begin{lemma}\label{lem:2.04} 
The following identity holds,
\begin{align}\label{eq:2.034} 
\lambda^k\Big(\ch(\underline{E_1}),\wi{\ch}(\underline{E_0},
\underline{E_1})\Big)
=\Big(\ch(\Lambda^k (\underline{E_1})), \wi{\ch}(\Lambda^k 
(\underline{E_0}), \Lambda^k (\underline{E_1}))\Big)\in \Gamma(Y).
\end{align}
\end{lemma}
\begin{proof} 
From \eqref{eq:2.024} and \eqref{eq:2.033},
we have modulo exact forms,
\begin{multline}\label{eq:2.035} 
\widetilde{\lambda^k(\ch)}(\underline{E_0},\underline{E_1})
=\int_0^1\{\lambda^k(\ch)(\underline{\pi^*E})\}^{ds}ds
\\
=\int_0^1\{\ch(\Lambda^k(\underline{\pi^*E}))\}^{ds}ds
=\widetilde{\ch}(\Lambda^k(\underline{E_0}),
\Lambda^k(\underline{E_1})).
\end{multline}
So from Lemma \ref{lem:2.03}, \eqref{eq:2.027}, \eqref{eq:2.033} 
and \eqref{eq:2.035}, we get
\begin{multline}\label{eq:2.036}
\lambda^k\big(\ch(\underline{E_1}),\wi{\ch}(\underline{E_0},
\underline{E_1})\big)
=\lambda^k\big(\tilde{f}_{\underline{E}}(\ch) \big)
=\tilde{f}_{\underline{E}}(\lambda^k(\ch))
\\
=\Big(\lambda^k(\ch)(\underline{E_1}), \widetilde{\lambda^k(\ch)}
(\underline{E_0},\underline{E_1})\Big)
=\Big(\ch(\Lambda^k (\underline{E_1})), \wi{\ch}(\Lambda^k 
(\underline{E_0}), \Lambda^k (\underline{E_1}))\Big).
\end{multline}

The proof of Lemma \ref{lem:2.04}  is completed.
\end{proof}

\subsection{Pre-$\lambda$-ring structure in differential $K$-theory}
\label{s0302}
We introduce now a pre-$\lambda$-ring structure for differential $K$-ring.
It can be understood as the differential $K$-theory 
version of the pre-$\lambda$-ring structure for arithmetic $K$-theory 
in \cite[Theorem 7.3.4]{GS90c}.

Let $Y$ be a compact manifold.

\begin{defn}\label{defn:2.05}\cite[Definition 2.16]{FreedLott10}
	A cycle for differential $K$-theory of $Y$ is a 
	pair $(\underline{E}, \phi)$ where 
	$\underline{E}$ is a geometric triple and $\phi$ is an element in
	$\Omega^{\mathrm{odd}}(Y, \R)/\Im \, d$.
	Two cycles $(\underline{E_1}, \phi_1)$ and $(\underline{E_2}, 
	\phi_2)$ are equivalent if there exist a geometric triple 
	$\underline{E_3}$ and a vector bundle isomorphism
	\begin{align}\label{eq:2.037} 
	\Phi:E_1\oplus E_3\rightarrow E_2\oplus E_3
	\end{align}
	such that
	\begin{align}\label{eq:2.038}
	\widetilde{\ch}\left(\underline{E_1}\oplus\underline{E_3}, 
	\Phi^{*} 
	\left(\underline{E_2}\oplus\underline{E_3}\right)\right)
	=\phi_2-\phi_1.
	\end{align} 
	We define the sum in the obvious way by
	\begin{align}\label{eq:2.039}
	(\underline{E}, \phi)+(\underline{F}, 
	\psi)=(\underline{E}\oplus 
	\underline{F}, \phi+\psi).
	\end{align} 	
	The differential $K$-group 
	$\widehat{K}^0(Y)$ is defined as the Grothendieck group of 
	equivalence classes of cycles. 
	\end{defn}

We denote by $[\underline{E}, \phi]$
the equivalence class of a cycle $(\underline{E}, \phi)$. Then
\begin{align}\label{eq:2.38a}
\widehat{K}^0(Y)=\left\{[\underline{E}-\underline{E_1}, \phi-\phi_1]:
(\underline{E}, \phi), (\underline{E_1}, \phi_1) \text{ are
	cycles as above} \right\}
\end{align}
and $\widehat{K}^0(Y)$ is an abelian group.
	For $[\underline{E}, \phi], 
	[\underline{F}, \psi]\in \widehat{K}^0(Y)$,
set
\begin{align}\label{eq:2.040}
[\underline{E}, \phi]\cup[\underline{F}, \psi]
=\Big[\underline{E}\otimes \underline{F}, 
[(\ch(\underline{E}),\phi)*(\ch(\underline{F}),\psi)]_{\mathrm{odd}}
\Big],
\end{align} 
where $[\cdot]_{\mathrm{odd}}$ is the 
component of $\Gamma(Y)$ in 
$\Omega^{\mathrm{odd}}(Y,\R)/\Im \, d$.	
It is easy to check that this product \eqref{eq:2.040} is well-defined, 
commutative and associative (cf. also Lemma
\ref{lemma:2.15a}). 
Let $\underline{\C}$ be the trivial complex line bundle over $Y$
with the trivial metric and connection.
Then the 
element 
\begin{align}\label{eq:2.041} 
1:=[\underline{\C},0]
\end{align}
is a unit for the product $\cup$. 
Thus $(\widehat{K}^0(Y), +, \cup)$ is a commutative ring
with unit $1$.

From \eqref{local126}, \eqref{eq:2.037} and \eqref{eq:2.038}, we see that 
if $[\underline{E},0]=[\underline{F},0]\in \widehat{K}^0(Y)$, we have
\begin{align}\label{eq:2.042} 
\ch(\underline{E})=\ch(\underline{F})\in \Omega^{\bullet}(Y,\R).
\end{align}

\begin{thm}\label{thm:2.06}
	There exists a pre-$\lambda$-ring structure on $\widehat{K}^0(Y)$.
\end{thm} 
\begin{proof}
	Let $(\underline{E},\phi)$ be a cycle.
	Observe that $(\ch(\underline{E}),\phi)\in \Gamma(Y)$. 
	Since $\Gamma(Y)$ is a pre-$\lambda$-ring,
	we define
	\begin{align}\label{eq:2.043} 
	\lambda^k(\underline{E},\phi):=(\Lambda^k(\underline{E}), 
	[\lambda^k(\ch(\underline{E}),\phi)]_{\mathrm{odd}}).
	\end{align}		
	It is clear that
	\begin{align}\label{eq:2.044}
	\lambda^0(\underline{E},\phi)=1, \quad
	\lambda^1(\underline{E},\phi)=(\underline{E},\phi).
	\end{align}	
	 By \eqref{eq:2.033}, \eqref{eq:2.039}, \eqref{eq:2.040}
	 and  \eqref{eq:2.043}, 
	 we have for any cycles $(\underline{E},\phi)$ and
	 $(\underline{F},\psi)$,
	\begin{multline}\label{eq:2.045}
	\lambda^k((\underline{E},\phi)+(\underline{F},\psi))
	=\lambda^k(\underline{E}\oplus\underline{F},\phi+\psi)
	\\
	=
\left(\Lambda^k(\underline{E}\oplus\underline{F}), 
[\lambda^k(\ch(\underline{E}\oplus\underline{F}),
\phi+\psi)]_{\mathrm{odd}}
\right)
\\
	=\left(\sum_{j=0}^k\Lambda^j(\underline{E})
	\otimes\Lambda^{k-j}(\underline{F}),
	\left[\sum_{i=0}^k\lambda^i(\ch(\underline{E}),\phi)*
	\lambda^{k-i}(\ch(\underline{F}),\psi)
	\right]_{\mathrm{odd}} \right)
	\\
	=\sum_{j=0}^k\Big(\Lambda^j\left(\underline{E}\right)
	\otimes\Lambda^{k-j}
	\left(\underline{F}\right),\Big[\Big(\ch\left(
	\Lambda^j\left(\underline{E}\right)\right),
	\left[\lambda^j\left(\ch\left(\underline{E}\right),
	\phi\right)\right]_{\mathrm{odd}}\Big)
	\\
	*
	\Big(\ch\left(\Lambda^{k-j}\left(\underline{F}
	\right)\right),
	\left[\lambda^{k-j}\left(\ch\left(\underline{F}\right)
	,\psi\right)\right]_{\mathrm{odd}}\Big)\Big]_{\mathrm{odd}}\Big)
	\\
	=\sum_{j=0}^k\Big(\Lambda^j(\underline{E}), 
	[\lambda^j(\ch(\underline{E}),\phi)]_{\mathrm{odd}}\Big)
	\cup\Big(\Lambda^{k-j}(\underline{F}),
	[\lambda^{k-j}(\ch(\underline{F}),\psi)]_{\mathrm{odd}}\Big)
	\\
	=\sum_{j=0}^k\lambda^j(\underline{E},\phi)\cup 
	\lambda^{k-j}(\underline{F},\psi),
	\end{multline}
	where the second equality 
	is implied by \eqref{eq:2.043}, the third equality	
	follows from Definition \ref{defn:2.01}c) and the pre-$\lambda$-ring 
	structure on $\Gamma(Y)$, the fourth 
	equality is a consequence of the fact that $p$ in 
	(\ref{eq:2.010}) is a homomorphism of pre-$\lambda$ rings,
	\begin{align}\label{eq:2.046} 
\lambda^j(\ch(\underline{E}),\phi)=\Big(\lambda^j\left(\ch
\left(\underline{E}\right)\right),
\left[\lambda^j\left(\ch\left(\underline{E}\right),
\phi\right)\right]_{\mathrm{odd}}\Big)
	\end{align}
and \eqref{eq:2.033}, the last two equalities 
follows from \eqref{eq:2.040} and \eqref{eq:2.043}.
So we only need to prove that $\lambda^k$ is well-defined on 
$\widehat{K}^0(Y)$.
	
	If $(\underline{E_1},\phi_1)\sim (\underline{E_2},\phi_2)$,
	there exist $\underline{E_3}$ and isomorphism
	$\Phi:E_1\oplus E_3\rightarrow E_2\oplus E_3$ such that 
	\begin{align}\label{eq:2.047} 
	\widetilde{\ch}\left( (\Phi^{-1})^{*} 
	(\underline{E_1}\oplus\underline{E_3}), 
	\underline{E_2}\oplus\underline{E_3}\right)
	=
	\widetilde{\ch}\left(\underline{E_1}\oplus\underline{E_3}, \Phi^{*} 
	\left(\underline{E_2}\oplus\underline{E_3}\right)\right)
	=\phi_2-\phi_1.
	\end{align}
	From \eqref{eq:2.034}, \eqref{eq:2.043}
	and (\ref{eq:2.047}), we have 
	\begin{multline}\label{eq:2.048}
	\lambda^k(\underline{E_2}\oplus\underline{E_3},
	\widetilde{\ch}\left( 
	\underline{E_1}\oplus\underline{E_3}, 
	\Phi^{*} (\underline{E_2}\oplus\underline{E_3}))\right)
	\\
	=\left( \Lambda^k\left(\underline{E_2}\oplus\underline{E_3}\right),
	\left[\lambda^k\left(\ch(\underline{E_2}\oplus\underline{E_3}),
	\widetilde{\ch}\left( (\Phi^{-1})^{*} 
	(\underline{E_1}\oplus\underline{E_3}), 
	\underline{E_2}\oplus\underline{E_3}\right)
	\right)\right]_{\mathrm{odd}}\right)
	\\
	=\Big( \Lambda^k\left(\underline{E_2}\oplus\underline{E_3}\right), 
	\widetilde{\ch}\left( (\Phi^{-1})^{*} 
	\Lambda^k(\underline{E_1}\oplus\underline{E_3}), 
\Lambda^k(	\underline{E_2}\oplus\underline{E_3})\right)\Big).
	\end{multline}
By Definition \ref{defn:2.05}, \eqref{eq:2.011} and \eqref{eq:2.043}, 
	we get
\begin{multline}\label{eq:2.049}
\lambda^k(\underline{E_1}\oplus\underline{E_3},0)
= \big( \Lambda^k\left(\underline{E_1}\oplus\underline{E_3}\right),
	\left[\lambda^k\left(\ch(\underline{E_1}\oplus\underline{E_3}),
	0\right)\right]_{\mathrm{odd}}\big)\\
	= \big( 
	\Lambda^k\left(\underline{E_1}\oplus\underline{E_3}\right), 0\big)
	\sim\Big( \Lambda^k\left(\underline{E_2}\oplus\underline{E_3}\right), 
	\widetilde{\ch}\left((\Phi^{-1})^{*}
	\Lambda^k(\underline{E_1}\oplus\underline{E_3}),
	\Lambda^k\left(\underline{E_2}
	\oplus\underline{E_3}\right)\right)\Big).
	\end{multline}
From \eqref{eq:2.048} and \eqref{eq:2.049}, we get 
	\begin{align}\label{eq:2.050} 
	\lambda^k(\underline{E_1}\oplus\underline{E_3},0)\sim 
	\lambda^k\left(\underline{E_2}\oplus\underline{E_3},
	\widetilde{\ch}\left( 
	\underline{E_1}\oplus\underline{E_3}, 
	\Phi^{*} (\underline{E_2}\oplus\underline{E_3})\right)\right). 
	\end{align}
	Since for $j=1,2$,
	\begin{align}\label{eq:2.051}
	\lambda_t(\underline{E_j}\oplus\underline{E_3},\phi_j)
	=\lambda_t((\underline{E_j}\oplus\underline{E_3},0)+(0,\phi_j))
	=\lambda_t(\underline{E_j}\oplus\underline{E_3},0)\cup
	\lambda_t(0,\phi_j),
	\end{align}	
	by \eqref{eq:2.047} and \eqref{eq:2.050}, we have 
	\begin{multline}\label{eq:2.052} 
\lambda^k(\underline{E_1}\oplus\underline{E_3},\phi_1)
=\sum_{i=0}^k\lambda^i(\underline{E_1}\oplus\underline{E_3},0)\cup
\lambda^{k-i}(0,\phi_1)
\\
\sim \sum_{i=0}^k\lambda^i(\underline{E_2}\oplus\underline{E_3},
\phi_2-\phi_1)\cup\lambda^{k-i}(0,\phi_1) 
= \lambda^k(\underline{E_2}\oplus\underline{E_3},\phi_2).
	\end{multline}
By (\ref{eq:2.045}), for any $k\geq 1$, $j=1,2$, we have 
	\begin{align}\label{eq:2.053} 
	\lambda^k(\underline{E_j}\oplus\underline{E_3},\phi_j)=
	\sum_{i=0}^k\lambda^i(\underline{E_j},\phi_j)\cup 
	\lambda^{k-i}(\underline{E_3},0).
	\end{align}
	Note that $\lambda^1(\underline{E_1},\phi_1)
	=(\underline{E_1},\phi_1)\sim 
	(\underline{E_2},\phi_2)=\lambda^1(\underline{E_2},\phi_2)$. 
	We assume that
	$\lambda^i(\underline{E_1},\phi_1)\sim \lambda^i
	(\underline{E_2},\phi_2)$
	holds for all $1\leq i\leq k-1$. Then 
	\begin{align}\label{eq:2.054}
	\sum_{i=0}^{k-1}\lambda^i(\underline{E_1},\phi_1)\cup 
	\lambda^{k-i}(\underline{E_3},0)\sim
	\sum_{i=0}^{k-1}\lambda^i(\underline{E_2},\phi_2)\cup 
	\lambda^{k-i}(\underline{E_3},0).
	\end{align}
From \eqref{eq:2.052}-\eqref{eq:2.054}, since $\lambda^0(x)=1$, 
	$\lambda^k(\underline{E_1},\phi_1)\sim 
	\lambda^k(\underline{E_2},\phi_2)$.
	So	 by induction, for any $k\geq 1$, we have
	$\lambda^k(\underline{E_1},\phi_1)\sim 
	\lambda^k(\underline{E_2},\phi_2)$.
	
	The proof of Theorem \ref{thm:2.06} is completed.
	\end{proof}

\begin{lemma}\label{lem:2.07} 
If $\sigma: X\rightarrow Y$ is a $\cC^{\infty}$ map of 
compact manifolds, then its pull-back maps
\begin{align}\label{eq:2.54b}
\begin{split}
&\sigma^*:\Gamma(Y)\rightarrow \Gamma(X), \qquad
\sigma^*(\omega, \phi)=(\sigma^*\omega, \sigma^*\phi);
\\
&\hat{\sigma}^*:\widehat{K}^0(Y)\rightarrow \widehat{K}^0(X),\quad
\hat{\sigma}^*[\underline{E},\phi]
=\left[\sigma^*\underline{E}, \sigma^*\phi\right],
\end{split}
\end{align}
are morphisms of pre-$\lambda$-rings.
\end{lemma}
\begin{proof}
From (\ref{eq:2.008}) and (\ref{eq:2.009}), $\sigma^*$ commutes with
the operators $*$ and $\Psi^k$, thus 
$\sigma^*:\Gamma(Y)\rightarrow \Gamma(X)$ is a morphism of 
pre-$\lambda$-rings. Now from (\ref{eq:2.023}) and
(\ref{eq:2.024}), we have  $\sigma^*f_{\underline{E}}
=f_{\sigma^*\underline{E}}\sigma^*$, $\widetilde{\varphi}\sigma^*
=\sigma^*\widetilde{\varphi}$.
From (\ref{eq:2.043}) and 
$\sigma^*\Lambda^k(\underline{E})
=\Lambda^k(\sigma^*\underline{E})$, 
we get that the second map of (\ref{eq:2.54b})
is well-defined and a morphism of 
pre-$\lambda$-rings. 
\end{proof}

\begin{rem}
In \cite[Chapter V]{BerthelotTh}, the $\lambda$-ring is well studied, 
and needs two additional conditions 
compared to Definition \ref{defn:2.01}.
It is well-known that the topological $K$-group $K^0(Y)$ is 
a  $\lambda$-ring (cf. \cite[Theorem 1.5]{ATall69}). In 
\cite{Ro01}, the author proves that the arithmetic 
$K$-group is also a  $\lambda$-ring. It is natural to ask whether 
the differential $K$-group $\widehat{K}^0(Y)$ is also a $\lambda$-ring,
and we will come back to this question later.
However, for our application here, the pre-$\lambda$-ring structure for
 differential $K$-theory is enough. 
\end{rem}

\subsection{$\gamma$-filtration}\label{s0303}

\begin{defn}\cite[(1.25)]{BerthelotTh}\label{defn:2.08}
Let $R$ be any pre-$\lambda$-ring with an augmentation homomorphism 
$\rank:R\rightarrow \Z$. The $\gamma$-operations are defined by
\begin{align}\label{eq:2.055} 
\gamma_t(x)=\sum_{j\geq 0}\gamma^j(x)t^j
:=\lambda_{\frac{t}{1-t}}(x).
\end{align}	
\end{defn}
By Definition \ref{defn:2.01}, we have 
\begin{align}\label{eq:2.056}
\gamma^0=1,\quad \gamma^1(x)=x,\quad \gamma_t(x+y)
=\gamma_t(x)\gamma_t(y),\quad \text{for any}\ x,y\in R.
\end{align} 

\begin{defn}
Set $F^nR:=R$ for $n\leq 0$ and $F^1R$ the kernel of 
$\rank:R\rightarrow \Z$. Let $F^nR$ be the additive subgroup 
generated by $\gamma^{r_1}(x_1)\cdots \gamma^{r_k}(x_k)$, where 
$x_1,\cdots,x_k\in F^1R$ and $\sum_{j=1}^kr_j\geq n$. The filtration
\begin{align}\label{eq:2.057}
F^1R\supseteq F^2R\supseteq F^3R\supseteq\cdots
\end{align}  
is called the $\gamma$-filtration of $R$. The $\gamma$-filtration is 
said to be \textbf{locally nilpotent} at $x\in F^1R$, if there 
exists $M(x)\in \N$, such that $\gamma^{r_1}(x)\cdots 
\gamma^{r_k}(x)=0$ for any $\sum_{j=1}^kr_j>M(x)$. 
\end{defn}

Let $Y$ be a compact connected manifold.

It is well-known that the classical $\gamma$-filtration of  $K^0(Y)$ is 
locally nilpotent for any $x\in F^1K^0(Y)$ \cite[Proposition 
3.1.5]{A67} with the augmentation homomorphism
$\mathrm{rk}:E\rightarrow \rank E$, 
the rank of the complex vector bundle $E$. Since 
$\widehat{K}^0(Y)$ is a pre-$\lambda$-ring, the augmentation 
homomorphism 
$\rank (\underline{E},\phi):=\rank E$ defines a 
$\gamma$-filtration of  $\widehat{K}^0(Y)$.  

We identify a geometric triple $\underline{E}$
to the cycle $(\underline{E},0)$.
By \eqref{eq:2.001}, \eqref{eq:2.011} and \eqref{eq:2.043}, we have
\begin{align}\label{eq:2.058} 
\lambda_t(\underline{E})=\sum_{j\geq 0}\Lambda^j(\underline{E})t^j.
\end{align}
 So 
$\gamma^i(\underline{E})$, defined as in \eqref{eq:2.055}, 
is a finite dimensional virtual Hermitian vector 
bundle with induced metric and connection. We also denote it by 
$\underline{\gamma^i(E)}$.

 Let $\underline{k}$ 
 be the $k$-dimensional trivial complex
vector bundle 
with trivial metric and connection.
Then $F^1\widehat{K}^0(Y)$ is generalized by cycles 
$(\underline{E}-\underline{\rank E},\phi)$ for 
$\phi\in \Omega^{\mathrm{odd}}(Y,\R)/\Im \, d$. By 
\eqref{eq:2.041}, (\ref{eq:2.055}) and \eqref{eq:2.058},
\begin{align}\label{eq:2.059} 
\lambda_t(\underline{\C})=1+t,\quad \gamma_t(\underline{\C})
=1+\frac{t}{1-t}\cdot 1=\frac{1}{1-t}.
\end{align}

From \eqref{eq:2.055}, \eqref{eq:2.056}, \eqref{eq:2.058}
and \eqref{eq:2.059}, letting $r=\rank E$, we have
\begin{multline}\label{eq:2.060} 
\gamma_t(\underline{E}-\underline{\rank 
E})=\gamma_t(\underline{E})\gamma_t(\underline{\C})^{-r}
=\lambda_{\frac{t}{1-t}}(\underline{E})(1-t)^{r}
=\sum_{i=0}^{r}\Lambda^i(\underline{E})t^i(1-t)^{r-i}
\\
=\sum_{i=0}^{r }\sum_{j=0}^{r-i}(-1)^j \left(\begin{array}{c}
r-i \\ 
j
\end{array} \right) \Lambda^i(\underline{E})t^{i+j}
=\sum_{k=0}^{r}\left(\sum_{i=0}^k (-1)^{k-i} \left( \begin{array}{c}
r-i \\ 
k-i
\end{array} \right) 
\Lambda^i(\underline{E}) \right)t^k.
\end{multline}
So
\begin{align}\label{eq:2.061}
\gamma^k(\underline{E}-\underline{\rank 	E})=
\begin{cases}
\sum_{i=0}^k (-1)^{k-i}\begin{pmatrix}
r-i \\ 
k-i
\end{pmatrix}   
\Lambda^i(\underline{E})  & \hbox{ if }0\leq k\leq r=\rank E; \\
0  & \hbox{ if } k>r.
\end{cases}
\end{align}	

Since $\lambda_t(x)=\gamma_{t/(1+t)}(x)$, from (\ref{eq:2.061}), we 
have
\begin{multline}\label{eq:2.062}
\lambda_{t}(\underline{E})=\lambda_t(\underline{E}-\underline{\rank
	E})\cdot \lambda_t(\underline{\rank 
	E })=\gamma_{\frac{t}{1+t}}(\underline{E }-\underline{\rank 
	E })\cdot(1+t)^{r}
\\
=(1+t)^{r}\Big(1+ \sum_{i=1}^{r}\gamma^i(\underline{E }
-\underline{\rank E })t^i(1+t)^{-i}\Big).
\end{multline}
Formally, we have
\begin{multline}\label{eq:2.063}
\lambda_{t}(\underline{E })^{-1}
=(1+t)^{-r}\Big(1+ \sum_{i=1}^{r}\gamma^i(\underline{E }
-\underline{\rank
	E })t^i(1+t)^{-i}\Big)^{-1}
\\
=(1+t)^{-r}\left(1+\sum_{j=1}^{\infty}(-1)^j\Big(
\sum_{i=1}^{r}\gamma^i(\underline{E }
-\underline{\rank	E })t^i(1+t)^{-i}\Big)^j \right)
\\
=(1+t)^{-r}\left(1+\sum_{k=1}^{\infty}t^k(1+t)^{-k}
\sum_{\substack{(n_1,\cdots,n_r)\in \N^r,\\ 
\sum_{i=1}^ri\cdot n_i=k}}(-1)^{\sum_{i=1}^r n_i}
\frac{(\sum_{i=1}^r n_i)!}{\prod_{i=1}^rn_i!} 
\prod_{i=1}^r\left(\gamma^{i}(\underline{E }
-\underline{\rank
E })\right)^{n_i}   \right).
 \end{multline}
 
To simplify the notations, 
we denote by
\begin{align}\label{eq:2.064} 
\lambda_{t}(\underline{E })^{-1}
=(1+t)^{-r}\left(1+\sum_{k=1}^{\infty}t^k(1+t)^{-k}
\big(P_{k,+}(\underline{E})-P_{k,-}(\underline{E})\big) \right),
\end{align} 
by 
using \eqref{eq:2.061} in \eqref{eq:2.063}.
Remark that $P_{k,\pm}(\underline{E})$ here are finite dimensional
Hermitian vector 
bundles with induced metrics and connections
obtained from $\underline{E}$.
We denote by  $P_{k,+}(\underline{E})=P_{k,-}(\underline{E})=0$
for $k<0$ and $P_{0,+}(\underline{E})=\underline{\C}$,
$P_{0,-}(\underline{E})=0$. From (\ref{eq:2.062}) and (\ref{eq:2.064}),
we know that for any $l\in \N^*$, 
\begin{align}\label{eq:2.66b} 
\sum_{i=0}^r\gamma^{i}(\underline{E }
-\underline{\rank E })\Big(P_{l-i,+}(\underline{E})
-P_{l-i,-}(\underline{E})\Big)=0,
\end{align}
as geometric triples.

The following theorem is the differential $K$-theory version of
the locally nilpotent property in topological $K$-theory
 \cite[Propositions 3.1.5, 3.1.10]{A67}. The corresponding 
 arithmetic $K$-theory version was proved by Roessler 
 \cite[Proposition 4.5]{Ro99}.

\begin{thm}\label{thm:2.10}
The $\gamma$-filtration of  $\widehat{K}^0(Y)$ is locally nilpotent 
at $[\underline{E}-\underline{\rank E},0]$. Explicitly, 
there exists $\mN_{r,m}>0$ (depending only on $r,m$) such that 
for any geometric triple $\underline{E}$ on $Y$  with $r=\rank E$,
$m=\dim Y$, and $(n_1,\cdots,n_r)\in \N^r$ such that 
	\begin{align}\label{eq:2.065} 
\sum_{i=1}^r 
i\cdot n_i>\mN_{r,m},
	\end{align}
we have	 
		\begin{align}\label{eq:2.066}
		\prod_{i=1}^r\left(\gamma^{i}([\underline{E }
		-\underline{\rank
			E },0])\right)^{n_i}=
	\left[\prod_{i=1}^r\left(\gamma^{i}(\underline{E }
	-\underline{\rank
		E })\right)^{n_i} ,0\right]
	=0\in\widehat{K}^0(Y).
	\end{align}	
\end{thm}
\begin{proof}  
	For any complex vector bundle $F$ over $Y$, 
we get from (\ref{eq:2.061}) the $\gamma^k$-operation on $K^0(Y)$
	by forgetting the metric and connection,
	\begin{align}\label{eq:2.065b}
	\gamma^k(F-\rank F)=
	\left\{
	\begin{array}{ll}
	\sum_{i=0}^k (-1)^{k-i}\begin{pmatrix}
	\rank F-i \\ 
	k-i
	\end{pmatrix} 
	\Lambda^i(F) & \hbox{if }  0\leq k\leq  \rank F; \\
	0 & \hbox{if } k>\rank F.
	\end{array}
	\right.
	\end{align}	
From the classical property of the $\gamma$-operation
in $K$-theory \cite[Propositions 3.1.5, 3.1.10]{A67}, we know that 
there exists $a_{Y}>0$ (depending only on $Y$) such that 
\begin{align}\label{eq:2.067}
\gamma^{i_{1}}(x_{1})\gamma^{i_{2}}(x_{2})
\cdots \gamma^{i_{k}}(x_{k})=0\in K^0(Y)\quad \text{ if }\, 
x_{j}\in F^{1}K^{0}(Y),  i_{j}\in \N, \, \, 
\sum_{j=1}^{k}i_{j}\geq a_{Y}.
\end{align}

Let $\underline{E }$ be a geometric triple on $Y$
with $r=\rank E$, $m=\dim Y$.
Let $\mathrm{U}(E)$ be the $\mathrm{U}(r)$-principal
bundle of unitary frames of the Hermitian vector bundle ($E,h^E$). 
Then $\mathrm{U}(E)\times_{\mathrm{U}(r)}\C^r\simeq E$ and 
$h^E$ coincides with the Hermitian metric induced by the 
canonical Hermitian inner product on $\C^r$. The Hermitian connection
$\nabla^E$ corresponds uniquely to a 
connection $\omega$ on $\mathrm{U}(E)$.
	
Let $\mathrm{Gr}(r,\C^p)$ be the 
Grassmannian, which is the space parameterizing all complex linear 
subspaces of $\C^p$ of given dimension $r$. 
Let $\mathrm{V}(p,r)$ be the canonical $\mathrm{U}(r)$-principal
bundle over $\mathrm{Gr}(r,\C^p)$ with 
the canonical connection $\omega_0$,
which is induced by the Maurer-Cartan form
on $\mathrm{U}(p)$ on the $\mathrm{U}(r)$-principal
bundle $\mathrm{V}(p,r)=\mathrm{U}(p)/(\mathrm{I}_r
\times \mathrm{U}(p-r))\rightarrow
 \mathrm{U}(p)/(\mathrm{U}(r)
\times \mathrm{U}(p-r))=\mathrm{Gr}(r,\C^p)$
via the canonical matrix decomposition of 
$\mathfrak{u}(p)$.
Let $H=\mathrm{V}(p,r)\times_{\mathrm{U}(r)}\C^r$. 
Let $h^H$ be the Hermitian metric on $H$
induced by the canonical inner product on $\C^r$.
Let $\nabla^H$ be the Hermitian connection
on $H$ induced by $\omega_0$.
Let $\underline{\rank H}$ be the $r$-dimensional trivial vector bundle 
over $\mathrm{Gr}(r,\C^p)$ with trivial metric and connection.

By a theorem of Narasimhan and Ramanan 
\cite[Theorem 1]{NarasimhanR61}, for 
\begin{align}\label{eq:2.068} 
p=(m+1)(2m+1)r^3,
\end{align}
there exists a map $f:Y\rightarrow \mathrm{Gr}(r,\C^p)$ 
such that $f^*\mathrm{V}(p,r)=\mathrm{U}(E)$ and 
$f^*\omega_0=\omega$. 
Thus $f^*(H,h^H,\nabla^H) =(E,h^E, \nabla^E)$.  
Let $\hat{f}^*:\widehat{K}^0(\mathrm{Gr}(r,\C^p))\rightarrow
\widehat{K}^0(Y)$ be the pull-back map. 
From \eqref{eq:2.043}, \eqref{eq:2.055}
and  \eqref{eq:2.061}, we have
\begin{multline}\label{eq:2.069}
\prod_{i=1}^r\left(\gamma^{i}([\underline{E }
-\underline{\rank E },0])\right)^{n_i}=
\left[\prod_{i=1}^r\left(\gamma^{i}(\underline{E }
-\underline{\rank E})\right)^{n_i},0\right]
	\\
=\hat{f}^*\left\{\left[\prod_{i=1}^r\left(\gamma^{i}(\underline{H }
-\underline{\rank H })\right)^{n_i},0\right]\right\}\in\widehat{K}^0(Y).
\end{multline}
By (\ref{eq:2.067}) and (\ref{eq:2.068}), there exists 
$a_{r,m}>0$ (depending only on $r,m$)	such that 
\begin{align}\label{eq:2.070}
\prod_{i=1}^r\left(\gamma^{i}(H-\rank H)\right)^{n_i}=0
\in K^0(\mathrm{Gr}(r,\C^p)) 
\quad \text{ if } \sum_{i=1}^r i\cdot n_i>a_{r,m}.
\end{align}
From Definition \ref{defn:2.05} and \eqref{eq:2.070},
if $\sum_{i=1}^r 	i\cdot n_i>a_{r,m}$, there exists 
$\alpha\in \Omega^{\mathrm{odd}}(\mathrm{Gr}(r,\C^p),\R)$
such that	
	\begin{align}\label{eq:2.071}
	\Big[\prod_{i=1}^r\left(\gamma^{i}(\underline{H }
	-\underline{\rank
		H })\right)^{n_i},0\Big]=[0,\alpha]\in 
	\widehat{K}^0(\mathrm{Gr}(r,\C^p))
	\end{align}
	and
	\begin{multline}\label{eq:2.072}
	-d\alpha=\ch\Big(\prod_{i=1}^r\left(\gamma^{i}(\underline{H }
	-\underline{\rank	H })\right)^{n_i}\Big)
	=\prod_{i=1}^r\Big(\ch\left(\gamma^{i}(\underline{H }
	-\underline{\rank	H })\right)\Big)^{n_i}
	\\
	\in \Omega^{\mathrm{even}}(\mathrm{Gr}(r,\C^p),\R).
	\end{multline}

From \eqref{eq:2.023}, \eqref{eq:2.033} and \eqref{eq:2.060}, we have
	\begin{multline}\label{eq:2.073}
	\ch(\gamma_t(\underline{H}-\underline{\rank H}))
		=\sum_{i=0}^r\ch(\Lambda^i(\underline{H}))t^i(1-t)^{r-i}
	\\
	=\sum_{i=0}^r\lambda^i(\ch)(\underline{H})t^i(1-t)^{r-i}
	=\sum_{i=0}^r
	f_{\underline{H}}\big(\lambda^i(\ch)\big)t^i(1-t)^{r-i},
	\end{multline}
	where $f_{\underline{H}}$ is the map \eqref{eq:2.023} 
	with respect to $\underline{H}$.
	From \eqref{eq:2.032}, we have
	\begin{multline}\label{eq:2.074}
	\sum_{i=0}^r\theta^*\lambda^i(\ch)t^i(1-t)^{r-i}
	=\sum_{i=0}^r\sigma_i(e^{u_1},\cdots,e^{u_r})t^i(1-t)^{r-i}
	\\
	=\prod_{j=1}^r\Big((1-t)+te^{u_j}\Big)
	=\prod_{j=1}^r\Big(1+t(e^{u_j}-1)\Big)
	=\sum_{i=0}^r\sigma_i(e^{u_1}-1,\cdots,e^{u_r}-1)t^i.
	\end{multline}
	Since $\theta^*$ is injective and 
	$\sigma_i(e^{u_1}-1,\cdots,e^{u_r}-1)$
	is symmetric with respect to $u_i$'s, from \eqref{eq:2.073}
	and \eqref{eq:2.074}, we have
	\begin{align}\label{eq:2.075} 
	\ch(\gamma^{i}(\underline{H}-\underline{\rank H}))=f_{\underline{H}}
	\circ (\theta^*)^{-1}\big(\sigma_i(e^{u_1}-1,\cdots,e^{u_r}-1) \big)
	\in Z^{\mathrm{even}}(\mathrm{Gr}(r,\C^p),\R).
	\end{align} 
	From Lemma \ref{lem:2.03},   
	we see that the Adams operation 
	$\Psi^k$ commutes with $f_{\underline{H}}
	\circ (\theta^*)^{-1}$. 
Since
	\begin{align}\label{eq:2.076}
	\Psi^k\big(\sigma_i(e^{u_1}-1,\cdots,e^{u_r}-1) \big)
	= \sigma_i(e^{ku_1}-1,\cdots,e^{ku_r}-1)
	\end{align}
	is a power series with respect to $k$ such that the coefficients of 
	$1,k,\cdots,k^{i-1}$ vanish, so is 
	$\Psi^k(\ch(\gamma^{i}(\underline{H}-\underline{\rank H})))$, too. 
	Since $\Psi^k \beta=k^l\beta$ for $\beta\in 
	Z^{2l}(\mathrm{Gr}(r,\C^p),\R)$,  we have
	\begin{align}\label{eq:2.077}
	\ch(\gamma^{i}(\underline{H}-\underline{\rank H}))\in 
	\Omega^{\geq 2i}(\mathrm{Gr}(r,\C^p),\R).
	\end{align}

	Since $\dim_{\R} 
	\mathrm{Gr}(r,\C^p)=2r(p-r)=2r^2((m+1)(2m+1)r^2-1)$, for any 
	$(n_1,\cdots,n_r)\in \N^r$ such that $\sum_{i=1}^r 
	i\cdot n_i>r^2((m+1)(2m+1)r^2-1)$, we have
	\begin{align}\label{eq:2.078} 
	\ch\left(\prod_{i=1}^r(\gamma^{i}(\underline{H}-\underline{\rank 
		H}))^{n_i}\right)=0
		\in \Omega^{\mathrm{even}}(\mathrm{Gr}(r,\C^p),\R).
	\end{align}
By \eqref{eq:2.072} and \eqref{eq:2.078}, we have
$d\alpha=0$ if $\sum_{i=1}^r i\cdot n_i>\mN_{r,m}$, with
\begin{align}\label{eq:2.080} 
\mN_{r,m}=\sup\{r^2((m+1)(2m+1)r^2-1),a_{r,m}\}.
\end{align}
Since the cohomology of the Grassmannian vanishes 
in odd degrees, 
the differential form $\alpha$ is exact. From \eqref{eq:2.071}, 
$\left[\prod_{i=1}^r(\gamma^{i}(\underline{H}-\underline{\rank 
H}))^{n_i},0\right]=0\in \widehat{K}^0(\mathrm{Gr}(r,\C^p))$. 
By \eqref{eq:2.069}, 	for any 
$(n_1,\cdots,n_r)\in \N^r$ such that $\sum_{i=1}^r 
i\cdot n_i>\mN_{r,m}$, we have
\begin{align}\label{eq:2.079}
\left[\prod_{i=1}^r(\gamma^{i}(\underline{E}-\underline{\rank 
	E}))^{n_i},0\right]
=0\in\widehat{K}^0(Y).
\end{align}
	
The proof of Theorem \ref{thm:2.10} is completed.		
\end{proof}

\begin{rem}\label{rem:2.11} 
a). 	Comparing with \eqref{eq:2.066} and \eqref{eq:2.067}, 
a natural question about the $\gamma$-filtration of $\widehat{K}^0(Y)$
is  whether we can replace $\mN_{r,m}$ in \eqref{eq:2.065} 
by a constant depending only on $Y$ such that 
\eqref{eq:2.066} still holds. More generally, whether 
the analogue of \eqref{eq:2.067} holds in $\widehat{K}^0(Y)$.

b). Let  $\{U_i \}$ be an open covering of $\mathrm{Gr}(r,\C^p)$ 
such that $U_i$'s are diffeomorphic to open balls,
and the covers can be divided into
 $\dim_{\R} \mathrm{Gr}(r,\C^p)+1$ $=2r(p-r)+1$ 
 classes $\mU_1, \cdots, \mU_{2r(p-r)+1}$ such that no two 
$U_i$'s of the same class intersect,
 from a result of Nash (cf. Ann. of Math. 63 (1956), p.61).
  Let $V_j=\bigcup_{U_i\in \mU_j}U_i$,
$j=1,\cdots, 2r(p-r)+1$. Then the union is a disjoint union
with the trivialization $h_j:H|_{V_j}\rightarrow V_j\times \C^r$.
Let $\{\varphi_j \}$ be a partition of unity of $\{V_j \}$.
Let $\pi_1:H\rightarrow \mathrm{Gr}(r,\C^p)$ and
$\pi_2:V_j\times \C^r\rightarrow \C^r$ be the 
canonical projections. Then the map 
$$v\in H\mapsto \Big(\pi_1(v), \varphi_1(\pi_1(v))\cdot
\pi_2(h_1(v)),\cdots, \varphi_{2r(p-r)+1}(\pi_1(v))\cdot
\pi_2(h_{2r(p-r)+1}(v))\Big)$$ 
induces a bundle map 
$H\rightarrow \mathrm{Gr}(r,\C^p)\times \C^{r(2r(p-r)+1)}$
such that it is injective at each fiber. Then we obtain 
a complex vector bundle $F$ over
$\mathrm{Gr}(r,\C^p)$ with rank $2r^2(p-r)$ as the complement of $H$
in $\mathrm{Gr}(r,\C^p)\times \C^{r(2r(p-r)+1)}$.
Thus we have
$$\gamma_t(\rank H-H)=\gamma_t((\rank H+F)-(H+F))
=\gamma_t(F-\rank F).$$
From (\ref{eq:2.065b}) and the above equation, $\gamma_t(\rank H-H)$
is a polynomial in $t$ with degree $\leq \rank F$.
By applying 
the remark before \cite[Proposition 3.1.5]{A67} to the polynomials
$\gamma_t(H-\rank H)$, $\gamma_t(\rank H-H)$,
we can take 
$$a_{r,m}=\rank H\cdot \rank F=2r^3(p-r)=2r^4((m+1)(2m+1)r^2-1).$$
Thus we can take $\mN_{r,m}=2r^4((m+1)(2m+1)r^2-1)$.
\end{rem}

From Theorem \ref{thm:2.10}, \eqref{eq:2.063} and \eqref{eq:2.064}, 
we have the following corollary.
\begin{cor}\label{cor:2.11}
If $k>\mN_{r,m}$ with $\mN_{r,m}$ as in \eqref{eq:2.080}, we have
\begin{align}\label{eq:2.081}
[P_{k,+}(\underline{E}),0]=[P_{k,-}(\underline{E}),0]\in\widehat{K}^0(Y).
\end{align}
\end{cor}

\subsection{The $g$-equivariant differential $K$-theory.\ 
	Proof of Theorem \ref{i06}}\label{s0304}
Let $Y$ be a compact manifold with $S^1$-action.
In the following, we adapt the notation in Section \ref{s02}. 

The following definition is an extension of Definition \ref{defn:2.05}.
\begin{defn}\label{defn:2.12} 
	For $g\in S^1$, a cycle for the $g$-equivariant differential $K$-theory
	of $Y$ is a pair $(\underline{E}, \phi)$ where 
	$\underline{E}=(E, h^E, \nabla^E)$
	is an $S^1$-equivariant geometric triple on $Y$ and $\phi$
	is an element in $\Omega^{\mathrm{odd}}(Y^g, \C)/\Im \, d$.
	Two cycles $(\underline{E_1}, \phi_1)$ and $(\underline{E_2}, 
	\phi_2)$ are equivalent if there exist an $S^1$-equivariant 
	geometric triple $\underline{E_3}=(E_3, h^{E_3}, \nabla^{E_3})$ 
	and an $S^1$-equivariant vector bundle isomorphism over $Y$
	\begin{align}\label{eq:2.082}
	\Phi:E_1\oplus E_3\rightarrow E_2\oplus E_3
	\end{align}
	such that
	\begin{align}\label{eq:2.083}
	\widetilde{\ch}_g\left(\underline{E_1}\oplus\underline{E_3}, 
	\Phi^{*} 
	\left(\underline{E_2}\oplus\underline{E_3}\right)\right)
	=\phi_2-\phi_1.
	\end{align} 
	We define the sum in the same way as in \eqref{eq:2.039}.	
	We define the $g$-equivariant differential $K$-group 
	$\widehat{K}_g^0(Y)$ 
	to be the Grothendieck group of equivalence classes of cycles. 
\end{defn}
	
We denote by $[\underline{E}, \phi]\in \widehat{K}^0_g(Y)$
the equivalence class of a cycle $(\underline{E}, \phi)$.
For $[\underline{E}, \phi], [\underline{F}, \psi]\in \widehat{K}^0_g(Y)$,
set
	\begin{align}\label{eq:2.084}
	[\underline{E}, \phi]\cup[\underline{F}, \psi]
	=[\underline{E}\otimes \underline{F}, 
	\ch_g(\underline{E})\wedge\psi+\phi\wedge\ch_g(\underline{F})
	-d\phi\wedge\psi  ].
	\end{align} 
	From (\ref{eq:2.084}), 
we deduce that the element $1:=[\underline{\C},0]$ is a unit: 
	the circle action on the total space $Y\times\C$ is defined by 
	$g(y,c)=(gy,c)$ for any $(y,c)\in Y\times \C$,
	the trivial metric and connection are obviously $S^1$-invariant.
	
\begin{lemma}\label{lemma:2.15a}
The product (\ref{eq:2.084}) on $\widehat{K}^0_g(Y)$ is well-defined, 
 associative and commutative. 
\end{lemma}	
\begin{proof}
If two cycles $(\underline{E_1}, \phi_1)$ and $(\underline{E_2}, \phi_2)$
are equivalent, 
then by (\ref{local40}) and (\ref{eq:2.083}), we have
\begin{align*}
	\widetilde{\ch}_g\left(\left(\underline{E_1}\oplus\underline{E_3}
	\right)\otimes \underline{F}, 
\Phi^{*} 
\left(\underline{E_2}\oplus\underline{E_3}\right)\otimes \underline{F}
\right)
=(\phi_2-\phi_1)\wedge \ch_g\left(\underline{F}\right),
\end{align*}
and
\begin{align*}
d(\phi_2-\phi_1)=\ch_g\left(\underline{E_2}\oplus\underline{E_3}\right)
-\ch_g\left(\underline{E_1}\oplus\underline{E_3}\right)
=\ch_g\left(\underline{E_2}\right)
-\ch_g\left(\underline{E_1}\right).
\end{align*}
Thus $(\underline{E_1}\otimes \underline{F}, 
\ch_g(\underline{E_1})\wedge\psi+\phi_1\wedge\ch_g(\underline{F})
-d\phi_1\wedge\psi)$ and $(\underline{E_2}\otimes \underline{F}, 
\ch_g(\underline{E_2})\wedge\psi+\phi_2\wedge\ch_g(\underline{F})
-d\phi_2\wedge\psi)$ are equivalent. 
The product (\ref{eq:2.084}) on $\widehat{K}^0_g(Y)$ is well-defined.

The commutativity of (\ref{eq:2.084}) follows from the facts that
$\phi, \psi\in \Omega^{\mathrm{odd}}(Y^g, \C)/\Im \, d$
and $d\phi\wedge \psi=d\psi\wedge \phi+d(\phi\wedge \psi)$.

We verify now the associativity of the product.
From (\ref{eq:2.084}), we have
\begin{multline}\label{eq:2.087a}
[\underline{E_1}, \phi_1]\cup\Big([\underline{F_1}, \psi_1]
\cup [\underline{F}, \psi] \Big)
= \Big[ \underline{E_1}\otimes \underline{F_1}\otimes \underline{F},
\\
\big(\ch_g(\underline{E_1})-d\phi_1\big)\wedge 
\big(\ch_g(\underline{F_1})\wedge\psi
+\psi_1\wedge\ch_g(\underline{F})
-d\psi_1\wedge\psi\big)
+\phi_1\wedge \ch_g(\underline{F_1}\otimes \underline{F})\Big],
\end{multline}
and
\begin{multline}\label{eq:2.088a}
\Big([\underline{E_1}, \phi_1]\cup[\underline{F_1}, \psi_1]\Big)
\cup [\underline{F}, \psi] 
= \Big[ \underline{E_1}\otimes 
\underline{F_1}\otimes \underline{F},
\\ \ch_g(\underline{E_1}\otimes \underline{F_1})\wedge \psi 
-d\big(\ch_g(\underline{E_1})\wedge\psi_1
+\phi_1\wedge\ch_g(\underline{F_1})
-d\phi_1\wedge\psi_1 \big)\wedge\psi  \\
+\big(\ch_g(\underline{E_1})\wedge
\psi_1+\phi_1\wedge\ch_g(\underline{F_1})
-d\phi_1\wedge\psi_1 \big)\wedge \ch_g(\underline{F})\Big].
\end{multline}

Since $\ch_g(\cdot)$ is a closed even form and $\phi_1, \psi_1, \psi$
are odd forms, from (\ref{eq:2.087a}) and (\ref{eq:2.088a}), we have
\begin{align}
[\underline{E_1}, \phi_1]\cup\left([\underline{F_1}, \psi_1]
\cup [\underline{F}, \psi] \right)
=\left([\underline{E_1}, \phi_1]\cup[\underline{F_1}, \psi_1]\right)
\cup [\underline{F}, \psi].
\end{align}

The proof of Lemma \ref{lemma:2.15a} is completed.
\end{proof}
	
	Thus $(\widehat{K}^0_g(Y), +, \cup)$ is a commutative ring
	with unit $1$.

\begin{rem}\label{rem:2.13} 
Certainly, we can replace $S^1$ by any compact Lie group in 
Definition \ref{defn:2.12}.
\end{rem}

As in \eqref{eq:2.042}, if  $[\underline{E},0]=[\underline{F},0]\in 
\widehat{K}^0_g(Y)$, from (\ref{local126}) and (\ref{eq:2.083}), we have
\begin{align}\label{eq:2.085} 
\ch_g(\underline{E})=\ch_g(\underline{F})\in \Omega^{\bullet}(Y^{g},\C).
\end{align}

Note that $g\in S^1$ defines a prime ideal $I(g)$ in $R(S^1)$, 
the representation ring of $S^1$, namely all 
characters of $S^1$ which vanish at $g$. For any $R(S^1)$-module 
$\mathcal{M}$, 
we denote by $\mathcal{M}_{I(g)}$ the module obtained from 
$\mathcal{M}$ by localizing at this prime ideal.  An element of 
$R(S^1)_{I(g)}$ is a ``fraction" $u/v$ with $u,v\in R(S^1)$ and 
$\chi_v(g)\neq 0$, and two fractions $u/v$ and $u'/v'$ represent the 
same element of $R(S^1)_{I(g)}$ if there exists $w\in R(S^1)$ with 
$\chi_w(g)\neq 0$ and $wuv'=wu'v\in R(S^1)$.
Elements of $\mathcal{M}_{I(g)}$ are 
``fractions" $u/v$ ($u\in \mathcal{M}, v\in R(S^1), \chi_v(g)\neq 0$) 
with a 
similar equivalence relation. Thus $\mathcal{M}_{I(g)}$ is a module over 
the local ring $R(S^1)_{I(g)}$. Since we do not distinguish
the finite dimensional virtual representations and the characters of 
elements in $R(S^1)$, we usually write an element of 
$\mathcal{M}_{I(g)}$ by 
\begin{align}\label{eq:2.086} 
u/\chi\quad  \text{with}\ u\in \mM, \chi\in \Z[h,h^{-1}]\quad 
\text{for}\ h\in S^1,  \chi(g)\neq 0.
\end{align}

For a finite dimensional $S^1$-representation $M$, we consider the 
$S^1$-action on $Y\times M$ given by 
\begin{align}\label{eq:2.86b} 
g(y,u)=(gy,gu).
\end{align}
 Thus
 $Y\times M\rightarrow Y$ is an equivariant vector bundle over $Y$.  
 We denote this equivariant vector bundle by $E_M$. 
By construction,
the trivial metric $h^M$ and the trivial connection $d$ 
on $E_M$ are naturally $S^1$-invariant.  
Let $\underline{M}=(E_M, h^M, d)$. Note that $E\mapsto E_M\otimes E$ 
endows the $S^1$-equivariant $K$-group $K_{S^1}^0(Y)$ of $Y$
with the structure of an $R(S^1)$-module.
From \eqref{eq:2.083}, since 
$\ch_g(\underline{M})=\chi_{M}(g)$, constant on $Y^g$,
$(\underline{E}, \phi)\mapsto (\underline{M}\otimes \underline{E},  
\chi_{M}(g)\cdot\phi)$  makes $\widehat{K}_g^0(Y)$ an 
$R(S^1)$-module.

In the following,  we will denote by 
$\underline{'E}$ the corresponding geometric triple when forgetting 
the group action.

Recall that $Y^{S^1}=\{Y^{S^1}_{\alpha} \}_{\alpha\in \mathfrak{B}}$ 
is the fixed point 
set of the circle action and $N_{\alpha}$ is the normal bundle 
of $Y^{S^1}_{\alpha}$ in $Y$. We consider $N_{\alpha}$ 
as a complex vector bundle. 
By \cite[Proposition 2.2]{Segal68},
\begin{align}\label{eq:2.087} 
K_{S^1}^0(Y_{\alpha}^{S^1})\simeq R(S^1)\otimes
K^0(Y_{\alpha}^{S^1}) .
\end{align}

By (\ref{local203}), 
we write in the sense of \eqref{eq:2.087},
\begin{align}\label{eq:2.088}
\underline{\lambda_{-1}(N_{\alpha}^{*})}\simeq\bigotimes_{v=1}^{q}
\lambda_{-h^{-v}}\left(\underline{'N_{\alpha,v}^{*}}\right)
=\bigotimes_{v=1}^{q}\left(1+\sum_{k=1}^{\rank N_{\alpha,v}}
(-h^{-v})^k\cdot\Lambda^k\left(\underline{'N_{\alpha,v}^{*}}\right) 
\right).
\end{align}
Set 
 \begin{align}\label{eq:2.089} 
 r_{\alpha,v}=\rank N_{\alpha,v},\quad m_{\alpha}=\dim 
 Y_{\alpha}^{S^1}.
 \end{align}
By \eqref{eq:2.063} and \eqref{eq:2.064}, we have formally,
\begin{multline}\label{eq:2.090}
\lambda_{-h^{-v}}\left(\underline{'N_{\alpha,v}^* 
}\right)^{-1}
\\
=(1-h^{-v})^{-r_{\alpha,v}}
\left(1+\sum_{k=1}^{\infty}
(-h^{-v})^k(1-h^{-v})^{-k}
\left(P_{k,+}\left(\underline{'N_{\alpha,v}^*}\right)-P_{k,-}
\left(\underline{'N_{\alpha,v}^*}\right)\right)\right)
\\
=\frac{h^{vr_{\alpha,v}}}{(h^v-1)^{r_{\alpha,v}}}
\left(1+\sum_{k=1}^{\infty}
\frac{(-1)^k}{(h^v-1)^{k}}
\left(P_{k,+}\left(\underline{'N_{\alpha,v}^*}\right)-P_{k,-}
\left(\underline{'N_{\alpha,v}^*}\right)\right)\right).
\end{multline}

By Corollary \ref{cor:2.11}, we know that for any $k> 
\mN_{r_{\alpha,v}, m_{\alpha}}$, 
\begin{align}\label{eq:2.091}
\left[P_{k,+}\left(\underline{'N_{\alpha,v}^*}\right)-P_{k,-}
\left(\underline{'N_{\alpha,v}^*}\right),0\right]=0
\in \widehat{K}^0(Y_{\alpha}^{S^1}).
\end{align}

We define
\begin{align}\label{eq:2.092} 
\begin{split}
&\lambda_{-h^{-v}}\left(\underline{'N_{\alpha,v}^*}\right)^{-1}_{\mN}
:=\frac{h^{vr_{\alpha,v}}}{(h^v-1)^{r_{\alpha,v}}}
\left(1+\sum_{k=1}^{\mN}
\frac{(-1)^k}{(h^v-1)^{k}}
\left(P_{k,+}(\underline{'N_{\alpha,v}^*})-P_{k,-}
(\underline{'N_{\alpha,v}^*}) \right) 
\right),
\\
&\underline{\lambda_{-1}(N_{\alpha}^*)^{-1}_{\mN}}
:=\bigotimes_{v: \, r_{\alpha, v}\neq 0}
\lambda_{-h^{-v}}\left(\underline{'N_{\alpha,v}^*}\right)^{-1}_{\mN}.
\end{split}
\end{align}

It follows from \eqref{eq:1.6} that for $g\in S^1\backslash A$
we have $Y^g=Y^{S^1}$, thus
$g^v-1\neq 0$ if $\rank N_{\alpha,v}\neq 0$.

By \eqref{eq:2.080} and \eqref{eq:2.091}, we see that for any 
$\mN, \mN'>\sup_{\alpha,v}\mN_{r_{\alpha,v}, m_{\alpha}}$,
\begin{align}\label{eq:2.093} 
\left[\underline{\lambda_{-1}(N_{\alpha}^*)^{-1}_{\mN}},0\right]
=\left[\underline{\lambda_{-1}(N_{\alpha}^*)^{-1}_{\mN'}},0\right]\in
\widehat{K}_{g}^0(Y_{\alpha}^{S^1})_{I(g)}.
\end{align}
Then from (\ref{eq:2.66b}), \eqref{eq:2.088}-\eqref{eq:2.093}, for any
$\mN>\sup_{\alpha,v}\mN_{r_{\alpha,v}, m_{\alpha}}$, we have
\begin{align}\label{eq:2.094} 
\left[\underline{\lambda_{-1}(N_{\alpha}^*)},0\right]\cup
\left[\underline{\lambda_{-1}(N_{\alpha}^*)^{-1}_{\mN}},0\right]=1\in 
\widehat{K}_{g}^0(Y_{\alpha}^{S^1})_{I(g)}.
\end{align}
Summarizing, we obtain the following precise version of Theorem \ref{i06}.
A version for arithmetic $K$-group was obtained in \cite[Lemma 4.5]{KRo01}.

\begin{thm}\label{thm:2.14} 
For $g\in S^1\backslash A$,
$\left[\underline{\lambda_{-1}(N_{\alpha}^*)}, 0\right]$ is invertible in 
$\widehat{K}_{g}^0(Y_{\alpha}^{S^1})_{I(g)}$ and for any 
$\mN>\sup_{\alpha,v}\mN_{r_{\alpha,v}, m_{\alpha}}$
in \eqref{eq:2.080}, we have
\begin{align}\label{eq:2.095} 
\left[\underline{\lambda_{-1}(N_{\alpha}^*)},0\right]^{-1}=
\left[\underline{\lambda_{-1}(N_{\alpha}^*)^{-1}_{\mN}},0\right]\in 
\widehat{K}_{g}^0(Y_{\alpha}^{S^1})_{I(g)}.
\end{align}
\end{thm} 
Remark that the lower bound 
$\sup_{\alpha,v}\mN_{r_{\alpha,v}, m_{\alpha}}$
does not depend on $g\in S^1\backslash A$.

From \eqref{eq:2.042} and \eqref{eq:2.091}, for any 
$k> \mN_{r_{\alpha,v},m_{\alpha}}$,
\begin{align}\label{eq:2.096} 
\ch\left(P_{k,+}\left(\underline{'N_{\alpha,v}^*}\right)\right)
=\ch\left(P_{k,-}\left(\underline{'N_{\alpha,v}^*}\right)\right)
\in \Omega^{\bullet}(Y_{\alpha}^{S^1},\R).
\end{align}

From \eqref{eq:2.092}, we have
\begin{multline}\label{eq:2.097} 
\ch_g\left(\underline{\lambda_{-1}(N_{\alpha}^*)^{-1}_{\mN}}\right)
\\
=\prod_{v: \, r_{\alpha, v}\neq 0}
\frac{g^{vr_{\alpha,v}}}{(g^v-1)^{r_{\alpha,v}}}
\left(1+\sum_{k=1}^{\mN} \frac{(-1)^k}{(g^v-1)^{k}}
\left(\ch\left(P_{k,+}\left(\underline{'N_{\alpha,v}^*}\right)\right)
-\ch\left(P_{k,-}\left(\underline{'N_{\alpha,v}^*}\right)\right) 
\right) \right).
\end{multline}
The following corollary follows directly from 
\eqref{eq:2.085}, \eqref{eq:2.094}, \eqref{eq:2.096} 
and \eqref{eq:2.097}.
\begin{cor}\label{cor:2.15} 
For any $\mN>\sup_{\alpha,v}\mN_{r_{\alpha,v}, m_{\alpha}}$,
$g\in S^1\backslash A$,
\begin{align}\label{eq:2.098} 
\ch_g\left(\underline{\lambda_{-1}(N_{\alpha}^*)}\right)\cdot
\ch_g\left(\underline{\lambda_{-1}(N_{\alpha}^*)^{-1}_{\mN}}\right)
=1\in \Omega^{\bullet}(Y_{\alpha}^{S^1},\C).
\end{align}
\end{cor}

\section{Localization formula for equivariant 
	$\eta$-invariants}\label{s04}

In this section, 
we establish Theorems \ref{i08} and \ref{i14} 
by combining the analytic results on the 
$\eta$-invariant in Section \ref{s02}
and the algebraic framework of $g$-equivariant
differential $K$-theory in Section \ref{s0304}.
We fix $g\in S^1\backslash A$ now. It is relatively
easy to verify, by using the exact sequence on 
$g$-equivariant differential $K$-theory and
equivariant $K$-theory, that the localization
at $g\in S^1\backslash A$ of the restriction 
of $g$-equivariant differential $K$-theory from
the total manifold $Y$ to the fixed point set 
$Y^{S^1}$ is an isomorphism. We can describe 
the inverse of this map by using the direct image 
for the embedding $Y^{S^1}\hookrightarrow Y$ 
constructed in Section \ref{s0204} and the
inverse of $\lambda_{-1}(\underline{N^*})$
constructed in Section \ref{s0304}. This result, combined
 with the embedding formula of $\eta$-invariants
Theorem \ref{local63}, implies 
that the difference of the equivariant $\eta$-invariant
and its contribution on the fixed point set is
the value at this element 
of a rational function with 
integral coefficients. Note that the coefficients of these
rational functions depend a priori on the element 
$g\in S^1\backslash A$, but thanks to Theorems
\ref{local90}, \ref{local96}, we can finally conclude that
these rational functions are the same 
for any $g\in S^1\backslash A$. This ends the proof of our main result,
Theorem \ref{i14}.

This section is organized as follows.
In Section \ref{s0400}, 
we establish the localization formula
in $g$-equivariant differential $K$-theory.
 In Section \ref{s0401}, 
we define a direct image map in $g$-equivariant differential $K$-theory.
In Section \ref{s0402}, 
we state  Theorem \ref{thm:3.03},
which is a precise formulation of our main result, Theorem \ref{i14}.
In Section \ref{s0403}, 
we establish first Theorem \ref{i08} 
by applying the embedding 
formula of $\eta$-invariants, Theorem \ref{local63}, the localization
and direct image map in $g$-equivariant differential $K$-theory.
By using Theorems \ref{i08}, \ref{local90} and \ref{local96},
we get  finally Theorem \ref{thm:3.03}.
In Section \ref{s0405}, 
we study the case when $Y^{S^1}=\emptyset$ 
and compute explicitly the equivariant reduced $\eta$-invariant
when the manifold $Y$ is the circle.


\subsection{Localization in $g$-equivariant differential 
$K$-theory}\label{s0400}

Let $Y$ be a compact manifold with $S^1$-action.
We explain first the $S^1$-equivariant $K^1$-theory on $Y$, 
and the equivariant odd Chern character for an element in 
the $S^1$-equivariant $K^1$-group.

Let $K_{S^1}^{1}(Y)$ be the $S^1$-equivariant $K^{1}$-group 
of $Y$.
By \cite[Definitions 2.7 and 2.8]{Segal68}, we have the exact sequence
\begin{align}\label{eq:3.032} 
0\rightarrow K_{S^1}^{1}(Y)\overset{\varsigma}{\rightarrow} 
K_{S^1}^0(Y\times \widehat{S^1})\overset{i^*}  
{\rightarrow}K_{S^1}^0(Y)\rightarrow 0,
\end{align}
where $\widehat{S^1}$ is a copy of $S^1$ 
with trivial $S^1$-action and 
there exists $b\in\widehat{S^1} $ such that the map $i$ is given by
$i:Y\ni y\rightarrow (y,b)\in Y\times \widehat{S^1}$.
Note that
$Y\times \widehat{S^1}=Y\times \R/\Z$. We will take $b=\frac{1}{2}$
thus $i(Y)=Y\times\{\frac{1}{2}\}$. 

By \eqref{eq:3.032}, an element $y$ of $K_{S^1}^{1}(Y)$ 
can be represented as an element
$x= \varsigma(y)\in K_{S^1}^0(Y\times \widehat{S^1})$ 
such that $i^*(x)=0\in K_{S^1}^0(Y)$. 
We write $x=W-U$, where $W$ and $U$ are equivariant complex
vector bundles over $Y\times \widehat{S^1}$. 
By \cite[Proposition 2.4]{Segal68},
we may and we will assume that $U$ is a trivial vector bundle
associated with a finite dimensional $S^1$-representation
as in (\ref{eq:2.86b}). Since
$i^*(x)=W|_{Y\times\{1/2 \}}-U|_{Y\times\{1/2 \}}
=0\in K_{S^1}^0(Y)$, by adding on $U$ a trivial vector bundle
associated with a finite dimensional $S^1$-representation,
we may and we will assume that 
$W|_{Y\times\{1/2 \}}$  is a trivial vector bundle over 
$Y\times\{1/2 \}$ associated with a finite dimensional 
$S^1$-representation $M$ as in (\ref{eq:2.86b}) and
\begin{align}\label{eq:3.34b} 
U=\left(Y\times \widehat{S^1}\right)\times M.
\end{align} 
Since $[0,1]$ is contractible, there exists 
an $S^1$-equivariant morphism $F\in \cC^{\infty}(Y,\mathrm{Aut}(M))$ 
such that
\begin{align}\label{eq:3.034}
W=\big(Y\times[0,1]\big)\times  M/\sim_{F},
\end{align}
where $\sim_{F}$ is the gluing map: $(y,1,m)\sim_{F}(y,0,F(y)m)$ 
for $y\in Y$, $m\in M$. 
Then it induces an equivariant vector bundle isomorphism 
$F:E_M\rightarrow E_M$, where the $S^1$-equivariant
vector bundle $E_M$ on $Y$ is defined as in (\ref{eq:2.86b}) by
\begin{align}\label{eq:3.035} 
E_M:=Y\times M.
\end{align}

When we restrict the above construction on $Y^{S^1}$, 
as $S^1$ acts trivially on $Y^{S^1}$, we have
\begin{align}\label{eq:3.35b} 
\begin{split}
U|_{Y^{S^1}\times \widehat{S^1}}
=\Big(Y^{S^1}&\times \widehat{S^1}\Big)\times M,
\quad W|_{Y^{S^1}\times \widehat{S^1}}=\left(Y^{S^1}\times 
[0,1]\right)\times M/\sim_F,
\\ &E_M|_{Y^{S^1}}=Y^{S^1}\times M,
\end{split}
\end{align}
where $S^1$ acts only on the factor $M$.

Let $\nabla$ be an $S^1$-invariant connection on $E_M$.
Then $F^*\nabla_{\cdot}\cdot=F^{-1}\nabla_{\cdot}(F \cdot)$
is also an $S^1$-invariant connection on $E_M$.
From \eqref{eq:3.034},
\begin{align}\label{eq:3.037}
\nabla^W=dt\wedge\frac{\partial}{\partial t}+(1-t)\nabla+t 
F^*\nabla=dt\wedge\frac{\partial}{\partial t}+\nabla+t F^{-1}\nabla F
\end{align}
is a well-defined $S^1$-invariant connection on $W$ over 
$Y\times \widehat{S^1}$.

Recall that the equivariant Chern character form $\ch_g(\underline{E})$ 
and the equivariant Chern-Simons class 
$\wi{\ch}_g(\underline{E_0},\underline{E_1})$ 
defined in \eqref{e01051} and \eqref{local40}
depend only on the connections, not on the
metrics. We often denote 
the equivariant Chern character form by $\ch_g(E,\nabla^E)$
and the equivariant Chern-Simons class by 
$\wi{\ch}_g(E,\nabla^{E_0},\nabla^{E_1})$.

Let $\nabla^U$ be the trivial connection on $U$. 
It is naturally $S^1$-invariant.
For $g\in S^1\backslash A$, the odd equivariant 
Chern character for $y\in K_{S^1}^1(Y)$ as above, is defined by
\begin{multline}\label{eq:3.038}
\ch_g(y):=\left[\int_{\widehat{S^1}}
\left(\ch_g(W,\nabla^W)
-\ch_g(U,\nabla^U)\right)\right]
\\
=\left[\int_{\widehat{S^1}}\ch_g(W,\nabla^W)\right]
\in H^{\mathrm{odd}}(Y^{S^1},\C)\subset
\Omega^{\mathrm{odd}}(Y^{S^1},\C)/\Im \, d,
\end{multline} 
where the fiberwise integral $\int_{\widehat{S^1}}$ is normalized 
such that
$\int_{\widehat{S^1}}\tilde{\pi}^*\alpha\wedge\beta
=\alpha\wedge\int_{\widehat{S^1}}\beta$
for the obvious projection $\tilde{\pi}:Y^{S^1}\times \widehat{S^1}
\rightarrow Y^{S^1}$ and $\alpha\in \Omega^{\bullet}(Y^{S^1})$,
$\beta\in \Omega^{\bullet}(Y^{S^1}\times \widehat{S^1})$.
As $[\ch_g(W,\nabla^W)]
\in H^{\bullet}(Y^{S^1}\times\widehat{S^1},\C)$
does not depend on the choice of $\nabla$, thus $\ch_g(y)$ 
also does not depend on $\nabla$.  
From \eqref{e01051}, \eqref{eq:1.27}, \eqref{eq:3.037} 
and \eqref{eq:3.038}, we have
\begin{multline}\label{eq:3.039} 
\ch_g(y)
=-\left[\int_{[0,1]}\left\{\ch_g\left([0,1]\times 
E_M, dt\wedge\frac{\partial}{\partial 
	t}+(1-t)\nabla+tF^{*}\nabla\right)\right\}^{dt}dt\right]
\\
=-\widetilde{\ch}_g(E_M, \nabla, F^{*}\nabla)\in 
\Omega^{\mathrm{odd}}(Y^{S^1},\C)/\Im \, d.
\end{multline}

If we choose $\nabla$ as the trivial connection $d$ 
on $E_M$, by \eqref{eq:3.037}, the curvature $R^W$ 
of $\nabla^W$ is given by
\begin{align}\label{eq:3.040} 
R^W=\left(\nabla^W \right)^2
=dt\wedge (F^{-1}d F)-t(1-t)( F^{-1}d F)^2.
\end{align}
From \eqref{e01051}, \eqref{eq:1.27},
\eqref{eq:3.039} and \eqref{eq:3.040}, we calculate that 
\begin{align}\label{eq:3.041} 
\ch_g(y)=\sum_{n\geq 0}\frac{1}{(2i\pi 
	)^{n+1}}\frac{n!}{(2n+1)!}
\left[\tr\left[g( F^{-1}d F)^{2n+1} \right]\right].
\end{align}
This is just the equivariant version of the odd Chern character in 
\cite{Getzler93} and \cite[(1.50)]{Z01}.

From 
\eqref{eq:3.032}, $K^1_{S^1}(Y)$ is an $R(S^1)$-module. Moreover, 
$\phi\mapsto \chi_{M}(g)\cdot \phi$  makes 
$\Omega^{\mathrm{odd}}(Y^{S^1}, \C)/\Im\, d$ an $R(S^1)$-module. 

The following Proposition is the $g$-equivariant extension of the 
corresponding results in \cite[Proposition 2.20]{Bunke2009}, 
\cite[Proposition 2.24]{BunkeSchick13} and \cite[(2.21)]{FreedLott10}, 
which 
are analogues of Gillet-Soul\'e's result 
\cite[Theorem 6.2]{GS90c} in arithmetic $K$-theory.

\begin{prop}\label{prop:3.09}
	If $g\in 
	S^1\backslash A$, we have
	the exact sequence of $R(S^1)$-modules,
	\begin{align}\label{eq:3.042}
	K_{S^1}^{1}(Y)\overset{{\ch_{g}}}{\longrightarrow} 
	\Omega^{\mathrm{odd}}(Y^{S^1}, \C)/\Im\, d
	\overset{{a}}{\longrightarrow} \widehat{K}_g^0(Y) 
	\overset{{\tau}}{\longrightarrow} K_{S^1}^0(Y)
	\longrightarrow 0,
	\end{align}
	where 
	\begin{align}\label{eq:3.043}
	a(\phi)=[0, \phi],\quad \tau([\underline{E}, \phi])=[E].
	\end{align}
\end{prop}
\begin{proof}
It is obvious from Definition \ref{defn:2.12} that $\tau$ is 
surjective and $\tau\circ a=0$.
	
	If $x\in \Ker \tau$,  it is easy to 
	see from Definition \ref{defn:2.12} that $x\in \Im (a)$.
	
	Now we prove $a\circ \ch_g=0$. 
	For $y\in K_{S^1}^{1}(Y)$, we  can construct 
	equivariant vector bundle
	$E_M$ over $Y$ as in \eqref{eq:3.035}. 
	Let $h^{M}$ be the metric on $E_M$ induced by 
	an $S^1$-invariant metric on $M$ via \eqref{eq:3.035}
	and $\nabla$ be an $S^1$-invariant Hermitian 
	connection on $(E_M,h^{M})$.
	By (\ref{eq:3.35b}) and \eqref{eq:3.039}, we have
	\begin{multline}\label{eq:3.044}
	a(\ch_g(y))=\left[0,-\widetilde{\ch}_g\left(E_M, 
	\nabla|_{Y^{S^1}}, F^{*}\nabla|_{Y^{S^1}}\right)\right]
	=[0,-\widetilde{\ch}_g(E_M, \nabla, F^{*}\nabla)]
	\\
	=[(E_M, h^{M}, \nabla), 0]- [(E_M, h^{M},
	\nabla), \widetilde{\ch}_g((E_M, h^{M}, \nabla), 
	(E_M, F^*h^{M}, F^{*}\nabla))].
	\end{multline}
	By Definition \ref{defn:2.12}, 
	$((E_M, h^{M}, \nabla), 0)$ and $((E_M, h^{M}, 
	\nabla), \widetilde{\ch}_g((E_M, h^{M}, \nabla), 
	(E_M, F^*h^{M}, F^{*}\nabla)))$ are equivalent 
	under the equivariant vector bundle 
	isomorphism $F$ over $Y$. That is, 
	$a(\ch_g(y))=0\in \widehat{K}_g^0(Y)$.
	
	At last, we prove $\Ker a\subseteq\Im \ch_g$.
	For $\phi'\in \Ker a$,  i.e., $[0,\phi']=0\in \widehat{K}_g^0(Y)$.
By Definition \ref{defn:2.12}, there exists 
an equivariant geometric triple 
$\underline{E'}=(E', h^{E'}, \nabla^{E'})$  
and an equivariant vector bundle isomorphism over $Y$:
\begin{align}\label{eq:3.045} 
\Phi':E'\rightarrow E'\quad \text{such that}
\quad \phi'=-\widetilde{\ch}_g\left(\underline{E'}, 
\Phi'^{*} \underline{E'}\right).
\end{align}
By \cite[Proposition 2.4]{Segal68}, there exists an $S^1$-vector bundle
$E$ on $Y$ such that $E\oplus E'=E_M$ 
where $M$ is a finite dimensional $S^1$-representation. Set 
\begin{align}\label{eq:3.046} 
\Phi: E\oplus E'\rightarrow E\oplus E',\quad (u,v)\mapsto (u, \Phi'(v)).
\end{align}
Let $h^E$ be an $S^1$-invariant Hermitian metric on $E$ 
and $\nabla^E$ be an $S^1$-invariant Hermitian connection
on $(E,h^E)$. Then 
$\widetilde{\ch}_g\left(\underline{E}\oplus\underline{E'}, \Phi^{*} 
	\left(\underline{E}\oplus\underline{E'}\right)\right)
	=\widetilde{\ch}_g\left(\underline{E'}, 
	\Phi'^{*} \underline{E'}\right)$.
	As in \eqref{eq:3.037}, 
	\begin{align}\label{eq:3.047} 
	\nabla^W=dt\wedge\frac{\partial}{\partial 
		t}+(1-t)(\nabla^{E}\oplus \nabla^{E'})+t 
	\Phi^*(\nabla^{E}\oplus \nabla^{E'})
	\end{align}
	is an $S^1$-invariant connection on 
	$W= Y\times[0,1]\times M/\sim_{\Phi}$. 
	Therefore, by \eqref{local40}, \eqref{eq:3.037} and \eqref{eq:3.039}, 
	modulo exact forms, we 	have
	\begin{multline}\label{eq:3.048}
	-\widetilde{\ch}_g\left(\underline{E}\oplus\underline{E'}, \Phi^{*} 
	\left(\underline{E}\oplus\underline{E'}\right)\right)	
	=-\widetilde{\ch}_g\left(E\oplus E', (\nabla^{E}\oplus 
	\nabla^{E'})|_{Y^{S^1}},
	\Phi^{*}(\nabla^{E}\oplus \nabla^{E'})|_{Y^{S^1}}\right)
	\\
	=\int_{\widehat{S^1}}\ch_g(W, \nabla^W).
	\end{multline}
From (\ref{eq:3.34b}), \eqref{eq:3.038}, \eqref{eq:3.045}, 
\eqref{eq:3.046} and \eqref{eq:3.048},
$\Phi$ defines an element $y\in K_{S^1}^1(Y)$
by $\varsigma(y)=W-U$ and $\phi'=\ch_g(y)$.
	
	It is obvious that the $R(S^1)$-action commutes with $\ch_g$, $a$ 
	and $\tau$. 
	
	The proof of Proposition \ref{prop:3.09} is completed.
\end{proof}

Let $\iota:Y^{S^1}\rightarrow Y$ be the canonical embedding. Let 
\begin{align}\label{eq:3.17b} 
\hat{\iota}^*:\widehat{K}^0_g(Y)_{I(g)}\rightarrow 
\widehat{K}^0_g(Y^{S^1})_{I(g)}, \quad
[\underline{E},\phi]/\chi\to [\underline{E}|_{Y^{S^1}}, \phi]/\chi,
\end{align}
be the induced homomorphism by restriction.

The following localization theorem, which is the differential $K$-theory 
version of the Atiyah-Segal localization theorem 
in topological $K$-theory \cite[Theorem 1.1]{ASegal68},
is inspired by \cite[Theorem 3.27]{BunkeSchick13}
and \cite[Theorem 5.5]{Ta12}.
\begin{thm}[Localization Theorem]\label{thm:3.10} 
	For $g\in S^1\backslash A$, the restriction map
	$\hat{\iota}^*:\widehat{K}_g^0(Y)_{I(g)}\rightarrow 
	\widehat{K}_g^0(Y^{S^1})_{I(g)}$ in (\ref{eq:3.17b})
	is an $R(S^1)_{I(g)}$-module isomorphism.
\end{thm} 
\begin{proof}
	Since localization preserves exact sequences 
	\cite[Proposition 3.3]{AM69}, from Proposition \ref{prop:3.09}, 
	we get an exact sequence of $R(S^1)_{I(g)}$-modules
	\begin{align}\label{eq:3.050}
	K_{S^1}^{1}(Y)_{I(g)}
	\overset{{\ch_{g}}}{\longrightarrow}
	\left(\Omega^{\mathrm{odd}}(Y^{S^1}, \C)/\Im\, d\right)_{I(g)}
	\overset{{a}}{\longrightarrow} 
	\widehat{K}_g^0(Y)_{I(g)} 
	\overset{{\tau}}{\longrightarrow} 
	K_{S^1}^0(Y)_{I(g)}
	\longrightarrow 0.
	\end{align}
	Replacing $Y$ by $Y^{S^1}$, since $(Y^{S^1})^{S^1}=Y^{S^1}$, 
	we get an exact 
	sequence of $R(S^1)_{I(g)}$-modules
	\begin{align}\label{eq:3.051}
	K_{S^1}^{1}(Y^{S^1})_{I(g)}
	\overset{{\ch_{g}}}{\longrightarrow}
	\left(\Omega^{\mathrm{odd}}(Y^{S^1}, \C)/\Im\, d\right)_{I(g)}
	\overset{{a}}{\longrightarrow} 
	\widehat{K}_g^0(Y^{S^1})_{I(g)} 
	\overset{{\tau}}{\longrightarrow} 
	K_{S^1}^0(Y^{S^1})_{I(g)}
	\longrightarrow 0.
	\end{align}
	Furthermore, we have the commutative diagram
		\begin{equation}\label{eq:3.052} 
		\begin{split}
	\xymatrix{
		K_{S^1}^{1}(Y)_{I(g)} \ar[d]^{\iota^*} \ar[r]^-{\ch_g}
		&\left(\Omega^{\mathrm{odd}}(Y^{S^1}, \C)/\Im\, d\right)_{I(g)}
		\ar[d]^{\Id} \ar[r]^-{a}
		&\widehat{K}_g^0(Y)_{I(g)}\ar[d]^{\hat{\iota}^*} 
		\ar[r]^-{\tau} &K_{S^1}^0(Y)_{I(g)}\ar[d]^{\iota^*} \ar[r] &0\\
		K_{S^1}^{1}(Y^{S^1})_{I(g)}  \ar[r]^-{\ch_g}
		&\left(\Omega^{\mathrm{odd}}(Y^{S^1}, \C)/\Im\, d\right)_{I(g)}
		\ar[r]^-{a} 
		&\widehat{K}_g^0(Y^{S^1})_{I(g)}\ar[r]^-{\tau} 
		&K_{S^1}^0(Y^{S^1})_{I(g)}\ar[r] &0.
	}
\end{split}
	\end{equation}
	Here $\iota^*:K_{S^1}^*(Y)_{I(g)}
	\rightarrow K_{S^1}^*(Y^{S^1})_{I(g)}$
	is the $R(S^1)_{I(g)}$-module map induced by $\iota$.
	From (\ref{eq:3.032}) for $Y$ and $Y^{S^1}$, 
	we have the commutative diagram
	\begin{align}\label{eq:3.52b} 
	\begin{split}
	\xymatrix{
	0\ar[r]	&K_{S^1}^{1}(Y)_{I(g)} \ar[d]^{\iota^*} \ar[r]^-{\varsigma} 
	&K_{S^1}^0(Y\times \widehat{S^1})_{I(g)}\ar[d]^{\iota^*} 
	\ar[r]^-{i^*} &K_{S^1}^0(Y)_{I(g)}\ar[d]^{\iota^*} \ar[r] &0\\
	0\ar[r]	&K_{S^1}^{1}(Y^{S^1})_{I(g)}  \ar[r]^-{\varsigma} 
	&K_{S^1}^0(Y^{S^1}\times \widehat{S^1})_{I(g)}\ar[r]^-{i^*} 
	&K_{S^1}^0(Y^{S^1})_{I(g)}\ar[r] &0.
}
\end{split}
	\end{align}
	Using localization in topological $K$-theory 
	\cite[Theorem 1.1]{ASegal68}, $\iota^*$ is 
	an isomorphism on $K_{S^1}^0(\cdot)_{I(g)}$.
By the five lemma on (\ref{eq:3.52b}), 
$\iota^*$ is an isomorphism on $K_{S^1}^{1}(\cdot)_{I(g)}$.
	Then by the five lemma on (\ref{eq:3.052}), $\hat{\iota}^*$ in 
	(\ref{eq:3.052}) is an 
	isomorphism.
	
	The proof of Theorem \ref{thm:3.10} is completed.
\end{proof}

As the restriction map $\hat{\iota}^*$ in (\ref{eq:3.17b})
is an isomorphism, it is a natural question to 
find explicitly its inverse. We solve this problem
by combining the construction of the geometric
direct image for embeddings in Section \ref{s0204} and 
the invertibility of the element $[\lambda_{-1}
(\underline{N^*}), 0]$ in $\widehat{K}^0_g(Y^{S^1})_{I(g)}$ 
obtained in Theorem \ref{thm:2.14}.

In the following definition 
we adopt  
the notation in Section \ref{s0204}.
\begin{defn}\label{defn:3.11}
	For $g\in S^1\backslash A$, the direct image map 
	\begin{align}\label{eq:3.053}
	\hat{\iota}_!:\widehat{K}^0_g(Y^{S^1})_{I(g)}\rightarrow 
	\widehat{K}^0_g(Y)_{I(g)}
	\end{align}
	is defined by 
	\begin{align}\label{eq:3.054}
	\hat{\iota}_!\left(\left[\underline{\mu},\phi\right]/\chi
	\right)=\left[\underline{\xi_+}, 
	\ch_g\left(\underline{\Lambda^{\mathrm{even}}(N^*)}\right)\wedge 
	\phi\right]/\chi-\left[\underline{\xi_-}, 
	\ch_g\left(\underline{\Lambda^{\mathrm{odd}}(N^*)}
	\right)\wedge \phi\right]/\chi.
	\end{align}
\end{defn}

\begin{thm}\label{thm:3.12} 
The direct image map $\hat{\iota}_!$ is a well-defined isomorphism and 
	\begin{align}\label{eq:3.055} 
	\hat{\iota}^*\circ 
	\hat{\iota}_!=\left[\underline{\lambda_{-1}(N^*)},0\right]\cup:
	\widehat{K}_g^0(Y^{S^1})_{I(g)}
	\overset{{\sim}}{\longrightarrow}
	\widehat{K}_g^0(Y^{S^1})_{I(g)}.
	\end{align}
Thus the inverse map of $\hat{\iota}^*:
 \widehat{K}^0_g(Y)_{I(g)}\rightarrow 
\widehat{K}^0_g(Y^{S^1})_{I(g)}$ in (\ref{eq:3.17b}) is given by
$\hat{\iota}_!\circ \left[\underline{\lambda_{-1}(N^*)},0\right]^{-1}\cup$.
\end{thm}
\begin{proof}
	From the construction of $\underline{\xi_{\pm}}$ 
	in \eqref{eq:1.40a}-\eqref{bl0037}, we have
	\begin{multline}\label{eq:3.056} 
	\hat{\iota}^*\left\{\left[\underline{\xi_+}, 
	\ch_g\left(\underline{\Lambda^{\mathrm{even}}(N^*)}\right)\wedge 
	\phi\right]/\chi-\left[\underline{\xi_-}, 
	\ch_g\left(\underline{\Lambda^{\mathrm{odd}}(N^*)}\right)\wedge 
	\phi\right]/\chi \right\}
	\\
	=\sum_{\alpha\in \mathfrak{B}}\left[\left(\underline{\Lambda^{
			\mathrm{even}}(N_{\alpha}^*)}\otimes 
	\underline{\mu_{\alpha}}\right)\oplus \underline{F_{\alpha}}, 
	\ch_g\left(\underline{\Lambda^{\mathrm{even}}(N_{\alpha
		}^*)}\right)
	\wedge \phi\right]/\chi
	\\
	-\sum_{\alpha\in \mathfrak{B}}\left[\left(\underline{
		\Lambda^{\mathrm{odd}}(N_{\alpha}^*)}
	\otimes \underline{\mu_{\alpha}}\right)\oplus 
	\underline{F_{\alpha}}, \ch_g\left(\underline{
		\Lambda^{\mathrm{odd}}(N_{\alpha}^*)}\right)
	\wedge \phi\right]/\chi
	\\
	=\left[\underline{\Lambda^{\mathrm{even}}(N^*)}\otimes
	\underline{\mu}, 
	\ch_g\left(\underline{\Lambda^{\mathrm{even}}(N^*)}\right)
	\wedge \phi\right]/\chi
	\\
	-\left[\underline{\Lambda^{\mathrm{odd}}(N^*)}
	\otimes \underline{\mu}, 
	\ch_g\left(\underline{\Lambda^{\mathrm{odd}}(N^*)}\right)
	\wedge \phi\right]/\chi.
	\end{multline}
	Since 
	$\underline{\lambda_{-1}(N^*)}
	=\underline{\Lambda^{\mathrm{even}}(N^*)}
	-\underline{\Lambda^{\mathrm{odd}}(N^*)}$, by \eqref{eq:2.084} 
	and \eqref{eq:3.056}, we have
	\begin{multline}\label{eq:3.057}
	\left[\underline{\lambda_{-1}(N^*)},0\right]
	\cup[\underline{\mu},\phi]/\chi
	\\
	=\hat{\iota}^*\left\{\left[\underline{\xi_+}, 
	\ch_g\left(\underline{\Lambda^{\mathrm{even}}(N^*)}\right)\wedge 
	\phi\right]/\chi-\left[\underline{\xi_-}, 
	\ch_g\left(\underline{\Lambda^{\mathrm{odd}}(N^*)}\right)
	\wedge \phi\right]/\chi \right\}.
	\end{multline}	
	Therefore, we have
	\begin{align}\label{eq:3.058}
	\hat{\iota}^*\circ \hat{\iota}_!
	=\left[\underline{\lambda_{-1}(N^*)},0\right]\cup.
	\end{align}
	
By Theorem \ref{thm:3.10}, 
$\hat{\iota}^*:\widehat{K}_g^0(Y)_{I(g)}\overset{{\sim}}
{\longrightarrow}\widehat{K}_g^0(Y^{S^1})_{I(g)}$
is an $R(S^1)_{I(g)}$-module isomorphism. 
From Theorem \ref{thm:2.14}, 
	\begin{align}\label{eq:3.059}
	(\hat{\iota}^*)^{-1}\circ \left[\underline{\lambda_{-1}(N^*)},
	0\right]\cup:
	\widehat{K}_g^0(Y^{S^1})_{I(g)}	\overset{{\sim}}{\longrightarrow}
	\widehat{K}_g^0(Y)_{I(g)}
	\end{align}
	is a well-defined isomorphism. 
	Equations \eqref{eq:3.058} and \eqref{eq:3.059} imply that 
	$\hat{\iota}_!$ in \eqref{eq:3.054}
	is a well-defined isomorphism and \eqref{eq:3.055} holds.
\end{proof}	

\subsection{Direct image in $g$-equivariant differential $K$-theory}
\label{s0401}
In the remainder of this section,  $Y$ is an odd dimensional
compact oriented manifold
and  has an $S^1$-equivariant spin$^c$ structure.

 Note that for $g\in S^1$, $\Q_g\subset \C$ was defined in \eqref{i05}
and $\ch_g(R(S^1)_{I(g)})=\Q_g$.

\begin{defthm}\label{defn:3.01}
	Let $g\in S^1$ be fixed.
	For an equivariant geometric triple $\underline{E}$, 
$\phi\in \Omega^{\mathrm{odd}}(Y^g,\C)/\Im \, d$, $\chi\in R(S^1)$
	such that $\chi(g)\neq 0$, the map
	\begin{align}\label{eq:3.001}
	\widehat{f_Y}_!\big((\underline{E}, 
	\phi)/\chi\big)
	:=\chi
	(g)^{-1}\left(-\int_{Y^g}\td_g(\nabla^{TY},\nabla^L)\wedge\phi
	+\bar{\eta}_g(\underline{TY}, \underline{L}, \underline{E})\right)
	\end{align}
defines a direct image map 
$\widehat{f_Y}_!: \widehat{K}_g^0(Y)_{I(g)}
\rightarrow \C/\Q_g$.
\end{defthm}

Note that  
for $g=1$, 
the family version of (\ref{eq:3.001}) is \cite[Definition 3.12]{FreedLott10}. 
In \cite[Proposition 4.3]{KRo01} 
K\"ohler-Roessler defined an arithmetic $K$-theory version of (\ref{eq:3.001}).

\begin{proof}	
	For an $S^1$-equivariant vector bundle isomorphism 
	$\Phi:E\rightarrow E$ over $Y$, 
	we have by Definition \ref{local128},
	\begin{align}\label{eq:3.002} 
	\bar{\eta}_g(\underline{TY}, \underline{L}, \Phi^{*}\underline{E})
	=\bar{\eta}_g(\underline{TY}, \underline{L}, 
	\underline{E}).
	\end{align}
For any finite dimensional $S^1$-representation $M$ 
	and triples $\underline{E}$, $\underline{E_1}$, $\underline{E_2}$, 
	we have from Definition \ref{local128},
	\begin{align}\label{eq:3.003} \begin{split}
&	\bar{\eta}_g(\underline{TY}, \underline{L}, \underline{M}\otimes 
	\underline{E})=\chi_{M}(g)\cdot\bar{\eta}_g(\underline{TY}, 
	\underline{L}, \underline{E}),\\
&	\bar{\eta}_g(\underline{TY}, \underline{L}, \underline{E_1}
	\oplus\underline{E_2})
	=\bar{\eta}_g(\underline{TY}, \underline{L}, 
	\underline{E_1})+\bar{\eta}_g(\underline{TY}, \underline{L}, 
	\underline{E_2}).
	\end{split}\end{align}

For cycles $(\underline{E_1}, 
\phi_1)/\chi_1$ and $(\underline{E_2}, 
\phi_2)/\chi_2$ of  $ \widehat{K}_g^0(Y)_{I(g)}$ we have
from \eqref{eq:3.001} and \eqref{eq:3.003}, 
\begin{align}\label{eq:3.005} 
\widehat{f_Y}_!\big((\underline{E_1}, 
\phi_1)/\chi_1+(\underline{E_2}, 
\phi_2)/\chi_2\big)
=\widehat{f_Y}_!\big((\underline{E_1}, 
\phi_1)/\chi_1\big)+\widehat{f_{Y}}_!\big((\underline{E_2}, 
\phi_2)/\chi_2\big).
\end{align}

	If $[\underline{E_2}-\underline{E_1}, 
	\phi]/\chi=0\in \widehat{K}_g^0(Y)_{I(g)}$, there exists a finite 
	dimensional $S^1$-representation $M$ such that 
	$[\underline{M}\otimes(\underline{E_2}-\underline{E_1}), 
	\chi_M(g)\phi]=0\in \widehat{K}_g^0(Y)$ and $\chi_M(g)\neq 0$. 
	Thus from Definition \ref{defn:2.12}, there exist $\underline{E_3}$
	and an equivariant vector bundle isomorphism 
	$\Phi:(M\otimes E_1)\oplus 
	E_3\rightarrow (M\otimes E_2)\oplus E_3$ such that
	\begin{align}\label{eq:3.006} 
	\phi=\chi_M(g)^{-1}\widetilde{\ch}_g((\underline{M}
	\otimes\underline{E_1})\oplus\underline{E_3},
	\Phi^*((\underline{M}\otimes\underline{E_2})\oplus\underline{E_3})).
	\end{align}
	From the variation formula \eqref{local73}, \eqref{eq:3.002} 
	and \eqref{eq:3.003}, there exists 
$\alpha_g\in\Z[g, g^{-1}]:=\{f(g)\in \C: f\in \Z[x, x^{-1}] \}$ such that
	\begin{multline}\label{eq:3.004} 
	\bar{\eta}_g(\underline{TY}, \underline{L}, \underline{M}
	\otimes\underline{E_2})
	-\bar{\eta}_g(\underline{TY}, \underline{L}, 
	\underline{M}\otimes\underline{E_1})
	\\
	=\bar{\eta}_g(\underline{TY}, \underline{L},
	(\underline{M}\otimes\underline{E_2})\oplus\underline{E_3})
	-\bar{\eta}_g(\underline{TY}, \underline{L}, 
	(\underline{M}\otimes\underline{E_1})\oplus\underline{E_3})
	\\
	=\bar{\eta}_g(\underline{TY}, \underline{L}, 
	\Phi^*((\underline{M}\otimes\underline{E_2})\oplus\underline{E_3}))
	-\bar{\eta}_g	(\underline{TY}, \underline{L},
	(\underline{M}\otimes\underline{E_1})\oplus\underline{E_3})
	\\
	=\int_{Y^g}\td_g(\nabla^{TY},\nabla^L)
	\widetilde{\ch}_g\Big((\underline{M}\otimes\underline{E_1})
	\oplus\underline{E_3},
	\Phi^*((\underline{M}\otimes\underline{E_2})
	\oplus\underline{E_3})\Big)
	+\alpha_g.
	\end{multline}  
From \eqref{eq:3.001}, \eqref{eq:3.003}, 
\eqref{eq:3.006} and \eqref{eq:3.004}, we have
	\begin{multline}\label{eq:3.007} 
\widehat{f_Y}_!\big((\underline{E_2}-\underline{E_1}, 
\phi)/\chi\big)
\\
=\chi
(g)^{-1}\left[\bar{\eta}_g(\underline{TY}, 
\underline{L}, \underline{E_2})
-\bar{\eta}_g(\underline{TY}, \underline{L}, 
\underline{E_1})-\int_{Y^g}\td_g(\nabla^{TY},\nabla^L)\phi \right]
\\
=\chi
(g)^{-1}\chi_M(g)^{-1}\Big[\bar{\eta}_g(\underline{TY}, \underline{L},
\underline{M}\otimes\underline{E_2})-\bar{\eta}_g(\underline{TY},
\underline{L}, 
\underline{M}\otimes\underline{E_1})
\\
-\left.\int_{Y^g}\td_g(\nabla^{TY},\nabla^L)
\widetilde{\ch}_g\Big((\underline{M}\otimes\underline{E_1})
\oplus\underline{E_3},\Phi^*((\underline{M}\otimes
\underline{E_2})\oplus\underline{E_3})\Big) \right]
\\
=\chi(g)^{-1}\chi_M(g)^{-1}\cdot\alpha_g\in \Q_g.
	\end{multline}
	
The proof of Theorem \ref{defn:3.01} is completed.
\end{proof}

\subsection{Main result: Theorem  \ref{i14}}\label{s0402}

Recall that the orientation of $Y_{\alpha}^{S^1}$ is given in 
Section \ref{s0203}.
From \eqref{eq:2.092}, there exist equivariant geometric triples 
$\underline{\mu_{\alpha,\mN,+}}$ and 
$\underline{\mu_{\alpha,\mN,-}}$ such that 
\begin{align}\label{eq:3.008} 
\underline{\lambda_{-1}(N_{\alpha}^*)^{-1}_{\mN}}=
\prod_{v:\,  r_{\alpha,v}\neq 0}(h^v-1)^{-r_{\alpha,v}-\mN}\Big(
\underline{\mu_{\alpha,\mN,+}}-\underline{\mu_{\alpha,\mN,-}}\Big).
\end{align}
In \eqref{eq:3.008}, we identify $f(h)\cdot \underline{F}$
with $\underline{M_f}\otimes \underline{F}$  for triple $\underline{F}$, 
$f\in \Z[x]$ and virtual $S^1$-representation $M_f$ associated with $f$.
For $g\in S^1\backslash A$, we define
\begin{multline}\label{eq:3.009} 
\bar{\eta}_g\left(\underline{TY_{\alpha}^{S^1}}, \underline{L_{\alpha}},
\underline{\lambda_{-1}(N_{\alpha}^*)^{-1}_{\mN}}\otimes 
\underline{E}|_{Y^{S^1}_{\alpha}}\right)
\\
=\prod_{v:N_{\alpha,v}\neq 0} (g^v-1)^{-r_{\alpha,v}-\mN}
\left[\bar{\eta}_g\left(\underline{TY_{\alpha}^{S^1}},
\underline{L_{\alpha}},\underline{\mu_{\alpha,\mN,+}}\otimes 
\underline{E}|_{Y^{S^1}_{\alpha}}\right)\right.\\
\left. -\bar{\eta}_g\left(\underline{TY_{\alpha}^{S^1}}, 
\underline{L_{\alpha}},\underline{\mu_{\alpha,\mN,-}}\otimes 
\underline{E}|_{Y^{S^1}_{\alpha}}\right)\right].
\end{multline}

\begin{rem}\label{rem:3.02} 
Note that from \eqref{eq:2.092} and \eqref{eq:3.008}, 
\begin{align}\label{eq:3.010}
\mu_{\alpha,\mN,\pm}
=\bigoplus_{k\geq 0}\xi_{\alpha,k,\pm}\in K_{S^1}^0(Y_{\alpha}^{S^1})
\end{align}
and $S^1$ acts on $\xi_{\alpha,k}$ with weight $k$. 
If $S^1$ acts on $L$ by sending $g\in S^1$ to 
$g^{l_{\alpha}}$ ($l_{\alpha}\in \Z$) on
 $Y_{\alpha}^{S^1}$, then by \cite[p139]{LMZ03} and (\ref{eq:2.089}),
 \begin{align}\label{eq:3.011}
\sum_v v\,r_{\alpha,v}+l_{\alpha}=0 \quad\,\mathrm{mod}\,   (2).
 \end{align}
Now by \eqref{local266}, \eqref{eq:1.33},
\eqref{eq:3.010} and \eqref{eq:3.011},  for $g\in S^{1}$, we have
 \begin{multline}\label{eq:3.012}
\bar{\eta}_g\left(\underline{TY_{\alpha}^{S^1}},
\underline{L_{\alpha}},\underline{\mu_{\alpha,\mN,+}}\otimes 
\underline{E}|_{Y^{S^1}_{\alpha}}\right)
-\bar{\eta}_g\left(\underline{TY_{\alpha}^{S^1}}, 
\underline{L_{\alpha}},\underline{\mu_{\alpha,\mN,-}}\otimes 
\underline{E}|_{Y^{S^1}_{\alpha}}\right)
\\
=g^{-\frac{1}{2}\sum_v v\, r_{\alpha,v}
+\frac{1}{2}l_{\alpha}}\sum_{k\geq 0, v} 
g^{k+v}\left[\bar{\eta}\left(\underline{TY_{\alpha}^{S^1}},
\underline{L_{\alpha}},\underline{\xi_{\alpha,k,+}}
\otimes \underline{E}_v\right)\right.\\
\left. -\bar{\eta}\left(\underline{TY_{\alpha}^{S^1}},
\underline{L_{\alpha}},\underline{\xi_{\alpha,k,-}}
\otimes \underline{E}_v\right)\right].
\end{multline}
\end{rem}

From (\ref{eq:2.080}) and (\ref{eq:2.095}), set
\begin{align}\label{eq:3.12b} 
\mN_0=\sup_{\alpha,v}\mN_{r_{\alpha,v}, m_{\alpha}}.
\end{align}

Now we state our main result of this paper, 
which is a precise formulation 
of Theorem \ref{i14}.
\begin{thm}\label{thm:3.03} 
	For any $\mN, \mN'\in \N$
	and $\mN'> \mN>\mN_0$, for any equivariant geometric 
	triple $\underline{E}$ on $Y$, 
	the functions on $S^1\backslash A$,
	\begin{multline}\label{eq:3.013} 
	P_{\mN,\mN'}(g):=\bar{\eta}_g\left(\underline{TY_{\alpha}^{S^1}}, 
	\underline{L_{\alpha}},
	\underline{\lambda_{-1}(N_{\alpha}^*)^{-1}_{\mN'}}\otimes 
	\underline{E}|_{Y^{S^1}_{\alpha}}\right)
	\\
	-\bar{\eta}_g\left(\underline{TY_{\alpha}^{S^1}}, 
	\underline{L_{\alpha}},
	\underline{\lambda_{-1}(N_{\alpha}^*)^{-1}_{\mN}}\otimes 
	\underline{E}|_{Y^{S^1}_{\alpha}}\right)
	\end{multline}
	and
	\begin{align}\label{eq:3.014} 
	Q_{\mN}(g):=\bar{\eta}_g\left(\underline{TY}, \underline{L},
	\underline{E}\right)-\sum_{\alpha}\bar{\eta}_g
	\left(\underline{TY_{\alpha}^{S^1}}, \underline{L_{\alpha}},
	\underline{\lambda_{-1}(N_{\alpha}^*)^{-1}_{\mN}}\otimes 
	\underline{E}|_{Y^{S^1}_{\alpha}}\right)
	\end{align}
	 are restrictions of
	 rational functions on $S^1$ with integral coefficients 
	which do not have poles on $S^1\backslash A$. 	
\end{thm}

\subsection{A proof of Theorems \ref{i08} and 
	\ref{thm:3.03}}\label{s0403}

Let $\widehat{f_{Y^{S^1}}}_!$ be the direct image map
$\widehat{f_{Y^{S^1}}}_!: 
\widehat{K}_g^0(Y^{S^1})_{I(g)}\rightarrow 
\C/{\Q_g}$ defined in Definition \ref{defn:3.01}. 
Explicitly, for any $[\underline{E}, 
\phi]/\chi\in \widehat{K}_g^0(Y^{S^1})_{I(g)}$,
\begin{multline}\label{eq:3.018}
\widehat{f_{Y^{S^1}}}_!\big([\underline{E}, 
\phi]/\chi\big)
:=\chi (g)^{-1}\sum_{\alpha}\Big[-\int_{Y_{\alpha}^{S^1}}
\td_g(\nabla^{TY_{\alpha}^{S^1}},\nabla^{L_{\alpha}})
\wedge \phi\\
+\bar{\eta}_g\big(\underline{TY_{\alpha}^{S^1}}, 
\underline{L_{\alpha}}, 
\underline{E}\big)\Big]\mod \Q_g.
\end{multline} 

\begin{lemma}\label{lem:3.04} 
	For any $\mN, \mN'\in \N$ and 
	$ \mN>\mN'>\mN_0$,  $g\in S^1\backslash A$, 
	we have $P_{\mN,\mN'}(g)\in \Q_g$.
\end{lemma}
\begin{proof} From \eqref{eq:2.093} and (\ref{eq:3.12b}), for any 
	$\mN, \mN'>\mN_0$, 
	$g\in S^1\backslash A$,
	\begin{align}\label{eq:3.015}
	\left[\underline{\lambda_{-1}(N_{\alpha}^*)^{-1}_{\mN}}\otimes 
	\underline{E}|_{Y^{S^1}_{\alpha}},0\right]
	=\left[\underline{\lambda_{-1}(N_{\alpha}^*)^{-1}_{\mN'}}\otimes 
	\underline{E}|_{Y^{S^1}_{\alpha}},0\right]
	\in \widehat{K}_{g}^0(Y_{\alpha}^{S^1})_{I(g)}.
	\end{align}
	Thus Lemma \ref{lem:3.04} follows directly from Definition and 
	Theorem \ref{defn:3.01}.
\end{proof}

Observe that
\begin{align}\label{eq:3.016} 
\widehat{K}_g^0(Y^{S^1})_{I(g)}=\bigoplus_{\alpha\in \mathfrak{B}}
\widehat{K}_g^0(Y_{\alpha}^{S^1})_{I(g)}.
\end{align}
From Theorem \ref{thm:2.14}, for $g\in S^1\backslash A$,
$\left[\underline{\lambda_{-1}(N_{\alpha}^*)}, 0\right]$ is invertible in 
$\widehat{K}_{g}^0(Y_{\alpha}^{S^1})_{I(g)}$.
We denote by
\begin{align}\label{eq:3.017} 
\left[\underline{\lambda_{-1}(N^*)},0\right]^{-1}=\bigoplus_{\alpha}
\left[\underline{\lambda_{-1}(N_{\alpha}^*)},0\right]^{-1}\in 
\widehat{K}_g^0(Y^{S^1})_{I(g)}.
\end{align}

\begin{proof}[Proof of Theorem \ref{i08}]
	The first part of Theorem \ref{i08} is Theorem \ref{defn:3.01}.
From Theorem \ref{thm:3.10}, for $g\in S^1\backslash A$, the
restriction map $\hat{\iota}^*$ in (\ref{eq:3.17b})
is an $R(S^1)_{I(g)}$-module isomorphism.

For $g\in S^1\backslash A$, by Theorem \ref{local63},  \eqref{local76}, 
\eqref{eq:3.054},	\eqref{eq:3.001}  and \eqref{eq:3.018}, for any 
	$[\underline{\mu},\phi]/\chi\in\widehat{K}^0_g(Y^{S^1})_{I(g)}$, 
	we have
	\begin{multline}\label{eq:3.060}
	\widehat{f_{Y}}_!\circ \hat{\iota}_!([\underline{\mu},\phi]/\chi)
	\\
	=\widehat{f_{Y}}_!\bigg\{\left[\underline{\xi_+}, 
	\ch_g\left(\underline{\Lambda^{\mathrm{even}}(N^*)}\right)\wedge 
	\phi\right]/\chi-\left[\underline{\xi_-}, 
	\ch_g\left(\underline{\Lambda^{\mathrm{odd}}(N^*)}
	\right)\wedge \phi\right]/\chi\bigg\}
	\\
	=\chi(g)^{-1}\left\{-\int_{Y^{S^1}}\td_g(\nabla^{TY},\nabla^L)
	\ch_g\left(\underline{\lambda_{-1}(N^*)}\right)\wedge
	\phi+\bar{\eta}_g(\underline{TY}, \underline{L}, \underline{\xi_+})
	-\bar{\eta}_g(\underline{TY}, \underline{L}, 
	\underline{\xi_-})\right\}
	\\
	=\chi (g)^{-1}\sum_{\alpha}\Big(-\int_{Y_{\alpha}^{S^1}}
	\td_g(\nabla^{TY_{\alpha}^{S^1}},\nabla^{L_{\alpha}})\wedge \phi\
	+\bar{\eta}_g(\underline{TY_{\alpha}^{S^1}},
	\underline{L_{\alpha}}, \underline{\mu})\Big)
	\,\mathrm{mod}\,{\Q_g}
	\\
	=\widehat{f_{Y^{S^1}}}_!([\underline{\mu},\phi]/\chi).
	\end{multline}
It means that
	\begin{align}\label{eq:3.061}
	\widehat{f_{Y}}_!\circ \hat{\iota}_!=\widehat{f_{Y^{S^1}}}_!:
	\widehat{K}^0_g(Y^{S^1})_{I(g)}\rightarrow \C/\Q_g.
	\end{align}
From Theorems \ref{thm:2.14}, \ref{thm:3.12} 
and \eqref{eq:3.061}, we have
	\begin{align}\label{eq:3.062} 
	\widehat{f_{Y}}_!=\widehat{f_{Y^{S^1}}}_!\circ 
	[\underline{\lambda_{-1}(N^*)},0]^{-1}\cup\ \hat{\iota}^*:
	\widehat{K}^0_g(Y)_{I(g)}\rightarrow \C/\Q_g.
	\end{align}
	Thus the diagram in \eqref{eq:0.14} commutes.

By Definition \ref{defn:3.01}, we have
\begin{align}\label{eq:3.020}
\widehat{f_{Y}}_!([\underline{E},0])=\bar{\eta}_g(\underline{TY}, 
\underline{L}, 
\underline{E}) \quad \mathrm{mod}\, \Q_g.
\end{align}
By Theorems \ref{thm:2.14}, \ref{defn:3.01}, (\ref{eq:3.12b}),
 \eqref{eq:3.018} 
and \eqref{eq:3.017}, for any  $\mN>\mN_0$, 
$g\in S^1\backslash A$, we get
\begin{multline}\label{eq:3.021}
\widehat{f_{Y^{S^1}}}_!\left([\underline{\lambda_{-1}(N^*)},0]^{-1}
\cup\,\hat{\iota}^*([\underline{E},0])\right)
=\widehat{f_{Y^{S^1}}}_!\left(\bigoplus_{\alpha}
[\underline{\lambda_{-1}(N_{\alpha}^*)_{\mN}^{-1}},0]
\cup\,\hat{\iota}^*([\underline{E},0])\right)
\\
=\sum_{\alpha}\bar{\eta}_g\left(\underline{TY_{\alpha}^{S^1}}, 
\underline{L_{\alpha}},
\underline{\lambda_{-1}(N_{\alpha}^*)^{-1}_{\mN}}\otimes 
\underline{E}|_{Y^{S^1}_{\alpha}}\right) \,\mathrm{mod}\, \Q_g.
\end{multline}
Thus by (\ref{eq:3.014}), (\ref{eq:3.062})-\eqref{eq:3.021}, 
for  $\mN>\mN_0$ and $g\in S^1\backslash A$, 
we have $Q_{\mN}(g)\in \Q_g$, i.e., (\ref{i13}) holds.
The proof of Theorem \ref{i08} is completed.
\end{proof}

Let $K\in \mathrm{Lie}(S^1)$ be fixed.

\begin{lemma}\label{lem:3.07} 
 For $g\in S^1\backslash A$, there exists 
$\beta>0$ such that for any $t\in \R$, $|t|\leq \beta$, 
$\mN'>\mN>\mN_0$,
$P_{\mN,\mN'}(ge^{tK})$ and $Q_{\mN}(ge^{tK})$ are 
real analytic in $t$.
\end{lemma}
\begin{proof}
Recall that $r_{\alpha,v}=\rank N_{\alpha,v}$. 
By \eqref{eq:3.008}, we have
\begin{align}\label{eq:3.022} 
F_{\alpha,\mN}(h)\cdot \underline{\lambda_{-1}
(N_{\alpha}^*)^{-1}_{\mN}}
=\underline{\mu_{\alpha,\mN,+}}-\underline{\mu_{\alpha,\mN,-}}
\, \, \,   \text{ with }  F_{\alpha,\mN}(h)=
\prod_{v: \, r_{\alpha,v}\neq 0 }	(h^v-1)^{r_{\alpha,v}+\mN}.
\end{align}
Set
\begin{align}\label{eq:3.024} 
F_{\mN}(x)= \prod_{v:  \max_{\alpha} r_{\alpha,v}\neq 0 } 
(x^v-1)^{\max_{\alpha}r_{\alpha,v}+\mN}\in \Z[x].
\end{align}
By Theorem \ref{local96} and \eqref{eq:1.52b}, for 
$g\in S^1\backslash A$, there exists $\beta>0$ such that for
$|t|<\beta$, we have
\begin{align}\label{eq:3.55a}
\bar{\eta}_{ge^{tK}}(\underline{TY},\underline{L},\underline{E})
=\bar{\eta}_{g,tK}(\underline{TY},\underline{L},\underline{E}).
\end{align}
From 
\eqref{eq:3.008}, \eqref{eq:3.013},
\eqref{eq:3.014} and \eqref{eq:3.55a}, 
for $g\in S^1\backslash A$, there exists $\beta>0$
such that for $|t|<\beta$, $\mN'>\mN>\mN_0$, we have
\begin{multline}\label{eq:3.025} 
F_{\mN'}(ge^{tK})\cdot P_{\mN,\mN'}(ge^{tK})
\\
=\sum_{\alpha}\bar{\eta}_{ge^{tK}}\left(\underline{TY_{\alpha}^{S^1}}, 
\underline{L_{\alpha}}, (\underline{\mu_{\alpha,\mN',+}}
-\underline{\mu_{\alpha,\mN',-}})\otimes 
\underline{E}|_{Y^{S^1}_{\alpha}}\right)\cdot\frac{F_{\mN'}(ge^{tK})}
{F_{\alpha,\mN'}(ge^{tK})}
\\
-\sum_{\alpha}\bar{\eta}_{ge^{tK}}\left(\underline{TY_{\alpha}^{S^1}}, 
\underline{L_{\alpha}},
(\underline{\mu_{\alpha,\mN,+}}-\underline{\mu_{\alpha,\mN,-}})
\otimes \underline{E}|_{Y^{S^1}_{\alpha}}\right)
\cdot\frac{F_{\mN'}(ge^{tK})}{F_{\alpha,\mN}(ge^{tK})},
\end{multline}
and
\begin{multline}\label{eq:3.026} 
F_{\mN}(ge^{tK})\cdot Q_{\mN}(ge^{tK})=F_{\mN}(ge^{tK})
\bar{\eta}_{g,tK}(\underline{TY}, \underline{L},
\underline{E})
\\
-\sum_{\alpha}\bar{\eta}_{ge^{tK}}\left(\underline{TY_{\alpha}^{S^1}}, 
\underline{L_{\alpha}},
(\underline{\mu_{\alpha,\mN,+}}-\underline{\mu_{\alpha,\mN,-}})
\otimes \underline{E}|_{Y^{S^1}_{\alpha}}\right)\cdot
\frac{F_{\mN}(ge^{tK})}{F_{\alpha,\mN}(ge^{tK})}.
\end{multline}

Recall that for $g\in S^1\backslash A$, $g^v-1\neq 0$ 
if $r_{\alpha,v}\neq 0$.  
So there exists $\beta>0$ such that for $|t|\leq \beta$, 
$F_{\mN}(ge^{tK})^{-1}$ is real analytic in $t$ for any $\mN$.
By (\ref{eq:3.012}), 
$\bar{\eta}_{ge^{tK}}(\underline{TY_{\alpha}^{S^1}},\cdots)$
in (\ref{eq:3.025}) and (\ref{eq:3.026}) are polynomials on 
$ge^{tK}$ and $(ge^{tK})^{-1}$.
Thus by Theorem \ref{local90} and (\ref{eq:3.012}), 
for $g\in S^1\backslash A$,
there exists $\beta>0$ such that for $|t|\leq \beta$, 
$\mN'>\mN>\mN_0$,
$P_{\mN,\mN'}(ge^{tK})$ and $Q_{\mN}(ge^{tK})$ 
are real analytic in $t$. 
\end{proof}

\begin{prop}\label{prop:3.08} 
For any $\mN'> \mN>\mN_0$,
the functions $P_{\mN,\mN'}$ and $Q_{\mN}$ on 
$S^1\backslash A$ are restrictions on 
$S^1\backslash A$ of
rational functions on $S^1$ with integral coefficients.
\end{prop}
\begin{proof}
We prove first this property for $Q_{\mN}$. 

For $g=e^{2i\pi t}\in S^1\backslash A$
and $\mN>\mN_0$, we have $Q_{\mN}(g)\in \Q_g$
by (\ref{i13}) and (\ref{eq:3.014}).
By (\ref{i05}), we 
can write
\begin{align}\label{eq:3.027}
Q_{\mN}(g)=\frac{\sum_{k=0}^{N(g)}a_k(g)g^k}{\sum_{k=0}^{M(g)}
b_k(g)g^k},
\end{align}
where $a_k(g), b_k(g)\in \Z$, $N(g), M(g)\in \N$, 
the polynomials $\sum_{k=0}^{N(g)}a_k(g)x^k$ and 
$\sum_{k=0}^{M(g)}b_k(g)x^k$ are relatively prime,
and $\sum_{k=0}^{M(g)}b_k(g)g^k\neq 0$.
	
Let $T_{M,N}=\{g\in S^1\backslash A: M(g)\leq M, N(g)\leq N \}$. Then 
\begin{align}\label{eq:3.028}
\cup_{M,N=1}^{\infty}T_{M,N}=S^1\backslash A. 
\end{align}	
Fix $g_0\in S^1\backslash A$. Let $U$ be a connected open 
neighbourhood of $g_0$ in $S^1\backslash A$ such 
that $Q_{\mN}$ is real analytic on $U$ by Lemma \ref{lem:3.07}.
We have $\cup_{M,N=1}^{\infty}(T_{M,N}\cap U)=U$.
Since $U$ is an 
uncountable set, there exist $M_0, N_0\in \N$ such that 
$T_{M_0,N_0}\cap 	U$ is an uncountable set. We define the map 
$\Psi: T_{M_0,N_0}\cap U\rightarrow \Z^{M_0+N_0+2}$ such that
\begin{align}\label{eq:3.029}
\Psi(g)=(a_{0}(g),\cdots,  a_{N_0}(g), b_0(g),\cdots, b_{M_0}(g)).
\end{align} 
Since $\Z^{M_0+N_0+2}$ is a countable set, there exists 
$I=(a_{0},\cdots,  a_{N_0}, b_0,\cdots, b_{M_0})
\in \mathrm{Im}(\Psi)$ such 
that $\Psi^{-1}(I)$ is a 	uncountable set. Set 
$h(x)=\frac{\sum_{k=0}^{N_0}a_kx^k}{\sum_{k=0}^{M_0}b_kx^k}$.
Then there is an open subset $U'\subset U$ such that $h$
is real analytic on $U'$ and 	$Q_{\mN}=h$ 
on a uncountable subset of $U'$. 
Moreover, since $h$ is a meromorphic function on $\C$, $Q_{\mN}$ 
can be extended as a holomorphic function 
on an open connected neighborhood $U_0\subset\C$ of $U$,
we have $h\equiv Q_{\mN}$ on $U_0$,
in particular, $h=Q_{\mN}$ on $U$.
So for any $g_0\in S^1\backslash A$, there is an open neighborhood 
$U$ of 	$g_0$ in 
$S^1\backslash A$ such that $Q_{\mN}$ is a rational function on $U$
with 	integral coefficients.
It means that $Q_{\mN}$ is a rational function on each connected 
component of $S^1\backslash A$ 	with integral coefficients. 

For $g\in A$ and for small $t\neq 0$ it follows from Theorem \ref{local96},
similarly to \eqref{eq:3.026}, that
\begin{multline}\label{eq:3.030}
Q_{\mN}(ge^{tK})=\bar{\eta}_{g,tK}(\underline{TY}, \underline{L}, 
\underline{E})-\mM_{g,tK}\left(\underline{TY}, \underline{L}, 
\underline{E}\right)
\\
-F_{\mN}(ge^{tK})^{-1}\cdot\sum_{\alpha}\bar{\eta}_{ge^{tK}}
\left(\underline{TY_{\alpha}^{S^1}}, \underline{L_{\alpha}},
(\underline{\mu_{\alpha,\mN,+}}-\underline{\mu_{\alpha,\mN,-}})
\otimes \underline{E}|_{Y^{S^1}_{\alpha}}\right)\cdot
\frac{F_{\mN}(ge^{tK})}{F_{\alpha,\mN}(ge^{tK})}.
\end{multline}
From Theorem \ref{local90}, \eqref{eq:3.012}, \eqref{eq:3.024} 
and \eqref{eq:3.030}, $Q_{\mN}(ge^{tK})$ is a meromorphic 
function in $t$ near $0$. But from the argument before \eqref{eq:3.030}, 
we know that for $t>0$ small
\begin{align}\label{eq:3.031} 
	Q_{\mN}(ge^{tK})=\frac{P_{+}(ge^{tK})}{Q_{+}(ge^{tK})}
\end{align}
is a rational function 
in $ge^{tK}$. As 
$\frac{P_{+}(ge^{tK})}{Q_{+}(ge^{tK})}$ is a meromorphic function
in $t$ near $0$, this implies 
\eqref{eq:3.031} holds for $t$ near $0$.
In particular, \eqref{eq:3.031} holds for $t<0$ small.
So $Q_{\mN}$ as a function on 
$S^1\backslash A$ is the restriction on $S^1\backslash A$
of a rational function on $S^1$ with integral coefficients.

By the argument after (\ref{eq:3.026}),
in particular 
by (\ref{eq:3.009}) and (\ref{eq:3.012}), we get that for
$\mN'>\mN>\mN_0$, $P_{\mN,\mN'}$ is the restriction on 
$S^1\backslash A$ of a rational function on $S^1$
with coefficients in $\R$. 
To show that that the coefficients are actually in $\Z$
we only need to apply the above argument again. 
	
The proof of Proposition \ref{prop:3.08}  is completed.
\end{proof}

By Lemma \ref{lem:3.07} and Proposition \ref{prop:3.08}, 
the proof of Theorem \ref{thm:3.03} is completed.

\subsection{The case when $Y^{S^1}=\emptyset$}\label{s0405}

In the remainder of this paper, we discuss the case when 
$Y^{S^1}=\emptyset$.

As $Y^{S^1}=\emptyset$, from Proposition \ref{local32},
$A=\{g\in S^1 : Y^g\neq\emptyset \}$ is a finite set. 
 From the variation formula 
\eqref{local73}, for 
$g\in S^1\backslash A$, up to 
$\Q_g$, $\bar{\eta}_g(\underline{TY}, \underline{L},
\underline{E})$ does not depend on the 
geometric data $g^{TY}$, $(h^{L}, \nabla^{L})$, $(h^E$, $\nabla^E$). 
Thus similarly as in Definition \ref{defn:3.01},  the 
map $f_!: K_{S^1}^0(Y)_{I(g)}\rightarrow 
\C/\Q_g$  for $g\in S^1\backslash A$, defined by 
\begin{align}\label{eq:3.063}
f_!([E]/\chi)=\chi
(g)^{-1}\bar{\eta}_g(\underline{TY}, \underline{L},
\underline{E})
\, \, \,\mathrm{mod}\,{\Q_g}
\end{align}
is well-defined. By \cite[Proposition 1.5]{ASegal68}, 
$K_{S^1}^0(Y)_{I(g)}=0$ for $g\in S^1\backslash A$. 
So $\bar{\eta}_g(\underline{TY}, \underline{L},
\underline{E})\in \Q_g$. Since 
Theorems \ref{local85} and \ref{local96} still hold for 
$Y^{S^1}=\emptyset$, following the same process as in  Lemma 
\ref{lem:3.07} and Proposition \ref{prop:3.08} (note that
for the last part of the proof of Proposition \ref{prop:3.08},
we only use Theorem \ref{local96}), we obtain: 
\begin{thm}\label{thm:3.13}
If $Y^{S^1}=\emptyset$, $A=\{g\in S^1 : Y^g\neq\emptyset \}$,
then $\bar{\eta}_g(\underline{TY}, \underline{L},
\underline{E})$ as a function on $S^1\backslash A$ 
is the restriction of a rational function on $S^1$ with
integral coefficients
 that does not have poles
on $S^1\backslash A$.
\end{thm}

\begin{eg} For $k\in \N^{*}$, we consider the circle action on 
	$Y=\widehat{S^1}$ with 
	\begin{align}\label{eq:3.064} 
	g.e^{i \theta}=e^{2\pi ikt+i\theta},\quad 
\text{ for }	g=e^{2\pi i t}\in S^1.
	\end{align}
	Here $\widehat{S^1}$ is a copy of $S^1$.
	For $x=e^{i \theta}\in \widehat{S^1}$, if $g.x=x$, we have 
	$kt\in \Z$, which means that $g^k=1$. So $Y^{S^1}=\emptyset$ 
	and $A=\{g\in S^1 : g^k=1\}$.

We identify $[0,2\pi)$ with $\widehat{S^1}$ by sending 
$\theta$ to $e^{i\theta}$. Then the canonical metric on $\widehat{S^1}$ 
is defined by $|\frac{\partial}{\partial \theta}|=1$, the spinor of
 $\widehat{S^1}$ is $\mS(\widehat{S^1})=\C$, and the Clifford action is 
 defined by $c\left(\frac{\partial}{\partial \theta}\right) =-i\in
 \End(\mS(\widehat{S^1}))$. Thus the untwisted Dirac operator 
 on $\widehat{S^1}$ is 
	\begin{align}\label{eq:3.065} 
	D=-i \frac{\partial}{\partial \theta}.
	\end{align}	
From \eqref{eq:3.064} and \eqref{eq:3.065}, we 
see that the circle action commutes with $D$.
	From \eqref{eq:3.065}, the eigenvalues of $D$ 
	are $\lambda_n=n$, $n\in \Z$, with eigenspaces
	$E_{\lambda_n}=\C\{e^{in \theta}\}$. 
	For $g=e^{2\pi i t}\in S^1$, $s\in \C$ 
	and $\Re (s)>1$, we see that
	\begin{align}\label{eq:3.066} 
	\eta_g(s):=\sum_{n=1}^{+\infty}
	\frac{\tr|_{E_{\lambda_n}}[g]}{n^s}-\sum_{n=1}^{+\infty}
	\frac{\tr|_{E_{\lambda_{-n}}}[g]}{n^s}
	=\sum_{n=1}^{+\infty}
	\frac{e^{2\pi i n kt}}{n^s}-\sum_{n=1}^{+\infty}
	\frac{e^{-2\pi i n kt}}{n^s}
	\end{align}
	is well-defined.

For $x,y\in \R$, $s\in \C$, let $S_1(x,y,s)$ be the 
Kronecker zeta function \cite[p53]{We76},
\begin{align}\label{eq:3.067} 
S_1(x,y,s)=\sum_{n\in\Z}{'}(x+n)|x+n|^{-2s}e^{-2\pi iny},
\end{align}
where $\sum_{n\in \Z}'$ is a sum over $n\in \Z$, $n\neq -x$. 
The series in \eqref{eq:3.067} converges absolutely 
for $\Re(s)>1$, and defines a holomorphic function of $s$. 
Moreover, $s\mapsto S_1(x,y,s)$ has a 
holomorphic continuation to $\C$ \cite[p57]{We76}. 
By \eqref{eq:3.066} and \eqref{eq:3.067}, we have
\begin{align}\label{eq:3.068} 
	\eta_g(s)=-S_1\left(0,kt,\frac{s+1}{2}\right).
\end{align}
Thus $\eta_g(s)$ has a holomorphic continuation to $\C$.
Also by \cite[p57]{We76},
we have the functional equation for $S_1(x,y,s)$,
\begin{align}\label{eq:3.069} 
\Gamma(s)S_1(x,y,s)=-i\pi^{2s-3/2}e^{2\pi i xy}
\Gamma\left(\frac{3}{2}-s \right)S_1\left(y,-x, \frac{3}{2}-s\right).
\end{align}
From \eqref{eq:3.068} and \eqref{eq:3.069}, we have
\begin{align}\label{eq:3.070} 
\eta_g(0)=\frac{i}{\pi}S_1(kt,0,1).
\end{align}
By \cite[p57]{We76}, for $x\notin \Z$, we have
\begin{align}\label{eq:3.071} 
S_1(x,0,1)=\pi \cot(\pi x).
\end{align}
If $g\in S^1\backslash A$, then $kt\notin\Z$.
Thus by \eqref{eq:3.064}, \eqref{eq:3.070} and \eqref{eq:3.071}, 
for $g\in S^1\backslash A$,  the equivairant eta invariant 
$\eta_g(\underline{\widehat{S^1}})$
of $S^1$ with the canonical metric, is given by 
	\begin{align}\label{eq:3.072} 
	\eta_g(\underline{\widehat{S^1}})=\eta_g(0)
	=i\cot(\pi kt)=-\frac{g^{k/2}+g^{-k/2}}{g^{k/2}-g^{-k/2}}
	=\frac{2}{1-g^k}-1.
	\end{align}
Since $\Ker (D)$ is the space of the complex valued constant functions 
on $\widehat{S^1}$, we have $\tr|_{\Ker (D)}[g]=1$. 
Thus from \eqref{eq:3.072}, for $g\in S^1\backslash A$,
the equivairant reduced eta invariant 
	\begin{align}\label{eq:3.073} 
	\bar{\eta}_g(\underline{\widehat{S^1}})
	=\frac{\eta_g(\underline{\widehat{S^1}})
		+\tr|_{\Ker (D)}[g]}{2}=\frac{1}{1-g^k}.
	\end{align}
It is a rational function on $S^1$ with integral 
coefficients and poles in $A$.

For $g\in A$, then from \eqref{eq:3.066}, we get $\eta_g(s)=0$, for any 
$s\in \C$. Since $\tr|_{\Ker (D)}[g]=1$,
\begin{align}\label{eq:3.074} 
\bar{\eta}_g(\underline{\widehat{S^1}})
=\bar{\eta}_1(\underline{\widehat{S^1}})
=\bar{\eta}(\underline{\widehat{S^1}})
=\frac{1}{2},\quad \text{for}\ g\in A.
\end{align} 
	
Remark that this reduced eta invariant could also be computed from the 
equivariant APS index theorem \eqref{eq:0.22}. Let $B$ be the unit disc 
in $\R^2$ and $\partial B=\widehat{S^1}$. 
In polar coordinates, the circle action on $B$ is defined by 
\begin{align}\label{eq:3.075} 
g.re^{i \theta}=re^{2\pi i kt+i\theta}, \quad
\text{for}\ g=e^{2\pi i t},
\,  k\in \N^{*}.
\end{align}
It induces the circle action on $\partial B=\widehat{S^1}$
in \eqref{eq:3.064}. If $g\in S^1\backslash A$, then $g^k\neq 1$
and the only fixed point set of  $g$ on $B$ is $B^g=\{0\}$. 
Let $N$ be the normal bundle of $\{0\}$ in $B$. 
Then the $g$-action on $N$ is as multiplication by $g^k$.
	
We take the metric on $B$ such that it 
has product structure
near the boundary, induces the canonical metric on $\widehat{S^1}$
and commutes with the circle action. 
We denote by $D^B_{+,\geq 0}$ the Fredholm 
operator with respect to the Dirac operator $D_+^B$ on $B$ and the 
APS boundary condition.
The equivariant APS index is defined by
	\begin{align}\label{eq:3.076} 
	\ind_{\mathrm{APS},g}(D_+^B)
	=\tr|_{\Ker \left(D^B_{+,\geq 0}\right)}[g]
	-\tr|_{\mathrm{Coker}\left(D^B_{+,\geq 0}\right)}[g].
	\end{align}
From the equivariant APS index theorem \cite{Donnelly78}, we have for 
$g\in S^1\backslash A$,
\begin{align}\label{eq:3.077} 
\bar{\eta}_g(\underline{\widehat{S^1}})
=\int_{B^g}\frac{1}{\det(1-g|_N)}-\ind_{\mathrm{APS},g}(D_+^B)
=\frac{1}{1-g^k}-\ind_{\mathrm{APS},g}(D_+^B).
\end{align}
	
Note that the equivariant APS index is invariant when the metric
varies near the boundary, without changing the metrics on the boundary.
We only need to compute $\ind_{\mathrm{APS},g}(D_+^B)$ 
using the canonical metric on $B$, which 
is not of product structure near the boundary.	In this case,
with  the coordinate $z=x+i y$ in $\R^2\simeq \C$,
the spinor of $B$ 
is $\mS(B)=\C \oplus \C (d\overline{z}/\sqrt{2})   
 \simeq \C\oplus \C$ and the Clifford 
action is given by
\begin{align}\label{eq:3.078}
c\left(\frac{\partial}{\partial x}\right):=\left(
\begin{array}{cc}
0 & -1 \\
1 & 0 \\
\end{array}
\right),\quad 
c\left(\frac{\partial}{\partial y} \right):=\left(
\begin{array}{cc}
0 & i \\
i & 0 \\
\end{array}
\right).
\end{align} 
Thus $D^B=c(\frac{\partial}{\partial x}) \frac{\partial}{\partial x}
+c(\frac{\partial}{\partial y})\frac{\partial}{\partial y}$ 
has the form
\begin{align}\label{eq:3.079}
D^B=\left(
\begin{array}{cc}
0 & D_-^B \\
D_+^B & 0 \\
\end{array}
\right)=\left(
\begin{array}{cc}
0 & -\frac{\partial}{\partial x}+i\frac{\partial}{\partial y} \\
\frac{\partial}{\partial x}+i\frac{\partial}{\partial y} & 0 \\
\end{array}
\right)=
\left(
\begin{array}{cc}
0 & -2\frac{\partial}{\partial z} \\
2\frac{\partial}{\partial \bar{z}} & 0 \\
\end{array}
\right),
\end{align} 
In polar coordinates,
	\begin{align}\label{eq:3.080} 
	D_+^B=-e^{i \theta}\left(-\frac{\partial}{\partial r}
	-\frac{i}{ r}\frac{\partial}{\partial \theta}\right).
	\end{align}
Note that $-\frac{\partial}{\partial r}$ here is the inward normal vector.
Let $P_{\geq0}$ be the orthogonal projection onto the 
direct sum of the eigenspaces associated with the nonnegative
eigenvalues of 
$\mathcal{A}:=-\frac{i}{ r}\frac{\partial}{\partial \theta}
|_{\partial B}$.
Recall that the eigenvalues of $\mA$ are $\lambda_n=n$, 
$n\in \Z$, with eigenspaces $\C\{e^{in \theta} \}$. 
Thus the APS boundary condition reads in complex coordinates
for $f\in \cC^{\infty}(B,\C)$,
\begin{align}\label{eq:3.081} 
P_{\geq 0}f|_{\partial B}=0 
\Leftrightarrow f|_{\partial B}=\sum_{n<0, n\in\Z}a_n z^{n}.
\end{align}
If $D_+^Bf=0$, $f$ is holomorphic on $B$. Thus 
	\begin{align}\label{eq:3.082} 
	\Ker (D^B_{+,\geq 0})=\{f\in \cC^{\infty}(B,\C) : D^B_+f=0, 
	P_{\geq 0}(f|_{\partial B})=0 \}=0. 
	\end{align}

By \eqref{eq:3.080}, the adjoint of $D^B_{+,\geq 0}$ is 
$D^B_{-,< 0}$, i.e., $D_-^B$ with boundary condition
\begin{align}\label{eq:3.083} 
P_{<0}(e^{-i\theta}f)|_{\partial B}=0, \,
\text{  with } \, P_{\geq 0}+P_{<0}=\Id
\, \text{  on }\, \cC^{\infty}(\partial B,\C).
\end{align}
 By \eqref{eq:3.079} and \eqref{eq:3.083}, if $D^B_{-,< 0}f=0$, 
 then $\bar{f}$ is holomorphic on $B$ and 
$f|_{\partial B}=\sum_{ n>0,n\in\Z}a_nz^n$. Thus 
$\mathrm{Coker}(D^B_{+,\geq 0})=\Ker (D^B_{-,< 0})=0$.

From \eqref{eq:3.076},  \eqref{eq:3.077} and \eqref{eq:3.082}, we have 
$\bar{\eta}_g(\underline{\widehat{S^1}})=(1-g^k)^{-1}$ for 
$g\in S^1\backslash A$, 
which is the same as \eqref{eq:3.073}.
\end{eg}

\end{document}